\documentclass[11pt]{article}

\usepackage{amsfonts}
\usepackage{amssymb}
\usepackage{amsmath}
\usepackage{mathptmx}
\usepackage{graphicx}
\usepackage{setspace}
\usepackage{amsthm}
\usepackage[labelfont=bf]{caption}
\usepackage[left=2.5cm,top=2.15cm,right=2.5cm,bottom=2.15cm]{geometry}
\theoremstyle{plain}
\newtheorem{thm}{Theorem}[section] 
\newtheorem{defn}[thm]{Definition} 
\newtheorem{exmp}[thm]{Example} 
\newtheorem{lem}[thm]{Lemma} 
\newtheorem{pro}[thm]{Proposition} 
\newtheorem{rem}[thm]{Remark} 
\newtheorem{co}[thm]{Corollary}

\newtheorem{open}[thm]{Open Question}
\newtheorem{conj}[thm]{Conjecture}
\makeatletter
\newcommand{\vast}{\bBigg@{3.2}}
\newcommand{\Vast}{\bBigg@{5}}
\makeatother
\begin{document}
\begin{center}
\section*{Spectral representations of quasi-infinitely divisible processes}
\subsection*{Riccardo Passeggeri\footnote{Department of Mathematics, Imperial College London. Email: riccardo.passeggeri14@imperial.ac.uk}}
\today
\end{center}
\tableofcontents
\begin{abstract}
	In this work we first introduce quasi-infinitely divisible (QID) random measures and formulate spectral representations. Then, we introduce QID stochastic integrals and present integrability conditions and continuity properties. Further, we introduce QID stochastic processes, \textit{i.e.~}stochastic processes with QID finite dimensional distributions. For example, a process $X$ is QID if there exist two ID processes $Y$ and $Z$ such that $X+Y\stackrel{d}{=}Z$ with $Y$ independent of $X$. The class of QID processes is strictly larger than the class of ID processes. We provide spectral representations and L\'{e}vy-Khintchine formulations for potentially all QID processes. Finally, we prove that QID random measures are dense in the space of random measures under convergence in distribution. Throughout this work we present many examples.
\end{abstract}
\textbf{Key words:} quasi-infinitely divisible distributions, random measure, stochastic integral, L\'{e}vy-Khintchine formulation, infinitely divisible.
\section{Introduction}
Infinitely divisible (ID) distributions represents one of the main class of probability distributions. Their development goes back to the work of L\'{e}vy and De Finetti. Concerning ID processes, one of the most pivotal work in the field is given by the Rajput and Rosinski paper in 1989 \cite{RajRos}. This work provides extremely useful results on the spectral representation of discrete and centred continuous ID process. 
\\Recently, a series of works shed new light on a broader class of distributions called quasi-ID (QID) distributions, where the L\'{e}vy measure is now a signed measure (thus taking negative values). The recent work of Lindner, Pan and Sato \cite{LPS} represents one of the most important papers of this series. It shows for example that the set of QID distributions is dense in the set of all probability distributions with respect to weak convergence. Moreover, they proved that a distribution concentrated on the integers is QID if and only if its characteristic function does not have zeroes. This last result is extended in \cite{Berger} to distributions which can be written as the sum of a distribution concentrated on the integers and a distribution which is absolutely continuous w.r.t.~the Lebesgue measure. An interesting result shown in \cite{Berger} states that a distribution which has a L\'{e}vy measure with complex values cannot exist. 
\\ QID distributions have been shown to have links to the field of prime numbers as well. Indeed, in \cite{Naka} it is shown that $\frac{\xi(\sigma-it)}{\xi(\sigma)}$, where $\xi$ is the complete Riemann zeta function and $t,\sigma\in\mathbb{R}$, has a characteristic function which is QID, but not ID, for $\sigma>1$. In addition, in \cite{Naka2} it is shown that any zeta distribution defined by Dirichlet series
$\sum_{n=1}^\infty a(n) n^{-s}$ with $a(1) >0$, $a(n) \ge 0$ when
$n \ge 2$ and $a(n) = O(n^\varepsilon)$ for any $\varepsilon >0$
is quasi infinitely divisible if $\Re (s) >1$ is sufficiently large. Further, QID distributions turn out to be of importance in mathematical physics too, as seen in \cite{Physics2} and \cite{Physics1}.

Therefore, it appears natural to see to what extent the results in \cite{RajRos} extend to the QID case. This is part of the content of the present work. Indeed, this work extends the results of the celebrated 1989 paper by Rajput and Rosinski to the QID framework. However, we also investigate general questions like: What is the L\'{e}vy-Khintchine representation of QID processes? What is the L\'{e}vy-Khintchine representation when the QID process belongs to $l_{2}$ (the Hilbert space of square summable sequences)? What is the connection between QID processes and L\'{e}vy processes? Are QID random measures dense in the space of random measures?
\\ This work is structured as follows. Section \ref{Notation} introduces the notation and preliminaries, including the introduction of QID random measures. Section \ref{chapter-IDvQID} shows the connections between QID and ID random measures. In Section \ref{Ch-QIDrm} we present general results on QID random measures and show explicit cases in which they arise. In Section \ref{Ch-QIDstochint} QID stochastic integrals are defined. We extend the measure theoretical results at the heart of the Rajput and Rosinski's 1989 paper (\cite{RajRos}) to the signed measure framework, and provide a L\'{e}vy-Khintchine representation and integrability conditions for QID stochastic integrals. In Section \ref{Ch-Continuity} we show a continuity property of the QID stochastic integral. In Section \ref{Ch-Spectr} quasi-L\'{e}vy measures on $l_{2}$ and QID processes are defined. We prove that $X$ is QID if there exist two ID processes $Y$ and $Z$ such that $X+Y\stackrel{d}{=}Z$ with $Y$ independent of $X$. In this case we say that $X$ is \textit{generated} by $Y$ and $Z$, and the class of such QID processes is called \textit{generated QID processes}. We provide a L\'{e}vy-Khintchine representation for generated QID processes on $l_{2}$ and for generated QID processes of any arbitrary index set, and present a spectral representation for discrete parameters generated QID processes \textit{\`{a} la} Rajput and Rosinski. In the last subsection of Section \ref{Ch-Spectr}, we show further results on general QID processes and provide some examples. Finally, in Section \ref{Sec-Atomless}, we discuss the atomless condition of random measures in the QID framework and then show that QID random measures are dense in the space of random measures under convergence in distribution.
\section{Notation and Preliminaries}\label{Notation}
In this section we introduce the notation and the preliminaries needed in this work. \\By a measure on a measurable space $(X,\mathcal{G})$ we always mean a positive measure on $(X,\mathcal{G})$, namely an $[0,\infty]$-valued $\sigma$-additive set function on $\mathcal{G}$ that assigns the value 0 to the empty set. Given a non-empty set $X$, the symbol $\mathcal{B}(X)$ stands for the Borel $\sigma$-algebra of $X$, unless stated differently. The law and the characteristic function of a random variable $X$ will be denoted by $\mathcal{L}(X)$ and by $\hat{\mathcal{L}}(X)$, respectively. We will use term \textit{measure} for a positive measure and the term \textit{signed measure} for a signed measure. Finally, due to their frequent use we abbreviate the following words: random variable by r.v., random measure by r.m., characteristic function by c.f.~and characteristic triplet by c.t..
\\Given the importance of signed measures in this work we recall now the definition and some properties.
\begin{defn}[signed measure]\label{Def-signedmeasure}
Given a measurable space $(X, \Sigma)$, that is, a set $X$ with a $\sigma$-algebra $\Sigma$ on it, an extended signed measure is a function $\ \mu :\Sigma \to {\mathbb {R}}\cup \{\infty ,-\infty \}$ s.t.~$\mu (\emptyset )=0$ and $\mu$ is sigma additive, that is, it satisfies the equality $ \mu \left(\bigcup _{{n=1}}^{\infty }A_{n}\right)=\sum _{{n=1}}^{\infty }\mu (A_{n})$ where the series on the right must converge in ${\mathbb {R}}\cup \{\infty ,-\infty \}$ absolutely (namely the value of the series is independent of the order of its elements), for any sequence $A_{1}, A_{2},...$ of disjoint sets in $\Sigma$.
\end{defn}
\noindent As a consequence any extended signed measure can take plus or minus infinity as value but not both. Recall also that the \textit{total variation} of a signed measure $\mu$ is defined as the measure $|\mu|:\Sigma\rightarrow [0, \infty]$ defined by
\begin{equation}\label{def-totalvariation}
|\mu|(A):=\sup\sum_{j=1}^{\infty}|\mu(A_{j})|
\end{equation}
where the supremum is taken over all the partitions $\{A_{j}\}$ of $A\in\Sigma$. The total variation $|\mu|$ is finite if and only if $\mu$ is finite. By the Hahn decomposition theorem, for any signed measure $\mu$, there exist disjoint Borel sets $C^{+}$ and $C^{-}$ with $C^{+}\cup C^{-}=X$ and $C^{+}\cap C^{-}=\emptyset$ s.t.~$\mu(A\cap C^{+})\geq 0$ and $\mu(A\cap C^{-})\leq 0$ for every $A\in\Sigma$. In the following, we present the definition of mutually singular measures and the Jordan decomposition theorem.
\begin{defn}
	Two measures $\mu$ and $\nu$ on $(X, \Sigma)$ are said to be \textnormal{mutually singular} if there are disjoint sets $A, B \in \Sigma$ with $X = A \cup B$ and $\mu(A) = 0$ while $\nu(B) = 0$. In this case, we write	$\mu\perp \nu$.
\end{defn}
\begin{thm}[Jordan Decomposition Theorem]\label{Jordan}
	Let $\mu$ be a signed measure on $(X, \Sigma)$. Then there exist two mutually singular positive measures $\mu^{+}$ and $\mu^{-}$ such that $\mu= \mu^{+} - \mu^{-}$. Furthermore, if $\lambda$ and $\nu$ are any two positive measures with $\mu= \lambda -\nu$, then for each $E \in \Sigma$ we have $\lambda(E) \geq \mu^{+}(E)$ and $\nu(E) \geq \mu^{-}(E)$. Finally, if $\lambda\perp\nu$, then $\lambda=\mu^{+}$ and $\nu=\mu^{-}$.
\end{thm}
\noindent Recall that while the Jordan decomposition is unique the Hahn decomposition is only essential unique, indeed $\mu$-null sets can be transferred from $C^{+}$ to $C^{-}$ and vice versa.
\begin{defn}[Signed bimeasure]
	Let $(X,\Sigma)$ and $(Y,\Gamma)$ be two measurable spaces, $\Sigma\times\Gamma$ the Cartesian product of $\Sigma$ and $\Gamma$ (to be distinguished from $\Sigma\otimes\Gamma$, which stands for the product $\sigma$-algebra of $\Sigma$ and $\Gamma$). A \textnormal{signed bimeasure} is a function $M:\Sigma\times\Gamma\rightarrow[-\infty,\infty]$ such that:
	\\\textnormal{(i)} the function $A\rightarrow M(A,B)$ is a signed measure on $\Sigma$ for every $B\in\Gamma$,
	\\\textnormal{(i)} the function $B\rightarrow M(A,B)$ is a signed measure on $\Gamma$ for every $A\in\Sigma$.
	
\end{defn}
\noindent For a signed bimeasure $M$, we denote by $M^{+}$ and $M^{-}$ the Jordan decomposition of $M(A,B)$ for fixed $A\in\Sigma$, and $M_{+}$ and $M_{-}$ the Jordan decomposition of $M(A,B)$ for fixed $B\in\Gamma$.

Now, we introduce the concept of a quasi-L\'{e}vy type measure. Although it is called measure it is not always a measure. This explains the need of the following definitions, which we recall from \cite{LPS}:
\begin{defn}\label{def1}
Let $\mathcal{B}_{r}(\mathbb{R}):=\{B\in\mathcal{B}(\mathbb{R})| B \cap(−r, r) = \emptyset\}$ for $r > 0$ and $\mathcal{B}_{0}(\mathbb{R}):= \bigcup_{r>0} \mathcal{B}_{r}(\mathbb{R})$ be the class of all
Borel sets that are bounded away from zero. Let $\nu : \mathcal{B}_{0}(\mathbb{R})\rightarrow\mathbb{R}$ be a function such that
$\nu_{|\mathcal{B}_{r}(\mathbb{R})}$ is a finite signed measure for each $r > 0$ and denote the total variation, positive and negative part of $\nu_{|\mathcal{B}_{r}(\mathbb{R})}$ by $|\nu_{|\mathcal{B}_{r}(\mathbb{R})}|$, $\nu^{+}_{|\mathcal{B}_{r}(\mathbb{R})}$ and $\nu^{-}_{|\mathcal{B}_{r}(\mathbb{R})}$ respectively. Then the \textnormal{total variation} $|\nu|$, the \textnormal{positive part} $\nu^{+}$ and the \textnormal{negative part} $\nu^{-}$ of $\nu$ are defined to be the unique measures on $(\mathbb{R},\mathcal{B}(\mathbb{R}))$ satisfying
\begin{equation*}
|\nu|(\{0\})=\nu^{+}(\{0\})=\nu^{-}(\{0\})=0
\end{equation*}
\begin{equation*}
\text{and}\quad|\nu|(A)=|\nu_{|\mathcal{B}_{r}(\mathbb{R})}|,\,\,\nu^{+}(A)=\nu_{|\mathcal{B}_{r}(\mathbb{R})}^{+}(A),\,\,\nu^{-}(A)=\nu_{|\mathcal{B}_{r}(\mathbb{R})}^{-}(A),
\end{equation*}
for $A\in\mathcal{B}_{r}(\mathbb{R})$, for some $r>0$.
\end{defn}
 As mentioned in \cite{LPS}, $\nu$ is not a a signed measure because it is defined on $\mathcal{B}_{0}(\mathbb{R})$, which is not a $\sigma$-algebra. In the case it is possible to extend the definition of $\nu$ to $\mathcal{B}(\mathbb{R})$ such that $\nu$ will be a signed measure then we will identify $\nu$ with its extension to $\mathcal{B}(\mathbb{R})$ and speak of $\nu$ as a signed measure. Moreover, the uniqueness of $|\nu|$, $\nu^{+}$ and $\nu^{-}$ is ensured by the Carath\'{e}odory's extension theorem.
\begin{rem}
	For the sake of clarity, notice that $\mathcal{B}_{0}(\mathbb{R})= \{B\in\mathcal{B}(\mathbb{R}):0\notin \overline{B} \}\neq \{B\in\mathcal{B}(\mathbb{R}):0\notin B \}$. Indeed, consider the set $A=\{\frac{1}{n}:n\in\mathbb{N}\}$ then $A$ is a Borel set since it is a countable intersection of closed Borel sets. Notice that $A\in \{B\in\mathcal{B}(\mathbb{R}):0\notin B \}$. Moreover, observe that for every $\epsilon>0$ we can find an element of $A$ that do not belong to $\mathcal{B}_{\epsilon}$. Hence, $A\notin\mathcal{B}_{0}(\mathbb{R})$.
\end{rem}
Throughout this work we define the centering function $\tau$ (on a general Hilbert space) as
\begin{equation*}
\tau(x):=\begin{cases}
x\,\,&\textnormal{if }\,\,\,\|x\|\leq 1,\\ \frac{x}{\|x\|} \,\,&\textnormal{if }\,\,\,\|x\|> 1.
\end{cases}
\end{equation*}
where $\|\cdot\|$ is the norm of the Hilbert space considered. Notice that this centering function satisfies equation (1) in \cite{LPS}. The following definition has been introduced for the first time in \cite{LPS}.
\begin{defn}[quasi-L\'{e}vy type measure, quasi-L\'{e}vy measure, QID distribution]
	A quasi-L\'{e}vy type measure is a function $\nu: \mathcal{B}_{0}(\mathbb{R})\rightarrow \mathbb{R}$ satisfying the
	condition in Definition \ref{def1} and such that its total variation $|\nu|$ satisfies $\int_{\mathbb{R}} (1\wedge x^{2} ) |\nu|(dx) <\infty$.
	\\ Let $\mu$ be a probability distribution on $\mathbb{R}$. We say that $\mu$ is quasi-infinitely divisible if its characteristic function has a representation
	\begin{equation*}
\hat{\mu}(\theta)=\exp\left(i\theta \gamma-\frac{\theta^{2}}{2}a+\int_{\mathbb{R}}e^{i\theta x}-1-i\theta\tau(x)\nu(dx)\right)
	\end{equation*}
	where $a, \gamma \in \mathbb{R}$ and $\nu$ is a quasi-L\'{e}vy type measure. The characteristic triplet $(a, \nu,\gamma)$
	of $\mu$ is unique (see [\cite{Sato}, Exercise 12.2]), and $a$ is called the \textnormal{Gaussian variance} of $\mu$.
	\\A quasi-L\'{e}vy type measure $\nu$ is called \textnormal{quasi-L\'{e}vy measure}, if additionally there exist a quasi-infinitely divisible distribution $\mu$ and some $a,\gamma\in\mathbb{R}$	such that $(
	a, \nu,\gamma)$ is the characteristic triplet of $\mu$. We call $\nu$
	the	quasi-L\'{e}vy measure of $\mu$.
\end{defn}
The above definition extend to the $\mathbb{R}^{d}$ case (for $d>1$) as shown in Remark 2.4 in \cite{LPS}.\\
A quasi-L\'{e}vy measure is always a quasi-L\'{e}vy type measure, while the converse is not true as pointed out in Example 2.9 of \cite{LPS}. Moreover, we say that a function $f$ is \textit{integrable with respect to quasi-L\'{e}vy type measure} $\nu$ if it is integrable with respect to $|\nu|$. Then, we define:
\begin{equation*}
\int_{B}fd\nu:=\int_{B}fd\nu^{+}-\int_{B}fd\nu^{-},\quad B\in\mathcal{B}(\mathbb{R}).
\end{equation*}
Given the above discussions, it appears clear why it is sometimes useful to work with the characteristic pair $(\zeta,\gamma)$ instead of $(
a, \nu,\gamma)$ where $\zeta$ is a signed measure defined as:
\begin{equation*}
\zeta(B)=a\delta_{0}(B)+\int_{B}(1\wedge x^{2})\nu(dx),\quad B\in\mathcal{B}(\mathbb{R})
\end{equation*}
Then, the representation of c.f.~of $\mu$ becomes:
	\begin{equation*}
	\hat{\mu}(\theta)=\exp\left(i\theta \gamma+\int_{\mathbb{R}}g_{\tau}(x,\theta)\zeta(dx)\right)
	\end{equation*}
	where $g_{\tau}:\mathbb{R}\times\mathbb{R}\rightarrow\mathbb{C}$ defined by
	\begin{equation*}
g_{\tau}(x,\theta)=\begin{cases}
(e^{i\theta x}-1-i\theta\tau(x))/(1\wedge x^{2}),\quad x\neq 0,\\
-\frac{\theta^{2}}{2}, \quad x=0.
\end{cases}
	\end{equation*}
	The function $g_{\tau}(\cdot,\theta)$ is bounded for each fixed $\theta\in\mathbb{R}$ and is continuous at zero (see \cite{LPS} for further details). 
	\\ Let us now mention one of the few known explicit results on QID distributions.
	\begin{thm}[Theorem 4.3.4 in \cite{Cuppens}] Let $d\in\mathbb{N}$. The c.t.~$(\gamma,0,\nu)$, where $\nu$ is a finite quasi-L\'{e}vy type measure, is the c.t.~of a QID distribution on $\mathbb{R}^{d}$ if and only if $\exp(\nu):=\sum_{n=1}^{\infty}\frac{\nu^{*n}}{n!}$ is a measure. In that case, $\mu\sim(\gamma,0,\nu)$ is given by
\begin{equation*}
\mu=\frac{\delta_{\gamma}*\exp(\nu)}{\exp(\nu(\mathbb{R}^{d}))}.
\end{equation*}	
	\end{thm}
We introduce next the framework needed to work with ID (and QID) processes.
\\Throughout the paper, we denote by $S$ an arbitrary non-empty set and by $\mathcal{S}$ a $\delta$-ring with the additional condition that there exists an increasing sequence of sets $S_{1},S_{2},\dots \in {\mathcal {S}}$ s.t.~$\bigcup _{n\in \mathbb {N} }S_{n}=S$. In this framework $S$ does not need to belong to $\mathcal{S}$ (thus $\mathcal{S}$ is not necessarily an algebra) and arbitrary subsets of $S$ do not need to satisfy the condition $\bigcup_{n\in \mathbb {N} }A_{n}\in\mathcal{S}$ (thus $\mathcal{S}$ is not necessarily a $\sigma$-ring). 
\begin{defn}
	[Signed measure on a ring] A real-valued, non-negative set function $\mu(A)$ defined on the elements of a ring $\mathcal{R}$ will be called a signed measure, if $\mu(\emptyset)=0$ and if for every sequence $A_{1},A_{2}, . . .$ of disjoint sets of $\mathcal{R}$ for which $A=\bigcup_{k=1}^{\infty}A_{k}\in\mathcal{R}$ we have
	\begin{equation}\label{measure on a ring}
	\mu(A)=\sum_{k=1}^{\infty}\mu(A_{k})
	\end{equation}
	and the relation (\ref{measure on a ring}) holds absolutely (namely independent of the order of its elements).
\end{defn}
\noindent Similarly, it is possible to extend the definition of bimeasures on rings.
\\Moreover, we remark that it is possible to see that in our framework we can find a positive and a negative part (as we have done for quasi-L\'{e}vy measures) at least for any signed measures $\mu$ on $\mathcal{S}$ s.t.~ $\mu(S_{n})<\infty$ for every $n\in\mathbb{N}$. Indeed, consider such measure $\mu$. Notice that $(S_{n},\{S_{n}\cap B:B\in\mathcal{S}\})$ is a $\sigma$-algebra. Then $\mu$ on $(S_{n},\{S_{n}\cap B:B\in\mathcal{S}\})$ is a signed measure as in Definition \ref{Def-signedmeasure} and so we can extract the two unique mutually singular finite measures $\mu^{+}_{|S_{n}}$ and $\mu^{-}_{|S_{n}}$ on $(S_{n},\{S_{n}\cap B:B\in\mathcal{S}\})$. Since for any $B\in (S_{k},\{S_{k}\cap B:B\in\mathcal{S}\})$ we have that $\mu^{+}_{|S_{n}}(B)=\mu^{+}_{|S_{k}}(B)$ for any $n\geq k$ and since for every $n\in\mathbb{N}$ we have that $\mu^{+}_{|S_{n}}(S_{n})$ is finite, then we can uniquely extent $\mu^{+}_{|S_{n}}$ to a $\sigma$-finite measure which we denote by $\mu^{+}$ on $(S,\sigma(\mathcal{S}))$ (and the same holds for $\mu^{-}_{|S_{n}}$). Then for every $B\in\mathcal{S}$ we have that $\mu(B)=\mu^{+}(B)-\mu^{-}(B)$.
\begin{defn}[QID L\'{e}vy random measure]\label{defQIDr.m.} Let $\Lambda = \{\Lambda(A):\,\, A\in\mathcal{S}\}$ be a real stochastic process defined on some probability space $(\Omega,\mathcal{F},\mathbb{P})$. We call $\Lambda$ to be an independently scattered r.m., if, for every sequence $\{A_{n}\}$ of disjoint sets in $\mathcal{S}$, the r.v.~$\Lambda(A_{n})$, $n= 1, 2, ...,$ are independent, and, if $\bigcup_{n=1}^{\infty}A_{n}\in\mathcal{S}$, then we have $\Lambda(\bigcup_{n=1}^{\infty}A_{n})=\sum_{n=1}^{\infty}\Lambda(A_{n})$ a.s. (where the series is assumed to converge almost surely). In addition, if $\Lambda(A)$ is a QID (ID) r.v., for every $A\in\mathcal{S}$, then we call $\Lambda$ a QID (ID) r.m..
\end{defn}
\noindent In this work $\Lambda=\{\Lambda(A):A\in\mathcal{S}\}$ will denote a QID r.m.. Since, for every $A\in\mathcal{S}$, $\Lambda(A)$ is a QID r.v., its c.f.~can be written in the L\'{e}vy-Khintchine form:
\begin{equation}\label{cf}
\hat{\mathcal{L}}(\Lambda(A))(\theta):=\mathbb{E}(e^{i\theta\Lambda(A)})=\exp\left(i\theta \nu_{0}(A)-\frac{\theta^{2}}{2}\nu_{1}(A)+\int_{\mathbb{R}}e^{i\theta x}-1-i\theta\tau(x)F_{A}(dx)\right)
\end{equation}
where $-\infty<\nu_{0}(A)<\infty$, $0\leq \nu_{1}(A)<\infty$ and $F_{A}$ is a quasi-L\'{e}vy measure on $\mathbb{R}$, for $A\in\mathcal{S}$. 

For the sake of completeness we report here Proposition 2.1 of \cite{RajRos}.
\begin{pro}[Proposition 2.1 in \cite{RajRos}]\label{PropositionRajRos}
	\textnormal{(a)} Let $\Lambda$ be an ID r.m.~with the c.f.~given by
	\begin{equation}\label{rajros-cf}
\mathbb{E}(e^{i\theta\Lambda(A)})=\exp\left(i\theta \nu_{0}(A)-\frac{\theta^{2}}{2}\nu_{1}(A)+\int_{\mathbb{R}}e^{i\theta x}-1-i\theta\tau(x)F_{A}(dx)\right)
	\end{equation}
	Then $\nu_{0}:\mathcal{S}\mapsto\mathbb{R}$ is a signed-measure, $\nu_{1}:\mathcal{S}\mapsto[0, \infty)$ is a measure, $F_{A}$ is
	a L\'{e}vy measure on $\mathbb{R}$, for every $A\in\mathcal{S}$, and $\mathcal{S}\ni A\mapsto F_{A}(B)\in[0, \infty)$ is a measure,
	for every $B\in\mathcal{B}(\mathbb{R})$ s.t.~$0\notin \overline{B}$.
	\\\textnormal{(b)} Let $\nu_{0}$, $\nu_{1}$ and $F_{\cdot}$ satisfy the conditions given in \textnormal{(a)}. Then there exists	a unique (in the sense of finite-dimensional distributions) ID r.m.~$\Lambda$ such
	that $(\ref{rajros-cf})$ holds.
	\\\textnormal{(c)} Let $\nu_{0}$, $\nu_{1}$ and $F_{\cdot}$ as in \textnormal{(a)} and define
\begin{equation*}
\lambda(A)=|\nu_{0}|(A)+\nu_{1}(A)+\int_{\mathbb{R}}(1\wedge x^{2})F_{A}(dx),\quad A\in\mathcal{S}.
\end{equation*}
Then $\lambda:\mathcal{S}\mapsto[0,\infty)$ is a measure such that $\lambda(A_{n})\rightarrow0$ implies $\Lambda(A_{n})\stackrel{p}{\rightarrow}0$ for every $\{A_{n}\}\subset\mathcal{S}$; further if $\Lambda( A '_{n})\stackrel{p}{\rightarrow}0$ for every $\{A'_{n}\}\subset\mathcal{S}$ s.t.~$A'_{n}\subset A_{n}\in\mathcal{S}$, then $\lambda(A_{n})\rightarrow0$.
\end{pro}
Moreover, we recall the following result from the work of Pr\'{e}kopa (Theorem 2.1 in \cite{PrekopaI}). In doing this we also correct a typo in that statement.
\begin{thm}[Theorem 2.1 in \cite{PrekopaI}]\label{TheoremPrekopa}
In order that a finitely additive random measure $\xi(A)$ defined on the elements of the ring $\mathcal{R}$ should be countably additive it is necessary and sufficient that, for every non-increasing sequence of sets $B_{1},B_{2},...$ with $B_{k}\in\mathcal{R}$ $(k = 1, 2, . . .)$ and $B_{n}\searrow\emptyset$, $\xi(B_{n})\stackrel{p}{\rightarrow}0$ as $n\rightarrow\infty$.
\end{thm}
We conclude this section with the following remark on the conditions under which a quasi-L\'{e}vy type measure is a signed measure. From the discussion done in this section we have not mentioned \textit{when} a quasi-L\'{e}vy measure is a signed measure, namely when we can extend $\nu$ from $\mathcal{B}_{0}(\mathbb{R})$ to the $\sigma$-algebra $\mathcal{B}(\mathbb{R})$. The idea behind this is that $\nu$ is the difference of two L\'{e}vy measures $\mu_{1}$ and $\mu_{2}$ and when these measures are both infinite then we cannot find a signed measure such that $\nu=\mu_{1}-\mu_{2}$. This is because we get the form $\infty-\infty$, which is not defined. However, a sufficient condition is that one of the two L\'{e}vy measures is finite because then $\nu=\mu_{1}-\mu_{2}$ defines a signed measure. Further, we do not make any general assumption about it but we remark that in almost all the interesting cases considered in the existing literature, mainly in \cite{LPS}, quasi-L\'{e}vy measures are signed measures.
\section{The connection between ID and QID random measures}\label{chapter-IDvQID}
In this section, we investigate the connection between ID and QID random measures. We start with the following simple lemma, which follows from the preliminaries introduced in the previous section. Indeed, it is a formalisation of some of the concepts of Definition \ref{def1}.
\begin{lem}\label{lemma-last?}
	A set function $\nu:\mathcal{B}_{0}(\mathbb{R})\rightarrow\mathbb{R}$ is a quasi-L\'{e}vy type measure if and only if there exist two L\'{e}vy measures $\nu^{(1)}$ and $\nu^{(2)}$ s.t.~$\nu_{|\mathcal{B}_{r}(\mathbb{R})}(A)=\nu_{|\mathcal{B}_{r}(\mathbb{R})}^{(1)}(A)-\nu_{|\mathcal{B}_{r}(\mathbb{R})}^{(2)}(A)$ for every $A\in\mathcal{B}_{r}(\mathbb{R})$ with $r>0$. Moreover, $\nu$ is unique.
\end{lem}
\begin{proof}
	$\Rightarrow:$ Let $\nu$ be a quasi-L\'{e}vy type measure. Recall that we denote the positive and the negative part of $\nu$ by $\nu^{+}$ and $\nu^{-}$, respectively. Since $\nu^{+}(\{0\})=\nu^{-}(\{0\})=0$ and $\int_{\mathbb{R}}(1\wedge x^{2})|\nu|(dx)<\infty$ which implies that $\int_{\mathbb{R}}(1\wedge x^{2})\nu^{+}(dx)<\infty$ and $\int_{\mathbb{R}}(1\wedge x^{2})\nu^{-}(dx)<\infty$, then $\nu^{+}$ and $\nu^{-}$ are two L\'{e}vy measure on $\mathbb{R}$.
	\\ $\Leftarrow:$ Let $\nu^{(1)}$ and $\nu^{(2)}$ s.t.~$\nu_{|\mathcal{B}_{r}(\mathbb{R})}=\nu_{|\mathcal{B}_{r}(\mathbb{R})}^{(1)}-\nu_{|\mathcal{B}_{r}(\mathbb{R})}^{(2)}$ for every $A\in\mathcal{B}_{r}(\mathbb{R})$ with $r>0$. Since $\nu_{|\mathcal{B}_{r}(\mathbb{R})}^{(1)}$ and $\nu_{|\mathcal{B}_{r}(\mathbb{R})}^{(2)}$ are two finite measures then $\nu_{|\mathcal{B}_{r}(\mathbb{R})}$ is a finite signed measure. Thus, we denote by $\nu^{+}_{|\mathcal{B}_{r}(\mathbb{R})}$ and $\nu^{-}_{|\mathcal{B}_{r}(\mathbb{R})}$ its (unique) Jordan decomposition. Notice that $\nu^{+}_{|\mathcal{B}_{r}(\mathbb{R})}(A)=\nu^{+}_{|\mathcal{B}_{s}(\mathbb{R})}(A)$ for every $A\in\mathcal{B}_{r}(\mathbb{R})$ with $0<s\leq r$. Moreover, let $\nu^{+}(\{0\})=\nu^{+}(\{0\})=0$. Then, by Carath\'{e}odory's extension theorem we obtain the existence of two unique $\sigma$-finite measures $\nu^{+}$ and $\nu^{-}$ on $(\mathbb{R},\mathcal{B}(\mathbb{R}))$. It is possible to see that they satisfy the conditions of Definition \ref{def1}. Moreover, the uniqueness of $\nu$ comes from the uniqueness of $\nu^{+}$ and $\nu^{-}$.
\end{proof}
\noindent The above result obviously extends to the $\mathbb{R}^{d}$ case. Note that saying that a property holds ``for every $A\in\mathcal{B}_{r}(\mathbb{R})$ with $r>0$" is equivalent to saying that it holds ``for every $B\in\mathcal{B}(\mathbb{R})$ s.t.~$0\notin \overline{B}$".\\ In the following result represents a partial QID analogue of Proposition \ref{PropositionRajRos}.
\begin{lem}\label{pr1}
	Let $\nu_{0}:\mathcal{S}\mapsto\mathbb{R}$ be a signed measure, $\nu_{1}:\mathcal{S}\mapsto\mathbb{R}$ be a measure, $G_{A}$ be a L\'{e}vy measure on $\mathbb{R}$ for every $A\in\mathcal{S}$ and $\mathcal{S}\ni A\mapsto G_{A}(B)\in[0,\infty)$ be a measure for every $B\in\mathcal{B}(\mathbb{R})$ s.t.~$0\notin \overline{B}$. Let $M$ be defined as $G$. Let $F_{A}(B):=G_{A}(B)-M_{A}(B)$ for every $B\in\mathcal{B}(\mathbb{R})$ s.t.~$0\notin \overline{B}$ and $A\in\mathcal{S}$. Then, there is a unique quasi-L\'{e}vy type measure $F_{A}$ for every $A\in\mathcal{S}$. Moreover, $\mathcal{S}\ni A\mapsto F_{A}(B)\in(-\infty,\infty)$ is a signed measure for every $B\in\mathcal{B}(\mathbb{R})$ s.t.~$0\notin \overline{B}$.\\
	Further, if $(\nu_{0}(A), \nu_{1}(A),F_{A})$ is the characteristic triplet of a QID random variable $\forall A\in\mathcal{S}$. Then there exists a unique (in the sense of finite-dimensional distributions) QID random measure $\Lambda$ such that $(\ref{cf})$ holds, and $F_{A}$ is a quasi-L\'{e}vy measure for every $A\in\mathcal{S}$.
\end{lem}
\begin{proof}
	Since $G_{A}$ and $M_{A}$ are L\'{e}vy measures and since $F_{A}(B)=G_{A}(B)-M_{A}(B)$ for every $B\in\mathcal{B}(\mathbb{R})$ s.t.~$0\notin \overline{B}$ and $A\in\mathcal{S}$, then there is a unique quasi-L\'{e}vy type measure $F_{A}$ for every $A\in\mathcal{S}$. Moreover, since $F_{A}(B)=G_{A}(B)-M_{A}(B)$, it follows immediately that $A\mapsto F_{A}(B)$ is a signed measure for every $B\in\mathcal{B}(\mathbb{R})$ s.t.~$0\notin \overline{B}$.\\
	If $(\nu_{0}(A), \nu_{1}(A),F_{A})$ is the characteristic triplet of a QID random variable then $B\mapsto F_{A}(B)$ is by definition a quasi-L\'{e}vy measure. Moreover, the existence of a finitely additive independently scattered random measure $\Lambda=\{\lambda(A):A\in\mathcal{S}\}$ follows by a standard application of Kolmogorov extension theorem. To prove that $\Lambda$ is countable additive let $A_{n}\searrow\emptyset$ with $\{A_{n}\}\subset\mathcal{S}$. Then, by definition of (signed) measures, $\nu_{0}(A_{n})\rightarrow0$, $\nu_{1}(A_{n})\rightarrow0$, $G_{A_{n}}(B)\rightarrow0$ and $M_{A_{n}}(B)\rightarrow0$ for every $B\in\mathcal{B}(\mathbb{R})$ s.t.~$0\notin \overline{B}$. Notice that $G_{A}(B)\geq F_{A}^{+}(B)$ and $M_{A}(B)\geq F_{A}^{-}(B)$ for all $A\in\mathcal{S}$ and $B\in\mathcal{B}(\mathbb{R})$ s.t.~$0\notin \overline{B}$. Then, we obtain that 
\begin{equation*}
\int_{\mathbb{R}} (1\wedge x^{2} ) |F_{A_{n}}|(dx)\leq \int_{\mathbb{R}} (1\wedge x^{2} ) G_{A_{n}}(dx)+\int_{\mathbb{R}} (1\wedge x^{2} ) M_{A_{n}}(dx)\rightarrow0\quad\textnormal{as $n\rightarrow \infty$}
\end{equation*}
Now, by the L\'{e}vy continuity theorem we obtain that $\Lambda(A_{n})\stackrel{d}{\rightarrow}0$, which implies that $\Lambda(A_{n})\stackrel{p}{\rightarrow}0$ and, thus, by Theorem \ref{TheoremPrekopa} the countable additivity of $\Lambda$. The uniqueness follows by the one to one relation between $(\nu_{0}(A), \nu_{1}(A),F_{A})$ and $\Lambda(A)$, for every $A\in\mathcal{S}$.
\end{proof}
Now, we present the first link between ID and QID random measures.
\begin{pro}\label{IDvsQID}
	For every $A\in\mathcal{S}$, let $\Lambda(A)$ be a random variable and let two ID random measures $\Lambda_{1}$ and $\Lambda_{2}$ with $\Lambda(A)+\Lambda_{2}(A)\stackrel{d}{=}\Lambda_{1}(A)$ s.t.~$\Lambda_{2}(A)$ is independent of $\Lambda(A)$. Then, there exists a unique QID random measure $\Lambda=\{\Lambda(A):A\in\mathcal{S}\}$. In this case, we say that the r.m.~$\Lambda$ is \textnormal{generated by two ID random measures}.
\end{pro}
\begin{rem}
	The above result can be restated using instead of the condition that $\Lambda(A)$ is a random variable and there exist two ID random measures $\Lambda(A)+\Lambda_{2}(A)\stackrel{d}{=}\Lambda_{1}(A)$, such that $\Lambda_{2}(A)$ is independent of $\Lambda(A)$, $\forall A\in\mathcal{S}$, the equivalent alternative conditions: \\\textnormal{(i)} $\forall A\in\mathcal{S}$ and $\theta\in\mathbb{R}$, $\mathcal{L}(\Lambda(A))$ is a distribution and
	\begin{equation}\label{two-ID}
	\hat{\mathcal{L}}(\Lambda(A))(\theta)=\dfrac{\hat{\mathcal{L}}(\Lambda_{1}(A))(\theta)}{\hat{\mathcal{L}}(\Lambda_{2}(A))(\theta)},
	\end{equation}
	\textnormal{(ii)} $\forall A\in\mathcal{S}$ $\mathcal{L}(\Lambda(A))$ is a distribution and the c.t.~of $\Lambda(A)$ can be rewritten as the difference of the c.t.~of $\Lambda_{1}(A)$ and $\Lambda_{2}(A)$.
\end{rem}
\begin{proof} Let $\nu^{(j)}_{0}$, $\nu^{(j)}_{1}$ and $F_{\cdot}^{(j)}$ the measures corresponding to $\Lambda_{j}$, for $j=1,2$ (see Proposition \ref{PropositionRajRos}). Let $\nu_{0}(A):=\nu^{(1)}_{0}(A)-\nu^{(2)}_{0}(A)$, $\nu_{1}(A):=\nu^{(1)}_{1}(A)-\nu^{(2)}_{1}(A)$ for all $A\in\mathcal{S}$. Moreover, let $F_{A}$ be the unique quasi-L\'{e}vy measure defined by the difference between the L\'{e}vy measures $F_{A}^{(1)}$ and $F_{A}^{(2)}$, namely $F_{A}(B):=F^{(1)}_{A}(B)-F^{(2)}_{A}(B)$ for all $B\in\mathcal{B}(\mathbb{R})$ s.t.~$0\notin \overline{B}$ and $|F_{A}|(\{0\})=0$, for all $A\in\mathcal{S}$. Then $\nu_{0}$, $\nu_{1}$ and $F_{\cdot}$ satisfies the conditions of Lemma \ref{pr1}. \\Further, the condition that $\Lambda(A)$ is a random variable together with $\Lambda(A)+\Lambda_{2}(A)\stackrel{d}{=}\Lambda_{1}(A)$ implies that $\mathcal{L}(\Lambda(A))$ is QID with c.t.~$(\nu_{0}(A),\nu_{1}(A),F_{A})$, hence, by Lemma 2.7 in \cite{LPS} we have $\nu_{1}(A)\geq0$, $\forall A\in\mathcal{S}$.
\\	
Then, by Lemma \ref{pr1} we obtain the stated result (including the statement that $\Lambda$ is uniquely determined).
\end{proof}
In the following we present two results, where the second represent a partial ``only if" result of Proposition \ref{IDvsQID}.
		\begin{lem}\label{pr1-2}
			Let $\Lambda$ be a QID random measure and let $K_{\cdot}(B):A\mapsto K_{A}(B)$ be a measure s.t.~$K_{A}(B)\geq F_{A}(B)$, for every $A\in\mathcal{S}$ and $B\in\mathcal{B}(\mathbb{R})$ s.t.~$0\notin \overline{B}$. Assume that $K_{A}(\cdot):B\mapsto K_{A}(B)$ is L\'{e}vy measure for every $A\in\mathcal{S}$. Then $\nu_{0}:\mathcal{S}\mapsto\mathbb{R}$ is a signed measure, $\nu_{1}:\mathcal{S}\mapsto[0,\infty)$ is a measure, $F_{A}$ is a quasi-L\'{e}vy measure on $\mathbb{R}$, for every $A\in\mathcal{S}$, and $\mathcal{S}\ni A\mapsto F_{A}(B)\in(-\infty,\infty)$ is a signed measure, for every $B\in\mathcal{B}(\mathbb{R})$ s.t.~$0\notin \overline{B}$.
		\end{lem}
		\begin{proof}
			First, since $\Lambda$ is a QID r.m.~it follows that $F_{A}$ is a quasi-L\'{e}vy measure on $\mathbb{R}$, for every $A\in\mathcal{S}$. 
			\\Now, let $\{A_{k}\}_{k=1}^{n}$ be pairwise disjoint sets in $\mathcal{S}$. By the uniqueness of the L\'{e}vy-Khintchine representation of a quasi-ID distribution, it follows, using $\hat{\mathcal{L}}(\Lambda(\bigcup_{k=1}^{n}A_{k}))=\prod_{k=1}^{n}\hat{\mathcal{L}}(\Lambda(A_{k}))$, that all three set functions $\nu_{0}$, $\nu_{1}$ and $F_{\cdot}(B)$ (for every fixed $B\in\mathcal{B}(\mathbb{R})$ s.t.~$0\notin \overline{B}$) are finitely additive. Let now $\{A_{n}\}\subset\mathcal{S}$, $A_{n}\searrow\emptyset$.
			Then since $\Lambda(A_{n})\stackrel{p}{\rightarrow}0$ we have that $\nu_{0}(A_{n})\rightarrow0$, $\nu_{1}(A_{n})\rightarrow0$ and $\int_{\mathbb{R}}(1\wedge x^{2})F_{A_{n}}(dx)\rightarrow0$. Concerning $F_{A_{n}}(B)$, observe that $K_{A_{1}}(B)\geq K_{A_{2}}(B)\geq...$ and that $K_{A_{j}}(B)\geq F^{+}_{A_{j}}(B)$, where $j\in\mathbb{N}$, for every $B\in\mathcal{B}(\mathbb{R})$ s.t.~$0\notin \overline{B}$ by assumption. Then,
			\begin{equation*}
			\lim\limits_{n\rightarrow\infty}\int_{\mathbb{R}}(1\wedge x^{2})F^{+}_{A_{n}}(dx)\leq \lim\limits_{n\rightarrow\infty}\int_{\mathbb{R}}(1\wedge x^{2})K_{A_{n}}(dx)
			= \lim\limits_{n\rightarrow\infty}\int_{|x|<\epsilon}(1\wedge x^{2})K_{A_{n}}(dx)+\lim\limits_{n\rightarrow\infty}\int_{|x|\geq \epsilon}(1\wedge x^{2})K_{A_{n}}(dx)
			\end{equation*}
			\begin{equation*}
			\leq \int_{|x|<\epsilon}(1\wedge x^{2})K_{A_{1}}(dx)+\lim\limits_{n\rightarrow\infty}K_{A_{n}}(\{|x|\geq\epsilon\}).
			\end{equation*}
			The second addend goes directly to zero for every $\epsilon>0$ because $K_{\cdot}(B)$ is a measure, while for the first we obtain the convergence to zero by letting $\epsilon\rightarrow0$. Moreover, since
			\begin{equation*}
			\int_{\mathbb{R}}(1\wedge x^{2})F_{A_{n}}(dx)\rightarrow0\,\,\,\,\textnormal{as $n\rightarrow\infty$, then}\,\,\,\,			\int_{\mathbb{R}}(1\wedge x^{2})F^{-}_{A_{n}}(dx)\rightarrow0\,\,\,\,\textnormal{as $n\rightarrow\infty$.}
			\end{equation*} 
			Applying the arguments of the proof of Proposition 2.1 point (i) in \cite{RajRos}, namely by the Chebychev's inequality
			\begin{equation*}
			F^{+}_{A_{n}}(\{|x|\geq\epsilon\})\leq \epsilon^{-2} \int_{\mathbb{R}}(1\wedge x^{2})F^{+}_{A_{n}}(dx),
			\end{equation*}
			we obtain that $F^{+}_{A_{n}}(B)\rightarrow 0$ for every $B\in\mathcal{B}(\mathbb{R})$ s.t.~$0\notin \overline{B}$; and the same holds for $F^{-}_{A_{n}}(B)$. Therefore, we have $|F_{A_{n}}(B)|\rightarrow0$, and thus $F_{A_{n}}(B)$ is a signed measure, for every $B\in\mathcal{B}(\mathbb{R})$ s.t.~$0\notin \overline{B}$.
		\end{proof}
		\begin{co}
			Let $\Lambda$ be a QID random measure with $F_{\cdot}^{+}(B):A\mapsto F_{A}^{+}(B)$ a measure for every $B\in\mathcal{B}(\mathbb{R})$ s.t.~$0\notin \overline{B}$. Then $\nu_{0}:\mathcal{S}\mapsto\mathbb{R}$ is a signed measure, $\nu_{1}:\mathcal{S}\mapsto[0,\infty)$ is a measure, $F_{A}$ is a quasi-L\'{e}vy measure on $\mathbb{R}$, for every $A\in\mathcal{S}$, and $\mathcal{S}\ni A\mapsto F_{A}(B)\in(-\infty,\infty)$ is a signed measure, for every $B\in\mathcal{B}(\mathbb{R})$ s.t.~$0\notin \overline{B}$. Moreover, we have that $\mathcal{S}\ni A\mapsto F^{-}_{A}(B)\in[0,\infty)$ is a measure, for every $B\in\mathcal{B}(\mathbb{R})$ s.t.~$0\notin \overline{B}$.
		\end{co}
		\begin{rem}
			Notice that we could use the assumption $F_{\cdot}^{-}(\mathbb{R}):A\mapsto F_{A}^{-}(\mathbb{R})$ is a measure, instead. Then, the arguments would be the same.
		\end{rem}
		\begin{proof}
			It follows from the same arguments of Lemma \ref{pr1-2}.
		\end{proof}
\begin{pro}\label{IDvsQID-2}
	Let $\Lambda$ be QID random measure and let $K_{A}(B)$ as in Lemma \ref{pr1-2}. Then, there exist two ID random measures $\Lambda_{1}$ and $\Lambda_{2}$ with $\Lambda(A)+\Lambda_{2}(A)\stackrel{d}{=}\Lambda_{1}(A)$ such that $\Lambda_{2}(A)$  is independent of $\Lambda(A)$ and has zero Gaussian part, $\forall A\in\mathcal{S}$.
\end{pro}
\begin{proof}
 Let $F_{A}^{(1)}:=K_{A}$, $F_{A}^{(2)}(B):=K_{A}(B)-F_{A}(B)$ for every $B\in\mathcal{B}(\mathbb{R})$ s.t.~$0\notin \overline{B}$ and $F_{A}^{(2)}(\{0\})=0$. In addition, let $\nu_{0}^{(1)}:=2\nu_{0}$, $\nu_{0}^{(2)}:=\nu_{0}$, $\nu_{1}^{(1)}:=\nu_{1}$, and $\nu_{1}^{(2)}:=0$. By Lemma \ref{pr1-2}, we obtain that $F_{A}^{(j)}(\cdot):\mathcal{B}(\mathbb{R})\rightarrow[0,\infty)$, $F_{\cdot}^{(j)}(B):\mathcal{S}\rightarrow[0,\infty)$ for any $B\in\mathcal{B}(\mathbb{R})$ s.t.~$0\notin \overline{B}$, $\nu_{0}^{(j)}:\mathcal{S}\rightarrow[0,\infty)$ and $\nu_{0}^{(j)}:\mathcal{S}\rightarrow[0,\infty)$ are all measures, for $j=1,2$. In addition, since $F_{A}$ is a quasi-L\'{e}vy measure and $K_{A}$ is a L\'{e}vy measure then $F_{A}^{(2)}$ is L\'{e}vy measure. 
	\\ Then, by Proposition 2.1 point (ii) of \cite{RajRos} we know that there exist two unique ID random measures $\Lambda_{1}$ and $\Lambda_{2}$ such that, for every $A\in\mathcal{S}$, they have c.t.~given by $(\nu_{0}^{(1)}(A), \nu_{1}^{(1)}(A),F^{(1)}_{A})$ and $(\nu_{0}^{(2)}(A),\nu_{1}^{(2)}(A),F^{(2)}_{A})$, respectively. Finally, it is possible to see that the following relation hold:
	for every $A\in\mathcal{S}$ and $\theta\in\mathbb{R}$, $\hat{\mathcal{L}}(\Lambda(A))(\theta)=\frac{\hat{\mathcal{L}}(\Lambda_{1}(A))(\theta)}{\hat{\mathcal{L}}(\Lambda_{2}(A))(\theta)}$.
\end{proof}
\begin{rem}
	The reason why $\Lambda_{1}$ and $\Lambda_{2}$ are not determined uniquely by $\Lambda$ is because given $\nu_{0}$ we cannot exclude the possibility of finding two measures $\tilde{\nu}_{0}^{(1)}$ and $\tilde{\nu}_{0}^{(2)}$ such that $\nu_{0}(A)=\tilde{\nu}_{0}^{(1)}(A)-\tilde{\nu}_{0}^{(2)}(A)$, where $\tilde{\nu}_{0}^{(1)}(A)\neq 2\nu_{0}(A)$ and $\tilde{\nu}_{0}^{(2)}(A)\neq \nu_{0}(A)$. This leads to two measures $\tilde{\Lambda}_{1}$ and $\tilde{\Lambda}_{2}$ such that they satisfies the same conditions as $\Lambda_{1}$ and $\Lambda_{2}$. Moreover, the same reasoning applied to $\nu_{0}$ may also applies to $F_{A}$.
\end{rem}
In the following example we observe that the condition on $F_{A}^{+}(B)$ being a measure is not always satisfied even if we restrict to random measure $\Lambda$ with $\mathcal{L}(\Lambda(A))$ concentrated on $\mathbb{Z}$, $\forall A\in\mathcal{S}$.
\begin{exmp}
	Consider $F_{A}=\sum_{l\in \mathbb{Z},l\neq0}b^{A}_{l}\delta_{l}$ where $b^{A}_{l}$ are real numbers, then $|F_{A}|=\sum_{l\in \mathbb{Z},l\neq0}|b^{A}_{l}|\delta_{l}$. Such a quasi-L\'{e}vy measure arises when we consider a distribution supported on the integers whose c.f.~has no zeroes (see Theorem 8.1 in \cite{LPS}). While $F_{A}^{+}$ and $F_{A}^{-}$ are L\'{e}vy measure for every fixed $A\in\mathcal{S}$ they are not always measures for fixed $B\in\mathcal{B}(\mathbb{R})$. Indeed, consider $A=A_{1}\cap A_{2}$ with $A_{1}\cap A_{2}=\emptyset$. Since $F_{A}$ is a signed measure then $b^{A_{1}\cup A_{2}}_{k}=b^{A_{1}}_{k}+b^{A_{2}}_{k}$. However, it might happen that $b^{A_{1}}_{k}$ and $b^{A_{2}}_{k}$ have different sign. This implies that $F_{A_{1}\cup A_{2}}^{+}(\{k\})\neq F_{A_{1}}^{+}(\{k\})+F_{A_{2}}^{+}(\{k\})$ and $F_{A_{1}\cup A_{2}}^{-}(\{k\})\neq F_{A_{1}}^{-}(\{k\})+F_{A_{2}}^{-}(\{k\})$.
	\\ On the other hand, this example shows that if we consider $F_{A}=\sum_{l\in \mathbb{Z},l\neq0}b^{A}_{l}\delta_{l}$ such that $b^{A}_{l}$, with $l\in\mathbb{Z}$, do not change sign for all $A\in\mathcal{S}$, then this represents an example where the theory presented so far applies.
\end{exmp}
\section{QID random measure}\label{Ch-QIDrm}
In the previous section we have seen that under certain conditions two ID random measure generates a QID random measure and vice versa. In this section we are going to explore other cases and show necessary and sufficient conditions for the existence and uniqueness of QID random measure.
\subsection{The distribution concentrated on the integers}
In this section we are going to investigate the case where supp($\mathcal{L}(\Lambda(A))$)$\subset\mathbb{Z}$. In other words, we consider $\mathcal{L}(\Lambda(A))=\sum_{m\in\mathbb{Z}}a_{m}\delta_{m}$, for $a_{m}\in\mathbb{R}$. We start by showing the equivalent of Proposition 2.1 of \cite{RajRos}.
\begin{pro}\label{pr1-Z}
	\textnormal{(i)} Let $\Lambda$ be a QID random measure with \textnormal{supp(}$\mathcal{L}(\Lambda(A))$\textnormal{)}$\subset\mathbb{Z}$, $\forall A\in\mathcal{S}$. Then for each $A\in\mathcal{S}$, $\nu_{1}(A)=0$, the drift of $\hat{\mathcal{L}}(\Lambda(A))\in\mathbb{Z}$ and $F_{A}$ is a quasi-L\'{e}vy measure concentrated on $\mathbb{Z}\setminus\{0\}$ and finite. Moreover, $\nu_{0}:\mathcal{S}\mapsto\mathbb{R}$ is a signed measure, $\nu_{1}:\mathcal{S}\mapsto\mathbb{R}$ is the zero measure, $F_{A}$ is a quasi-L\'{e}vy measure on $\mathbb{R}$, for every $A\in\mathcal{S}$, and $\mathcal{S}\ni A\mapsto F_{A}(B)\in(-\infty,\infty)$ is a signed measure, for every $B\in\mathcal{B}(\mathbb{R})$.
	\\ \textnormal{(ii)} Let $\nu_{0}$, $\nu_{1}$ and $F_{\cdot}$ be as in \textnormal{(i)} and such that $(\ref{cf})$ is the c.f.~of a distribution and has no zeros, for each $A\in\mathcal{S}$. Then there exists a unique (in the sense of finite-dimensional distributions) quasi-ID random measure $\Lambda$ s.t.~$(\ref{cf})$ holds.
\end{pro}
\begin{proof}
	(i) The first assertion is the content of Theorem 8.1 of \cite{LPS}. 
	For the finite additivity of $\nu_{0}$, $\nu_{1}$ and $F_{\cdot}(B)$ we follow the first part of the proof of Proposition 2.1 of \cite{RajRos}. Indeed, let $\{A_{k}\}_{k=1}^{n}$ be pairwise disjoint sets in $\mathcal{S}$. By the uniqueness of L\'{e}vy-Khintchine representation of the c.f.~of a QID distribution and by $\hat{\mathcal{L}}\left(\Lambda\left(\bigcup_{k=1}^{n} A\right)\right)=\prod_{k=1}^{n} \hat{\mathcal{L}}(\Lambda(A_{k}))$, we obtain the finite additivity property. Now, let $\{A_{n}\}\subset\mathcal{S}$, $A_{n}\searrow \emptyset$. Since $\Lambda(A_{n})\rightarrow0$ then $\hat{\mathcal{L}}(\Lambda(A_{n}))\rightarrow\hat{\delta_{0}}$ and by Theorem 8.5 of \cite{LPS} we obtain that $\nu_{0}(A_{n})\rightarrow0$, $\nu_{1}(A_{n})\rightarrow0$ and $\sum_{l\in \mathbb{Z}}|F_{A_{n}}(\{l\})|\rightarrow 0$. Now, observe that since $F_{A_{n}}=\sum_{l\in \mathbb{Z},l\neq0}b_{n,l}\delta_{l}$ for some $b_{n,l}\in\mathbb{R}$, then $|F_{A_{n}}|=\sum_{l\in \mathbb{Z},l\neq0}|b_{n,l}|\delta_{l}$, and since $F_{A_{n}}(\{l\})=b_{n,l}$, we obtain that $|F_{A_{n}}|(\mathbb{R})\rightarrow0$, because we have that $F_{A_{n}}(\{k\})\rightarrow0$ $\forall k\in\mathbb{Z}$. Therefore, we deduce that $F_{A_{n}}(B)\rightarrow0$ for every $B\in\mathcal{B}(\mathbb{R})$.
	
	(ii) From Theorem 8.5 of \cite{LPS} we obtain that $(\nu_{0}(A), \nu_{1}(A),F_{A})$ is the characteristic triplet of a QID random variable $\forall A\in\mathcal{S}$. Further, the existence of a finitely additive independently scattered random measure $\Lambda=\{\lambda(A):A\in\mathcal{S}\}$ follows by the Kolmogorov extension theorem. To prove that $\Lambda$ is countable additive let $A_{n}\searrow\emptyset$ with $\{A_{n}\}\subset\mathcal{S}$. Then, by definition of (signed) measures, $\nu_{0}(A_{n})\rightarrow0$, $\nu_{1}(A_{n})\rightarrow0$ and $F_{A_{n}}(B)\rightarrow0$ for every $B\in\mathcal{B}(\mathbb{R})$ s.t.~$0\notin \overline{B}$. From the last limit and since $F_{A_{n}}$ is concentrated on the integers, we obtain that $|F_{A_{n}}|(\mathbb{R})\rightarrow0$. Therefore, we have that $\Lambda(A_{n})\stackrel{p}{\rightarrow}0$, hence $\Lambda$ is countably additive.
\end{proof}
\begin{open}
	Following Corollary 8.2 in \cite{LPS}, it might be possible to extend the above result to the case of distributions concentrated on a lattice of the form $r+h\mathbb{Z}$ where $r\in\mathbb{R}$ and $h>0$. Is it really possible?
\end{open}
\subsection{The distribution with an atom of mass greater than 1/2}
In this section we are going to explore the case where a random measure is such that $\mathcal{L}(\Lambda(A))(\{k_{A}\})>1/2$ for some $k_{A}\in\mathbb{R}$. We have the following result for $k_{A}=0$.
\begin{pro}
	\textnormal{(i)} Let $\Lambda$ be an independent scattered random measure with $p_{A}:=\mathcal{L}(\Lambda(A))(\{0\})>1/2$, $\forall A\in\mathcal{S}$. Then $\Lambda$ is a QID random measure and $\nu_{1}(A)=0$, the drift of $\hat{\mathcal{L}}(\Lambda(A))$ equal zero, and $F_{A}$ is a finite quasi-L\'{e}vy measure given by
	\begin{equation}\label{formula-mass-atom}
F_{A}=\left(\sum_{m=1}^{\infty}\frac{1}{m}(-1)^{m+1}\left(\frac{1-p_{A}}{p_{A}}\right)^{m}(\delta_{0}\ast\sigma_{A})^{*m} \right)_{\mathbb{R}\setminus\{0\}},
	\end{equation}
	where $\sigma_{A}=(1-p_{A})^{-1}(\mathcal{L}(\Lambda(A))-p_{A}\delta_{0})$, for every $A\in\mathcal{S}$,\\
	Moreover, $\nu_{0}:\mathcal{S}\mapsto\mathbb{R}$ and $\nu_{1}:\mathcal{S}\mapsto\mathbb{R}$ are the zero measure, $F_{A}$ is a quasi-L\'{e}vy measure on $\mathbb{R}$, for every $A\in\mathcal{S}$, and $\mathcal{S}\ni A\mapsto F_{A}(B)\in(-\infty,\infty)$ is a signed measure, for every $B\in\mathcal{B}(\mathbb{R})$.
	\\ \textnormal{(ii)} Let $\nu_{0}$ and $\nu_{1}$ be the zero measures, let $\mu_{A}$ be a distribution with c.t.~$(0,0, F_{A})$ where $F_{A}$ is a quasi-L\'{e}vy measure on $\mathbb{R}$, for every $A\in\mathcal{S}$, and $\mathcal{S}\ni A\mapsto F_{A}(B)\in(-\infty,\infty)$ is a signed measure, for every $B\in\mathcal{B}(\mathbb{R})$. In particular, let
		\begin{equation*}
		F_{A}=\left(\sum_{m=1}^{\infty}\frac{1}{m}(-1)^{m+1}\left(\frac{1-p_{A}}{p_{A}}\right)^{m}(\delta_{0}\ast\sigma_{A})^{*m} \right)_{\mathbb{R}\setminus\{0\}},
		\end{equation*}
	where $p_{A}=\mu_{A}(\{0\})>1/2$ and $\sigma_{A}=(1-p_{A})^{-1}(\mu_{A}-p_{A}\delta_{0})$, $\forall A\in\mathcal{S}$. Then there exists a unique (in the sense of finite-dimensional distributions) quasi-ID random measure $\Lambda$ s.t.~$\mathcal{L}(\Lambda(A))=\mu_{A}$, for every $A\in\mathcal{S}$.
\end{pro}
\begin{proof}
	(i) The first statement is the content of Theorem 4.3.7 in \cite{Cuppens} (see also Theorem 3.1 of \cite{LPS}). Now, the finite additivity follows from the same arguments used before (see the proof of Proposition \ref{pr1-Z}). For countable additivity, let $\{A_{n}\}\subset\mathcal{S}$, $A_{n}\searrow \emptyset$. Since $\Lambda(A_{n})\stackrel{p}{\rightarrow}0$, then $\mathcal{L}(\Lambda(A_{n}))\rightarrow\delta_{0}$, which implies that $p_{A_{n}}\rightarrow1$. Then, from $(\ref{formula-mass-atom})$ and using the fact that $F_{A_{n}}$ is finite we get $F_{A_{n}}(B)\rightarrow0$ for every $B\in\mathcal{B}(\mathbb{R})$.
	
	(ii) The existence and the finite additivity of $\Lambda$ are straightforward. Concerning the countable additivity, notice that if $F_{A_{n}}(B)\rightarrow0$ for every $B\in\mathcal{B}(\mathbb{R})$, then $p_{A_{n}}\rightarrow1$, which implies that $\mathcal{L}(\Lambda(A_{n}))\stackrel{d}{\rightarrow}\delta_{0}$.
\end{proof}
\subsection{The strictly negative L\'{e}vy measure}
Under certain conditions, it is possible to have QID distributions with strictly negative quasi-L\'{e}vy measure. An example of such conditions are presented in Lemma 2.8 in \cite{LPS}. Thus, we have the following corollary.
\begin{co}
	All the results presented in Section II of \cite{RajRos} apply \textnormal{mutatis mutandis} to the case of QID random measures with quasi-L\'{e}vy measure which take only non-positive values.
\end{co}
\begin{proof}
	Consider $F_{A}(\cdot)$ to be our non-positive quasi-L\'{e}vy measure, where $A\in\mathcal{S}$. Then all the results in Section II of \cite{RajRos} hold for the total variation $|F_{A}|(\cdot)$. But since $F_{A}=-|F_{A}|$ we obtain the stated result.
\end{proof}
\section{QID Stochastic integrals}\label{Ch-QIDstochint}
In this and in the following subsections of this section we investigate different results concerning QID stochastic integrals, including the necessary and sufficient conditions for their existence. The presentation of these subsections are similar. However, each section has its own framework and its own particular result.
\\ This section represents the QID extension of Chapter II in \cite{RajRos}, which is at the heart of the theory of ID r.m. and processes.
\subsection{The generating two ID random measure case}\label{Section-geerating two ID}
In this section we focus on the cases seen in Section \ref{chapter-IDvQID}. In particular, the setting is the one of Lemma \ref{pr1} (or equivalently of Proposition \ref{IDvsQID}), which is more general than the one of Lemma \ref{pr1-2}. Indeed, all the results presented in this section apply to the setting of Lemma \ref{pr1-2} and they are obtained by simply substituting $G$ by $K$ (or $F^{+}$) and $M$ by $K-F$ (or $F^{-}$).\\ The first result regards the construction of the control measure of $\Lambda$. 
\begin{pro}
		Let $\nu_{0}$, $\nu_{1}$ and $F_{\cdot}$ be as in Lemma \ref{pr1}. Define
		\begin{equation}\label{lambda1}
		\lambda(A)=|\nu_{0}|(A)+\nu_{1}(A)+\int_{\mathbb{R}}(1\wedge x^{2})G_{A}(dx)+\int_{\mathbb{R}}(1\wedge x^{2})M_{A}(dx),\quad A\in\mathcal{S}.
		\end{equation}
		Then $\lambda:\mathcal{S}\mapsto[0,\infty)$ is a measure such that $\lambda(A_{n})\rightarrow0$ implies $\Lambda(A_{n})\stackrel{p}{\rightarrow}0$ for every $\{A_{n}\}\subset\mathcal{S}$.
\end{pro}
\begin{proof}
First, notice that all the element on the right hand side of $(\ref{lambda1})$ are measures. Hence, $\lambda$ is a measure. Further, let $A_{n}\searrow\emptyset$ with $\{A_{n}\}\subset\mathcal{S}$. Since $G_{A_{1}}(B)\geq G_{A_{n}}(B)$ for every $n\geq1$ and $B\in\mathcal{B}(\mathbb{R})$ s.t.~$0\notin \overline{B}$, we have that
\begin{equation*}
\lim\limits_{n\rightarrow\infty}\int_{\mathbb{R}}(1\wedge x^{2})G_{A_{n}}(dx)= \lim\limits_{n\rightarrow\infty}\int_{|x|<\epsilon}(1\wedge x^{2})G_{A_{n}}(dx)+\lim\limits_{n\rightarrow\infty}\int_{|x|\geq \epsilon}(1\wedge x^{2})G_{A_{n}}(dx)
\end{equation*}
\begin{equation*}
\leq \int_{|x|<\epsilon}(1\wedge x^{2})G_{A_{1}}(dx)+\lim\limits_{n\rightarrow\infty}G_{A_{n}}(\{|x|\geq\epsilon\})
\end{equation*}
The second addend goes directly to zero for every $\epsilon>0$ and for the first we obtain the convergence to zero by letting $\epsilon\rightarrow0$. The same holds for $M$. Hence, $\lambda$ is countably additive. Moreover, using Theorem 4.3 point (a) in \cite{LPS} or the L\'{e}vy continuity theorem it is straightforward to see that if $\lambda(A_{n})\rightarrow0$ then $\Lambda(A_{n})\stackrel{p}{\rightarrow}0$.
\end{proof}
\begin{defn}
	Since $\lambda(S_{n})<\infty$, $n=1,2,...$ we extend $\lambda$ to a $\sigma$-finite measure on $(S,\sigma(\mathcal{S}))$; we call $\lambda$ the \textnormal{control measure} of $\Lambda$.
\end{defn}
It is possible to use different control measure then the one mentioned, for example $\tilde{\lambda}=\lambda_{G}+\lambda_{M}$, where $\lambda_{G}$ and $\lambda_{M}$ are the control measure of $\Lambda_{G}$ and $\Lambda_{M}$. This $\tilde{\lambda}$ has the additional property that if $\Lambda_{G}( A '_{n})\rightarrow0$ and $\Lambda_{M}( A '_{n})\stackrel{p}{\rightarrow}0$ for every $\{A'_{n}\}\subset\mathcal{S}$ s.t.~$A'_{n}\subset A_{n}\in\mathcal{S}$, then (by Proposition \ref{PropositionRajRos}) $\lambda(A_{n})\rightarrow0$ .\\
The reason why we prefer $\lambda$ is because it is potentially the smallest (up to a constant) measure such that $\lambda\gg\log\hat{\mathcal{L}}(\Lambda)(\theta)$ for every $\theta\in\mathbb{R}$. The reason why $\lambda$ is only ``potentially" the smallest has to do with an old and classical problem, the Maharam's Problem posed in 1947 by Maharam \cite{Maharam}, which has been solved in 2008 by Talagrand \cite{Talagrand2}. In few words and applied to our case, the solution of this problem states that an exhaustive submeasure (think of $A\mapsto F^{+}_{A}(B)$) is not always absolutely continuous w.r.t.~a measure. Indeed, by considering that there exists two generating ID r.m., we impose that $G_{A}\geq F^{+}_{A}$ and $M_{A}\geq F^{-}_{A}$, and thus solving the issue. We will see in the next two sections how to obtain a control measure without this assumption (but with other assumptions).

Since in this section we are concerned with QID random measures then it is of secondary importance which specific couple of ID r.m.~generate $\Lambda$. Indeed, given two ID r.m.~$\Lambda_{G}$ and $\Lambda_{M}$ with c.f.~$(\nu_{0}^{G},\nu_{1}^{G},G_{\cdot})$ and $(\nu_{0}^{M},\nu_{1}^{M},M_{\cdot})$ that generate the QID r.m.~$\Lambda$, we can always find two different ID r.m.~with (for example) c.t.~$(0,\nu_{1},G_{\cdot})$ and $(-\nu_{0},0,M_{\cdot})$ which also generate $\Lambda$.\\
Throughout this section we let $\Lambda_{G}\sim(0,\nu_{1},G_{\cdot})$ and $\Lambda_{M}\sim(-\nu_{0},0,M_{\cdot})$. We denote by $\lambda_{G}$ and $\lambda_{M}$ the control measure of $\Lambda_{G}$ and $\Lambda_{M}$, namely $\lambda_{G}(A)=\nu_{1}(A)+\int_{\mathbb{R}}(1\wedge x^{2})G_{A}(dx)$ and $\lambda_{M}(A)=\|\nu_{0}\|(A)+\int_{\mathbb{R}}(1\wedge x^{2})M_{A}(dx)$.

The following result is a corollary of Lemma 2.3 in \cite{RajRos}.
\begin{co}\label{l1}
	Let $F_{\cdot}$ be as in Lemma \ref{pr1-2}. Then there exist two unique $\sigma$-finite signed measures $G$ and $M$ on $\sigma(\mathcal{S})\otimes\mathcal{B}(\mathbb{R})$ such that
	\begin{equation*}
	G(A\times B)=G_{A}(B)\quad\text{and}\quad M(A\times B)=M_{A}(B),\quad\text{for all $A\in\mathcal{S}$, $B\in\mathcal{B}(\mathbb{R})$}.
	\end{equation*}
	Moreover, there exist two functions $\rho_{G},\rho_{M}:S\times\mathcal{B}(\mathbb{R})\mapsto[-\infty,\infty]$ such that
	\\ \textnormal{(i)} $\rho_{G}(s,\cdot),\rho_{M}(s,\cdot)$ are L\'{e}vy measures on $\mathcal{B}(\mathbb{R})$, for every $s\in S$,
	\\ \textnormal{(ii)} $\rho_{G}(\cdot,B),\rho_{M}(\cdot,B)$ are Borel measurable function, for every $B\in\mathcal{B}(\mathbb{R})$,
	\\ \textnormal{(iii)} $\int_{S\times\mathbb{R}}h(s,x)G(ds,dx)=\int_{S}\int_{\mathbb{R}}h(s,x)\rho_{G}(s,dx)\lambda(ds)$, for every $\sigma(\mathcal{S})\otimes\mathcal{B}(\mathbb{R})$-measurable function $h:S\times\mathbb{R}\mapsto[0,\infty]$, and the same holds for $M$. This equality can be extended to real and complex-valued functions $h$.
\end{co}
\begin{proof}
	The above statement is almost identical to Lemma 2.3 in \cite{RajRos}. The only differences are two: a different function, call it $\tilde{\rho}_{G}$, instead of $\rho_{G}$, and $\lambda_{G}$ instead of $\lambda$. In the proof of Lemma 2.3 in \cite{RajRos} $\tilde{\rho}_{G}:=\frac{d\lambda_{0}}{d\lambda_{G}}(s)(1\wedge x^{2})^{-1}q(s,dx)$ where $q(s,B)$ is such that $s\mapsto q(s,B)$ is Borel measurable and $B\mapsto q(s,B)$ is a probability measure. Moreover, it is shown that $G(A,B)=\int_{A}q(s,B)\lambda_{0}ds$. Now, let
	\begin{equation*}
	\rho_{G}(s,dx)=\frac{d\lambda_{0}}{d\lambda}(s)(1\wedge x^{2})^{-1}q(s,dx).
	\end{equation*}
	It is possible to see that $s\mapsto\rho_{G}(s,B)$ is still measurable and since $\lambda_{0}(A)\geq\lambda(A)$ then $B\mapsto\rho_{G}(s,B)$ is still a Levy measure too. Finally, by the properties of the Radon-Nikodym derivative we get $\int_{S\times\mathbb{R}}h(s,x)G(ds,dx)=\int_{S}\int_{\mathbb{R}}h(s,x)\tilde{\rho}_{G}(s,dx)\lambda_{G}(ds)=\int_{S}\int_{\mathbb{R}}h(s,x)\rho_{G}(s,dx)\lambda(ds)$.
\end{proof}
\begin{pro}\label{pro-discordia}
	Let $F_{\cdot}$ be as in Lemma \ref{pr1-2}. Then there exists a unique function $\rho:S\times\mathcal{B}_{0}(\mathbb{R})\mapsto[-\infty,\infty]$ s.t.
	 \begin{equation*}
	 \rho(s,B)=\rho_{G}(s,B)-\rho_{M}(s,B),\quad s\in S,\text{ $B\in\mathcal{B}(\mathbb{R})$ s.t.~$0\notin \overline{B}$}
	 \end{equation*}
	The function $\rho(s,\cdot)$ is a quasi-L\'{e}vy type measure on $\mathcal{B}(\mathbb{R})$ for every $s\in S$. Moreover, its positive and negative part are two unique functions $\rho^{+},\rho^{-}:S\times\mathcal{B}(\mathbb{R})\mapsto[0,\infty]$ such that
	\\ \textnormal{(i)} $\rho^{+}(s,\cdot),\rho^{-}(s,\cdot)$ are L\'{e}vy measures on $\mathcal{B}(\mathbb{R})$, for every $s\in S$,
	\\ \textnormal{(ii)} $\rho^{+}(\cdot,B),\rho^{-}(\cdot,B)$ are Borel measurable function, for every $B\in\mathcal{B}(\mathbb{R})$,
	\\Moreover, there exist two unique $\sigma$-finite measures $\tilde{F}^{+}$ and $\tilde{F}^{-}$ on $\sigma(\mathcal{S})\otimes\mathcal{B}(\mathbb{R})$ such that 
	\begin{equation}\label{eq-discordia}
	\int_{S\times\mathbb{R}}h(s,x)\tilde{F}^{+}(ds,dx)=\int_{S}\int_{\mathbb{R}}h(s,x)\rho^{+}(s,dx)\lambda(ds),
	\end{equation}
	for every $\sigma(\mathcal{S})\otimes\mathcal{B}(\mathbb{R})$-measurable function $h:S\times\mathbb{R}\mapsto[0,\infty]$, and the same holds for $F^{-}$. This equality can be extended to real and complex-valued functions $h$. Finally, for every $A\in\mathcal{S}$ and for every real $\mathcal{B}(\mathbb{R})$-measurable function $g$ s.t.~$\int_{A}\int_{\mathbb{R}}g(x)|\rho|(s,dx)\lambda(ds)<\infty$, we have that
	\begin{equation*}
	\int_{\mathbb{R}}g(x)F_{A}(dx)=\int_{A}\int_{\mathbb{R}}g(x)\rho(s,dx)\lambda(ds)
	\end{equation*}
	and for every $B\in\mathcal{B}(\mathbb{R})$ s.t.~$0\notin \overline{B}$,
	\begin{equation*}
	\tilde{F}^{+}(A,B)\geq F_{A}^{+}(B)\quad\text{and}\quad \tilde{F}^{-}(A,B)\geq F_{A}^{-}(B).
	\end{equation*}
\end{pro}
\begin{rem}
	The reason why there is no a signed measure $\tilde{F}$ on $\sigma(\mathcal{S})\otimes\mathcal{B}(\mathbb{R})$ s.t.~$\tilde{F}(C)=\tilde{F}^{+}(C)-\tilde{F}^{-}(C)$ is that both $\tilde{F}^{+}$ and $\tilde{F}^{-}$ might have infinite values.
\end{rem}
\begin{proof}
		Since, for every $s\in S$, $\rho_{G}(s,B)$ and $\rho_{M}(s,B)$ are two L\'{e}vy measures, we obtain the first statement by the definition of quasi-L\'{e}vy type measure.
		\\ This implies that, for each $s\in S$, $\rho(s,B)$ has a positive and negative part defined as measures on $(\mathbb{R},\mathcal{B}(\mathbb{R}))$. Then, it is easy to obtain (i). Concerning (ii), we have the following. Let $f:\mathbb{R}\rightarrow\mathbb{R}$ be s.t.~$\int_{\mathbb{R}}f(x)|\rho|(s,dx)<\infty$ for every $s\in S$ and let $f^{+}(x):= f(x)\vee 0$ and $f^{-}(x):= -f(x)\vee 0$. Observe that the following are all measures
		\begin{equation*}
		B\mapsto\int_{B}f^{+}(x)\rho^{+}(s,dx),\quad B\mapsto \int_{B}f^{+}(x)\rho^{-}(s,dx),
		\end{equation*} 
		\begin{equation*}
		B\mapsto\int_{B}f^{+}(x)\rho_{G}(s,dx),\quad\textnormal{and}\quad B\mapsto\int_{B}f^{+}(x)\rho_{M}(s,dx).
		\end{equation*}
		Moreover, we have that
		\begin{equation*}
		\int_{B}f^{+}(x)\rho^{+}(s,dx)-\int_{B}f^{+}(x)\rho^{-}(s,dx)=\int_{B}f^{+}(x)\rho(s,dx)=\int_{B}f^{+}(x)\rho_{G}(s,dx)-\int_{B}f^{+}(x)\rho_{M}(s,dx).
		\end{equation*}
		Since $B\mapsto\int_{B}f^{+}(x)\rho^{+}(s,dx)-\int_{B}f^{+}(x)\rho^{-}(s,dx)$ is a finite signed measure, it has a Hahn decomposition. Denote it by $E^{+}$ and $E^{-}$. Further, since $\rho^{+}(s,B)$ and $\rho^{-}(s,B)$ are mutually singular we have that
		\begin{equation*}
		\int_{E^{+}}f^{+}(x)\rho^{+}(s,dx)-\int_{E^{+}}f^{+}(x)\rho^{-}(s,dx)=\int_{E^{+}}f^{+}(x)\rho^{+}(s,dx)
		\end{equation*}
		and similarly for $E^{-}$. Therefore, we have that for every $B\in\mathcal{B}(\mathbb{R})$ we deduce that
		\begin{equation*}
		\int_{B}f^{+}(x)\rho^{+}(s,dx)=\int_{B\cap E^{+}}f^{+}(x)\rho^{+}(s,dx)=\int_{B\cap E^{+}}f^{+}(x)\rho_{G}(s,dx)-\int_{B\cap E^{+}}f^{+}(x)\rho_{M}(s,dx).
		\end{equation*}
		Observe that the last two terms are Borel measurable functions because $\rho_{G}(s,B)$ and $\rho_{M}(s,B)$ are Borel measurable functions thanks to Corollary \ref{l1}. Since the difference of two measurable function is measurable we conclude that $\int_{B}f^{+}(x)\rho^{+}(s,dx)$ is measurable and this holds for every $B\in\mathcal{B}(\mathbb{R})$. The same holds for $\int_{B}f^{+}(x)\rho^{-}(s,dx)$ and for $\int_{B}f^{-}(x)\rho^{+}(s,dx)$ and $\int_{B}f^{-}(x)\rho^{-}(s,dx)$. Therefore, since
		\begin{equation*}
		\int_{B}f(x)\rho^{+}(s,dx)=\int_{B}f^{+}(x)\rho^{+}(s,dx)-\int_{B}f^{-}(x)\rho^{+}(s,dx)
		\end{equation*}
		(and similarly for $\rho^{-}$) for every $B\in\mathcal{B}(\mathbb{R})$,	we obtain that $\int_{B}f(x)\rho^{+}(s,dx)$ and so $f(x)\rho^{+}(s,dx)$ are Borel a measurable functions. 
		\\Let us now take $f(x)=1\wedge x^{2}$, it is possible to see that it is a valid choice since $\int_{\mathbb{R}}(1\wedge x^{2})|\rho|(s,dx)<\infty$. Then $s\mapsto(1\wedge x^{2})\rho^{+}(s,dx)$ is a measurable function and so (as done in the proof of Lemma 2.3 in \cite{RajRos}) we have that
		\begin{equation*}
		\rho^{+}(s,dx)=\left(1\wedge x^{2}\right)^{-1}(1\wedge x^{2}) \rho^{+}(s,dx)
		\end{equation*}
		is a Borel measurable function. The same applies to $\rho^{-}(s,dx)$ and, therefore, we obtain (ii).
		
		Now, notice that the set function 
		\begin{equation*}
		\tilde{F}^{+}(C):=\int_{S}\int_{\mathbb{R}}\textbf{1}_{C}(s,x)\rho^{+}(s,dx)\lambda(ds),\quad C\in\sigma(\mathcal{S})\otimes\mathcal{B}(\mathbb{R})
		\end{equation*}
		is a well defined measure on $\sigma(\mathcal{S})\otimes\mathcal{B}(\mathbb{R})$. Then, using standard measure theoretical arguments we get (\ref{eq-discordia}).
		
		 The last statement follows by simple computations. Indeed, let $g(x)$ be s.t.~$\int_{A}\int_{\mathbb{R}}g(x)|\rho|(s,dx)\lambda(ds)<\infty$ (notice that this includes also $g(x)=\textbf{1}_{B}(x)$ for every $B\in\mathcal{B}(\mathbb{R})$ s.t.~$0\notin \overline{B}$). Then, for every $A\in\mathcal{S}$, $s\mapsto\textbf{1}_{A}(s)g(x)\rho(s,dx)$ is a Borel measurable function and 
		\begin{equation*}
\int_{A}\int_{\mathbb{R}}g(x)\rho(s,dx)\lambda(ds)
		=\int_{S}\int_{\mathbb{R}}\textbf{1}_{A}(s)g(x)\rho^{+}(s,dx)\lambda(ds)-\int_{S}\int_{\mathbb{R}}\textbf{1}_{A}(s)g(x)\rho^{-}(s,dx)\lambda(ds)
		\end{equation*}
				\begin{equation*}
				=\int_{S}\int_{\mathbb{R}}\textbf{1}_{A}(s)g(x)\rho_{G}(s,dx)\lambda(ds)-\int_{S}\int_{\mathbb{R}}\textbf{1}_{A}(s)g(x)\rho_{M}(s,dx)\lambda(ds)
				\end{equation*}
		\begin{equation*}
		=\int_{\mathbb{R}}g(x)G_{A}(dx)-\int_{\mathbb{R}}g(x)M_{A}(dx)=\int_{\mathbb{R}}g(x)F_{A}(dx).
		\end{equation*}
		Further, let $g(x)\geq 0$ (\textit{e.g.}~$\textbf{1}_{B}(x)$) and let $E^{+}_{A}$ and $E^{-}_{A}$ be the Hahn decomposition of $\mathbb{R}$ under $F_{A}(\cdot)$, then
		\begin{equation*}
		\int_{\mathbb{R}}g(x)F^{+}_{A}(dx)=\int_{E_{A}^{+}}g(x)F_{A}(dx)=\int_{S}\int_{E_{A}^{+}}\textbf{1}_{A}(s)g(x)\rho(s,dx)\lambda(ds) 
		\end{equation*}
		\begin{equation*}
		\leq \int_{S}\int_{E_{A}^{+}}\textbf{1}_{A}(s)g(x)\rho^{+}(s,dx)\lambda(ds)\leq\int_{S}\int_{A}g(x)\rho^{+}(s,dx)\lambda(ds)=\int_{\mathbb{R}}g(x)\tilde{F}^{+}(A,dx).
		\end{equation*}
\end{proof}
We present now one of the key results of this section, because it represents the first step for the construction of the c.f.~of $\int_{S} f\Lambda$. Moreover the following results is frequently used in the field of (applied) probability especially when we the control measure is the Lebesgue measure.
\begin{pro}\label{pr-2}
	The characteristic function of $\Lambda(A)$ can be written as:
	\begin{equation*}
	\mathbb{E}(e^{i\theta\Lambda(A)})=\exp\left(\int_{A}K(\theta,s)\lambda(ds)\right),\quad \theta\in\mathbb{R},A\in\mathcal{S},
	\end{equation*}
	where
	\begin{equation*}
	K(\theta,s)=i\theta a(s)-\frac{\theta^{2}}{2}\sigma^{2}(s)+\int_{\mathbb{R}}e^{i\theta x}-1-i\theta\tau(x)\rho(s,dx),
	\end{equation*}
	$a(s)=\frac{d\nu_{0}}{d\lambda}(s)$, $\sigma^{2}(s)=\frac{d\nu_{1}}{d\lambda}(s)$ and $\rho$ has been introduced in Proposition \ref{pro-discordia}, and $\exp(K(\theta,s))$ is the characteristic function of a QID random variable if it exists. Moreover, we have
	\begin{equation*}
	|a(s)|+ \sigma^{2}(s)+\int_{\mathbb{R}}(1\wedge x^{2})\rho_{G}(s,dx)+\int_{\mathbb{R}}(1\wedge x^{2})\rho_{M}(s,dx)=1,\quad\quad \text{$\lambda$-a.e}..
	\end{equation*}
\end{pro}
\begin{proof}
	The first statement follows from the L\'{e}vy-Khintchine formulation, Corollary \ref{l1} and Proposition \ref{pro-discordia}.	In particular, since $\int_{A}\int_{\mathbb{R}}(1\wedge x^{2})|\rho|(s,dx)\lambda(ds)<\infty$ then
	\begin{equation*}
	\int_{\mathbb{R}}e^{i\theta x}-1-i\theta\tau(x)F_{A}(dx)=\int_{A}\int_{\mathbb{R}}e^{i\theta x}-1-i\theta\tau(x)\rho(s,dx)\lambda(ds)
	\end{equation*}
	and so
	\begin{equation*}
	\mathbb{E}(e^{i\theta\Lambda(A)})=\exp\left(i\theta \nu_{0}(A)-\frac{\theta^{2}}{2}\nu_{1}(A)+\int_{\mathbb{R}}e^{i\theta x}-1-i\theta\tau(x)F_{A}(dx)\right)
	\end{equation*}
	\begin{equation*}
	=\exp\left(\int_{A}i\theta a(s)-\frac{\theta^{2}}{2}\sigma^{2}(s)+\int_{\mathbb{R}}\left(e^{i\theta x}-1-i\theta\tau(x)\right)\rho(s,dx)\lambda(ds)\right)
	\end{equation*}
	The second statement follows from the fact that for every $A\in\mathcal{S}$, we have
	\begin{equation*}
	\int_{A}\left(|a(s)|+\sigma^{2}(s)+\int_{\mathbb{R}}(1\wedge x^{2})\rho_{G}(s,dx)+\int_{\mathbb{R}}(1\wedge x^{2})\rho_{M}(s,dx)\right)\lambda(ds)
	\end{equation*}
	\begin{equation*}
	=|\nu_{0}|(A)+\nu_{1}(A)+\int_{\mathbb{R}}(1\wedge x^{2})G_{A}(dx)+\int_{\mathbb{R}}(1\wedge x^{2})M_{A}(dx)=\lambda(A)=\int_{A}\lambda(ds).
	\end{equation*}
\end{proof}
The following definition of the stochastic integral is the same as the one presented in \cite{RajRos} but extended to the framework of QID random measure.
\begin{defn}\label{def-integral}
Let $f(s)=\sum_{j=1}^{n}x_{j}\mathbf{1}_{A_{j}}(s)$ be a real simple function on $S$, where $A_{j}\in\mathcal{S}$ are disjoint. Then, for every $A\in\sigma(\mathcal{S})$, we define
\begin{equation*}
\int_{A}fd\Lambda=\sum_{j=1}^{n}x_{j}\Lambda(A\cap A_{j}).
\end{equation*}
Further, a measurable function $f:(S,\sigma(\mathcal{S}))\rightarrow(\mathbb{R},\mathcal{B}(\mathbb{R}))$ is said to be $\Lambda$-integrable if there exists a sequence $\{f_{n}\}$ of simple functions such that 
\\\textnormal{(i)} $f_{n}\rightarrow f$, $\lambda$-a.e.,
\\\textnormal{(ii)} for every $A\in\sigma(\mathcal{S})$, the sequence $\{\int_{A}f_{n}d\Lambda\}$ converges in probability as $n\rightarrow\infty$.

If $f$ is $\Lambda$-integrable, then we write
\begin{equation*}
\int_{A}fd\Lambda=\mathbb{P}-\lim\limits_{n\rightarrow\infty}\int_{A}f_{n}d\Lambda
\end{equation*}
where $\{f_{n}\}$ satisfies $\textnormal{(i)}$ and $\textnormal{(ii)}$.
\end{defn}
Notice that for a simple function $f(s)=\sum_{j=1}^{n}x_{j}\mathbf{1}_{A_{j}}(s)$ we have the following representation of the c.f.:
\begin{equation*}
\mathbb{E}\left(e^{i\theta\int_{A}fd\Lambda}\right)=\prod_{j=1}^{n}\mathbb{E}\left(e^{i\theta x_{j}\Lambda(A\cap A_{j})}\right)
\end{equation*}
	\begin{equation*}
	=\prod_{j=1}^{n}\exp\left(\int_{A\cap A_{j}}i\theta x_{j} a(s)-\frac{\theta^{2}x_{j}^{2}}{2}\sigma^{2}(s)+\int_{\mathbb{R}}\left(e^{i\theta x_{j} x}-1-i\theta x_{j}\tau(x)\right)\rho(s,dx)\lambda(ds)\right)
	\end{equation*}
\begin{equation*}
=\exp\left(\int_{A}i\theta f(s) a(s)-\frac{\theta^{2}f^{2}(s)}{2}\sigma^{2}(s)+\int_{\mathbb{R}}\left(e^{i\theta f(s) x}-1-i\theta f(s)\tau(x)\right)\rho(s,dx)\lambda(ds)\right).
\end{equation*}
In the following result, we prove that $\int_{A}fd\Lambda$ does not depend on the approximating sequence, hence it is well defined. This does not follow from \cite{U-W} or from \cite{RosDissertation} since they focus on ID random variables, however we use some of their arguments. In particular, in \cite{U-W} (which is the work cited in \cite{RajRos}) the random measures considered is \textit{atomless} (see Section \ref{Sec-Atomless} for the definition and for further details) while we (and \cite{RosDissertation}) do not have such a restriction.
\begin{lem}\label{lm-well deined}Let $\Lambda$ be a QID random measure and let $f$ be $\Lambda$-integrable then $\int_{A}fd\Lambda$ is well defined, for every $A\in\sigma(\mathcal{S})$
\end{lem}
\begin{proof}
Let $\{f_{n}\}$ and $\{g_{n}\}$ be two real simple functions on $S$ as in the Definition \ref{def-integral} and satisfying (i) and (ii) with the same limit. Let $h_{n}=f_{n}-g_{n}$ ($n=1,2,...$). Then, the sequence $\{h_{n}\}\rightarrow0$ $\lambda$-a.e.~and $\int_{A}h_{n}d\Lambda$ converges in probability for every $A\in\sigma(\mathcal{S})$. It remains to show that $\int_{A}h_{n}d\Lambda\stackrel{p}{\rightarrow}0$ for every $A\in\sigma(\mathcal{S})$. 
\\ Let $N_{n}(A)=\int_{A}f_{n}d\Lambda$ then $N_{n}$ is a measure on $L_{0}(\Omega,\mathcal{F},\mathbb{P};\mathbb{R})$, which is the set of all measurable random elements defined on $(\Omega,\mathcal{F},\mathbb{P})$ with values in $\mathbb{R}$. In particular, $N_{n}$ is absolutely continuous with respect to $\lambda$ and for every $A\in\sigma(\mathcal{S})$ we have that $N(A):=\lim\limits_{n\rightarrow\infty}N_{n}(A)$ exists. In particular the existence is guaranteed by the convergence of $\int_{A}h_{n}d\Lambda$. Applying the Hahn-Saks-Vitali theorem we obtain that $N$ is a measure on $L_{0}(\Omega,\mathcal{F},\mathbb{P};\mathbb{R})$ and $N\ll\lambda$. 
\\ Since $\{h_{n}\}$ converges $\lambda$-a.e., we can apply the Egorov (or more correctly Lusin) theorem and obtain that for every set $A$ in $\sigma(\mathcal{S})$ we have $A=\bigcup_{k=0}^{\infty}A_{k}$ where $A_{k}$'s are pairwise disjoint sets belonging to $\sigma(\mathcal{S})$ s.t.~$\lambda(A_{0})=0$ and $h_{n}\rightarrow0$ uniformly on every set $A_{k}$'s for $k>0$. Then, for every $A_{k}$ we have
\begin{equation*}
\hat{\mathcal{L}}\left(N(A_{k}) \right)=\lim\limits_{n\rightarrow\infty}\hat{\mathcal{L}}\left(N_{n}(A_{k}) \right)=\lim\limits_{n\rightarrow\infty}\exp\left(\int_{A_{k}}K(\theta h_{n}(s),s)\lambda(ds)\right)
\end{equation*}
\begin{equation*}
=\lim\limits_{n\rightarrow\infty}\exp\left(\int_{A_{k}}i\theta h_{n}(s)a(s)-\frac{\theta^{2}}{2}h^{2}_{n}(s)\sigma^{2}(s)+\int_{\mathbb{R}}\left(e^{i\theta h_{n}(s) x}-1-i\theta h_{n}(s)\tau(x)\right)\rho(s,dx)\lambda(ds)\right)=1,
\end{equation*}
using the dominated convergence theorem and that $\sup\limits_{s\in A_{k}}h_{n}(s)\rightarrow0$ as $n\rightarrow\infty$. Hence, $N(A_{k})=0$ in probability for every $k>0$. Then, since $N(A)$ is a measure we have that $N(A)=\sum_{k=0}^{\infty}N(A_{k})=0$, for every $A\in\sigma(\mathcal{S})$.
\end{proof}
It is now possible to give a representation of the c.f.~of $\int_{S}fd\Lambda$.
\begin{pro}\label{pr-same-as-in-the-paper}
	If $f$ is $\Lambda$-integrable, then $\int_{S}|K(\theta f(s),s)|\lambda(ds)<\infty$, where $K$ is given in Proposition \ref{pr-2}, and
	\begin{equation*}
	\hat{\mathcal{L}}\left(\int_{S}fd\Lambda\right)(\theta)=\exp\left(\int_{S} K(\theta f(s),s)\lambda(ds)\right),\quad \theta\in\mathbb{R}.
	\end{equation*}
\end{pro}
\begin{proof}
	The statement follows from the same arguments used in the proof of Proposition 2.6 of \cite{RajRos}. Let us sketch them. First, notice that the statement holds for simple function as in Definition \ref{def-integral}. Let $\{f_{n}\}$ be a sequence of simple functions. Let $\mu_{\theta,n}(A):=\int_{A} K(\theta f_{n}(s),s)\lambda(ds)$ where $A\in\sigma(\mathcal{S})$ and observe that it is a complex measure for every $\theta\in\mathbb{R}$ and $n\in\mathbb{N}$. Since 
	\begin{equation*}
	\lim\limits_{n\rightarrow\infty}\mu_{\theta,n}(A)=\lim\limits_{n\rightarrow\infty}\log\hat{\mathcal{L}}\left(\int_{A}f_{n}d\Lambda\right)(\theta)=\log\hat{\mathcal{L}}\left(\int_{S}fd\Lambda\right)(\theta)=\mu_{\theta}(A),
	\end{equation*}
	for every $A\in\sigma(\mathcal{S})$ and $\theta\in\mathbb{R}$, we conclude that by the Hahn-Saks-Vitali theorem $\mu_{\theta}$ is a countably additive complex measure. Further, since $\mu_{\theta}<<\lambda$ then $\mu_{\theta}(A)=\int_{A}h_{\theta}(s)\lambda(ds)$. By the continuity of $K(\cdot,s)$ for each $s\in S$ we deduce that $K(\theta f_{n}(s), s)\rightarrow K(\theta f(s), s)$ a.e.-$\lambda$, and using Egorov (Lusin) theorem we have that $S=\bigcup_{j=0}^{\infty}A_{j}$, where $\lambda(A_{0})=0$ and $\lambda(A_{j})<\infty$ and with $K(\theta f_{n}(s), s)\rightarrow K(\theta f(s), s)$ a.e.-$\lambda$ uniformly in $s\in A_{j}$ for $j\in\mathbb{N}$. Thus, $h_{\theta}(s)=K(\theta f(s), s)$ a.e.-$\lambda$ on $A_{j}$ with $j\geq 1$ because for every $A\in\sigma(\mathcal{S})$
	\begin{equation*}
\int_{A\cap A_{j}}h_{\theta}(s)\lambda(ds)=\mu_{\theta}(A\cap A_{j})=\lim\limits_{n\rightarrow\infty}\int_{A\cap A_{j}} K(\theta f_{n}(s),s)\lambda(ds)=\int_{A\cap A_{j}} K(\theta f(s),s)\lambda(ds)
	\end{equation*}
	But since $A_{0}$ is a $\lambda$-null set then $h_{\theta}(s)=K(\theta f(s), s)$ a.e.-$\lambda$ on $S$.
\end{proof}
We are now ready to present the main result of this subsection, which concerns the integrability conditions for $\int_{S}fd\Lambda$.
\begin{thm}\label{theorem1}
	Let $f:S\rightarrow\mathbb{R}$ be a $\mathcal{S}$-measurable function. Then $f$ is $\Lambda$-integrable if the following three conditions hold:
	\\\textnormal{(i)} $\int_{S}|U(f(s),s)|\lambda(ds)<\infty$,
	\\\textnormal{(ii)} $\int_{S}|f(s)|^{2}\sigma^{2}(s)\lambda(ds)<\infty$,
	\\\textnormal{(iii)} $\int_{S}V_{0}(f(s),s)\lambda(ds)<\infty$,
	\begin{equation*}
	\textnormal{where}\quad U(u,s)=ua(s)+\int_{\mathbb{R}}\tau(xu)-u\tau(x)\rho(s,dx),\quad V_{0}(u,s)=\int_{\mathbb{R}}(1\wedge |xu|^{2})|\rho|(s,dx),
	\end{equation*}
	Further, the c.f.~of $\int_{S}fd\Lambda$ can be written as
	\\\textnormal{(iv)} $\hat{\mathcal{L}}\left(\int_{S}fd\Lambda\right)(\theta)=\exp\left(i\theta a_{f}-\frac{1}{2}\theta^{2}\sigma_{f}^{2}+\int_{\mathbb{R}}e^{i\theta x}-1-i\theta\tau(x)F_{f}(dx) \right)$,
	where
	\begin{equation*}
	a_{f}=\int_{S}U(f(s),s)\lambda(ds),\quad \sigma_{f}^{2}=\int_{S}|f(s)|^{2}\sigma^{2}(s)\lambda(ds),\quad\textnormal{and}
	\end{equation*}
	$F_{f}$ is the unique quasi-L\'{e}vy measure determined by the difference of the L\'{e}vy measures $\tilde{F}^{+}_{f}$ and $\tilde{F}^{-}_{f}$, which are defined as: for every $B\in\mathcal{B}(\mathbb{R})$
	\begin{equation*}
	\tilde{F}^{+}_{f}(B)=\tilde{F}^{+}(\{(s,x)\in S\times\mathbb{R}:\, f(s)x\in B\setminus\{0 \} \})\,\,\,\text{and}\,\,\, \tilde{F}^{-}_{f}(B)=\tilde{F}^{-}(\{(s,x)\in S\times\mathbb{R}:\, f(s)x\in B\setminus\{0 \} \})
	\end{equation*}
\end{thm}
\begin{proof}	
	In this proof we follow some of the arguments used in the proof of Theorem 2.7 in \cite{RajRos}. Assume that (i), (ii), (iii) hold. Consider the sets $S_{1},S_{2},... \in {\mathcal {S}}$ s.t.~$\bigcup _{n\in \mathbb {N} }S_{n}=S$. Let $A_{n}=\{s:|f(s)|\leq n \}\cap S_{n}$. We have that $\{A_{n}\}\subset\mathcal{S}$ and $A_{n}\nearrow S$. Consider a sequence $(f_{n})$ of simple $\mathcal{S}$-measurable functions, s.t.~$f_{n}(s)=0$ if $s\in A_{n}$, $|f_{n}(s)-f(s)|\leq \frac{1}{n}$ if $s\in A_{n}$ , and $|f_{n}(s)|\leq |f(s)|$ for all $s\in S$. Then, $f_{n}\rightarrow f$ everywhere on $S$ as $n\rightarrow\infty$. Further, notice that for every $A\in\sigma(\mathcal{S})$ and $n,m\geq 1$, $|\left(f_{n}(s)-f_{m}(s) \right)\textbf{1}_{A}(s)|\leq 2|f(s)|$. Then by Lemma \ref{l3} (see the next result) we derive
	\begin{equation*}
	|U\left(\left(f_{n}(s)-f_{m}(s) \right)\textbf{1}_{A}(s),s\right)|\leq 2|U(f(s),s)|+27V_{0}(f(s),s).
	\end{equation*}
	Thus, by the dominated convergence theorem, we obtain that for every $A\in\sigma(\mathcal{S})$,
	\begin{equation}\label{lim1}
	\lim\limits_{n,m\rightarrow\infty}\int_{S}U((f_{n}(s)-f_{m}(s)), \textbf{1}_{A}(s),s)\lambda(ds)=0,
	\end{equation}
	\begin{equation}\label{lim2}
	\lim\limits_{n,m\rightarrow\infty}\int_{S}(f_{n}(s)-f_{m}(s))^{2} \textbf{1}_{A}(s)\sigma^{2}(s)\lambda(ds)=0,
	\end{equation}
	\begin{equation}\label{lim3}
	\textnormal{and}\quad\lim\limits_{n,m\rightarrow\infty}\int_{S}V_{0}((f_{n}(s)-f_{m}(s)) \textbf{1}_{A}(s),s)\lambda(ds)=0.
	\end{equation}	
	Then, from (i), (ii), (iii) and Proposition \ref{pr-2} it is possible to obtain (iv) for the simple measurable functions $(f_{n}(s)-f_{m}(s)) \textbf{1}_{A}$. Indeed, let $g_{n,m,A}(s):=(f_{n}(s)-f_{m}(s)) \textbf{1}_{A}$, we have
	\begin{equation*}
	\hat{\mathcal{L}}\left(\int_{S}g_{n,m,A}d\Lambda\right)(\theta)
	\end{equation*}
	\begin{equation*}
	=\exp\left(i\theta a_{g_{n,m,A}}-\frac{1}{2}\theta^{2}\sigma_{g_{n,m,A}}^{2}+\int_{S}\int_{\mathbb{R}}\left(e^{i\theta g_{n,m,A}(s) x}-1-i\theta g_{n,m,A}(s)\tau(x)\right)\rho(s,dx)\lambda(ds) \right)
	\end{equation*}
	\begin{equation*}
	=\exp\bigg(i\theta a_{g_{n,m,A}}-\frac{1}{2}\theta^{2}\sigma_{g_{n,m,A}}^{2}+\int_{S}\int_{\mathbb{R}}\left(e^{i\theta g_{n,m,A}(s) x}-1-i\theta g_{n,m,A}(s)\tau(x)\right)\rho^{+}(s,dx)\lambda(ds)
	\end{equation*}
	\begin{equation*}
	-\int_{S}\int_{\mathbb{R}}e^{i\theta g_{n,m,A}(s) x}-1-i\theta g_{n,m,A}(s)\tau(x)\rho^{-}(s,dx)\lambda(ds)\bigg)
	\end{equation*}
	\begin{equation*}
	=\exp\left(i\theta a_{g_{n,m,A}}-\frac{1}{2}\theta^{2}\sigma_{g_{n,m,A}}^{2}+\int_{\mathbb{R}}e^{i\theta x}-1-i\theta\tau(x)\tilde{F}^{+}_{g_{n,m,A}}(dx)-\int_{\mathbb{R}}e^{i\theta x}-1-i\theta\tau(x)\tilde{F}^{-}_{g_{n,m,A}}(dx)\right)
	\end{equation*}
	\begin{equation*}
	=\exp\left(i\theta a_{g_{n,m,A}}-\frac{1}{2}\theta^{2}\sigma_{g_{n,m,A}}^{2}+\int_{\mathbb{R}}e^{i\theta x}-1-i\theta\tau(x)F_{g_{n,m,A}}(dx)\right)
	\end{equation*}
	Now, using (iv) for simple functions, (\ref{lim1}), (\ref{lim2}), and (\ref{lim3}), we get $\lim\limits_{n,m\rightarrow\infty}\hat{\mathcal{L}}\left(\int_{S}(f_{n}(s)-f_{m}(s)) \textbf{1}_{A}d\Lambda\right)(\theta)=1$ for every $\theta\in\mathbb{R}$ and $A\in\sigma(\mathcal{S})$. Therefore, the sequence $\{\int_{A}f_{n}d\Lambda \}_{n=1}^{\infty}$ converges in probability for every $A\in\sigma(\mathcal{S})$, namely $f$ is $\Lambda$-integrable and, then, the c.f.~of $\int_{S}fd\Lambda$ can be written as in (iv).
\end{proof}
\begin{lem}\label{l3}
	For every $u\in\mathbb{R}$, $s\in S$ and $d>0$,
	\begin{equation*}
	\sup\{|U(cu,s)|:|c|\leq d \}\leq d|U(u,s)|+(1+d)^{3}V_{0}(u,s).
	\end{equation*}
\end{lem}
\begin{proof}
	Let $|c|\leq d$. Define $R(c,u,s):=\int_{\mathbb{R}}\tau(xcu)-c\tau(ux)\rho(s,dx)$. We can rewrite $U(cu,s)$ as follows
	\begin{equation*}
	U(cu,s)=cua(s)+\int_{\mathbb{R}}\tau(xcu)-cu\tau(x)\rho(s,dx)
	\end{equation*}
	\begin{equation*}
	=cua(s)+c\int_{\mathbb{R}}\tau(xu)-u\tau(x)\rho(s,dx)+\int_{\mathbb{R}}\tau(xcu)-c\tau(ux)\rho(s,dx)=cU(u,s)+R(c,u,s).
	\end{equation*}
	Observe that $\tau(xcu)-c\tau(ux)=0$ if $|ux|\leq (1\wedge |c|^{-1})$ and $|\tau(xcu)-c\tau(ux)|\leq 1+d$ otherwise. Thus, we obtain, using Chebychev's inequality, that 
	\begin{equation*}
	|R(c,u,s)|\leq (1+d)\int_{|ux|>(1\wedge |c|^{-1})}|\rho|(s,dx)\leq (1+d)|\rho|(s,\{x:(1\wedge |ux|)\geq (1\wedge |c|^{-1}) \})
	\end{equation*}
	\begin{equation*}
	\leq\frac{1+d}{(1\wedge |c|^{-2})}\int_{\mathbb{R}}(1\wedge |ux|^{2})|\rho|(s,dx)\leq (1+d)^{3}V_{0}(u,s).
	\end{equation*}
\end{proof}
The reason why it is not possible to get an ``only if" result is because we have
\begin{equation*}
\int_{\mathbb{R}}(1\wedge x^{2})|F_{f}|(dx)\leq\int_{S}\int_{\mathbb{R}}(1\wedge |f(s)x|^{2})|\rho|(s,dx)\lambda(ds),
\end{equation*}
where $|F_{f}|$ is the total variation of the quasi-L\'{e}vy measure $F_{f}$. Hence, we do not have an equality but only an inequality, which implies that if $\int_{\mathbb{R}}(1\wedge x^{2})|F_{f}|(dx)<\infty$ then it is not necessarily true that $\int_{S}\int_{\mathbb{R}}(1\wedge |f(s)x|^{2})|\rho|(s,dx)\lambda(ds)<\infty$ (\textit{i.e.}~point (iii) in Theorem \ref{theorem1}). Indeed, notice that
\begin{equation*}
\int_{\mathbb{R}}(1\wedge x^{2})|F_{f}|(dx)\leq\int_{\mathbb{R}}(1\wedge x^{2})\tilde{F}^{+}_{f}(dx)+\int_{\mathbb{R}}(1\wedge x^{2})\tilde{F}^{-}_{f}(dx)
=\int_{S}\int_{\mathbb{R}}(1\wedge |f(s)x|^{2})|\rho|(s,dx)\lambda(ds),
\end{equation*}
where the inequality comes from the following argument. Recall that $\tilde{F}^{+}_{f}(B)=\tilde{F}^{+}(\{(s,x)\in S\times\mathbb{R}:\,\, f(s)x\in B\setminus\{0 \} \})$, for every $B\in\mathcal{B}(\mathbb{R})$ and let $C^{+}\in \sigma(\mathcal{S})\otimes\mathcal{B}(\mathbb{R})$ s.t.~$\tilde{F}^{+}(A)>0$ for every $A\subset C^{+}$ and similarly we define $C^{-}$ for $\tilde{F}^{-}$. We have that $\tilde{F}^{+}_{f}(B)\geq F^{+}_{f}(B)$ because there might exist two sets $A_{1},A_{2}\in \sigma(\mathcal{S})\otimes\mathcal{B}(\mathbb{R})$ with $A_{1}\in C^{+}$ and $A_{2}\in C^{-}$ s.t.~$f(s)x\in B\in\mathcal{B}(\mathbb{R})$ for every $(s,x)\in A_{1}$ and $(s,x)\in A_{2}$. So, for this specific $B$ we will have both $\tilde{F}^{+}_{f}(B)$ and $\tilde{F}^{-}_{f}(B)$ strictly positive, thus explaining the inequality. For example, think of $(s,x)\in A_{1}$ and $(s',x')\in A_{2}$, it might happen that $f(s)x=f(s')x'$, especially if the image of $f$ is the whole $\mathbb{R}$. 

However, by adding a further assumption we have the following result.
\begin{thm}\label{theorem2}
	Let $f:S\rightarrow\mathbb{R}$ be a $\mathcal{S}$-measurable function. Then, if $f$ is integrable with respect to $\Lambda_{G}$ and $\Lambda_{M}$, the three conditions \textnormal{(i)}, \textnormal{(ii)} and \textnormal{(iii)} of Theorem \ref{theorem1} hold, namely $f$ is $\Lambda$-integrable. Further, the c.f.~of $\int_{S}fd\Lambda$ can be written as in point \textnormal{(iv)} of Theorem \ref{theorem1}.
\end{thm}
\begin{proof}
From the fact that $|\hat{\mathcal{L}}\left(\int_{S}fd\Lambda_{G}\right)|^{2}$ and $|\hat{\mathcal{L}}\left(\int_{S}fd\Lambda_{M}\right)|^{2}$ are the c.f.~of ID-random variables it follows that 
\begin{equation*}
\bigg|\hat{\mathcal{L}}\left(\int_{S}fd\Lambda_{G}\right)(\theta)\bigg|^{2}=\exp\left(-\theta^{2}\sigma_{G,f}^{2}+2\int_{\mathbb{R}}\cos(\theta x)-1G_{f}(dx) \right).
\end{equation*}
where $G_{f}(B)=G(\{(s,x)\in S\times\mathbb{R}:\, f(s)x\in B\setminus\{0 \} \})$ and where $\sigma^{2}_{G,f}=\int_{S}|f(s)|\sigma^{2}_{G}(s)\lambda_{G}(ds)$. By the properties of the Radon-Nikodym derivative we have that
\begin{equation*}
\sigma^{2}_{G,f}=\int_{S}|f(s)|\sigma^{2}_{G}(s)\lambda_{G}(ds)=\int_{S}|f(s)|\frac{d\nu^{(1)}}{d\lambda_{G}}(s)\lambda_{G}(ds)=\int_{S}|f(s)|\frac{d\nu^{(1)}}{d\lambda}(s)\lambda(ds)=\sigma^{2}_{f}
\end{equation*}
Therefore we get that $\sigma^{2}_{f}<\infty$ and that
\begin{equation*}
\int_{\mathbb{R}}(1\wedge y^{2})\tilde{F}^{+}_{f}(dy)+\int_{\mathbb{R}}(1\wedge y^{2})\tilde{F}^{-}_{f}(dy)\leq\int_{\mathbb{R}}(1\wedge y^{2})G_{f}(dy)+\int_{\mathbb{R}}(1\wedge y^{2})M_{f}(dy)<\infty,
\end{equation*}
where $M_{f}$ is defined similarly to $G_{f}$, and thus we obtain (ii) and (iii). Further, since $|\tau(x)-\sin(x)|\leq 2(1\wedge x^{2})$, (i) follows from noticing that
		\begin{equation*}
		|U(u,s)|\leq \bigg|ua(s)+\int_{\mathbb{R}}\sin(xu)-u\tau(x)\rho(s,dx)\bigg| +\bigg|\int_{\mathbb{R}}\tau(xu)-\sin(xu)|\rho|(s,dx)\bigg|
		\leq |\textnormal{Im}\, K(u,s)|+2V_{0}(u,s),
		\end{equation*}
		which is finite because of (iii) and of Proposition \ref{pr-same-as-in-the-paper}.
\end{proof}
We end this section with the following result, which links $\int_{A}f_{n}d\Lambda $ with $\int_{A}f_{n}d\Lambda_{G}$ and $\int_{A}f_{n}d\Lambda_{M}$.
\begin{pro}\label{proposition-lambda-lambdag-lambdam} Let $f$ be integrable w.r.t.~$\Lambda_{G}$ and $\Lambda_{M}$. Then, for every $\theta\in\mathbb{R}$, we have
	\begin{equation}\label{cf-integral-QID}
	\hat{\mathcal{L}}\left(\int_{S}fd\Lambda\right)(\theta)=\frac{\hat{\mathcal{L}}\left(\int_{S}fd\Lambda_{G}\right)(\theta)}{\hat{\mathcal{L}}\left(\int_{S}fd\Lambda_{M}\right)(\theta)}.
	\end{equation}
\end{pro}
\begin{proof}
There are at least two ways to prove this result. The first is to use the properties of the Radon-Nikodym derivative and use point (iv) in Theorem \ref{theorem1}. The second is consider the sequence of simple functions $(f_{n})$ of the proof of Theorem \ref{theorem1} together with Proposition \ref{pr-same-as-in-the-paper} and Theorem \ref{theorem2}. Let's show the latter in detail.
\\ Consider $f$ to be $\Lambda^{(1)}$ and $\Lambda^{(2)}$ integrable. Then we know that $f$ is also $\Lambda$-integrable by Theorem \ref{theorem2}. Now consider the sequence $(f_{n})$ as in the proof of Theorem \ref{theorem1}, then we know that $\{\int_{A}f_{n}d\Lambda \}_{n=1}^{\infty}$, $\{\int_{A}f_{n}d\Lambda_{G} \}_{n=1}^{\infty}$ and $\{\int_{A}f_{n}d\Lambda_{M} \}_{n=1}^{\infty}$ converges in probability for every $A\in\sigma(\mathcal{S})$. Since for every $n\in\mathbb{N}$
\begin{equation*}
\hat{\mathcal{L}}\left(\int_{S}f_{n}d\Lambda\right)(\theta)=\frac{\hat{\mathcal{L}}\left(\int_{S}f_{n}d\Lambda_{G}\right)(\theta)}{\hat{\mathcal{L}}\left(\int_{S}f_{n}d\Lambda_{M}\right)(\theta)}
\end{equation*}
then we have that
\begin{equation*}
\hat{\mathcal{L}}\left(\int_{S}fd\Lambda\right)(\theta)=\lim\limits_{n\rightarrow\infty}\hat{\mathcal{L}}\left(\int_{S}f_{n}d\Lambda\right)(\theta)=\lim\limits_{n\rightarrow\infty}\frac{\hat{\mathcal{L}}\left(\int_{S}f_{n}d\Lambda_{G}\right)(\theta)}{\hat{\mathcal{L}}\left(\int_{S}f_{n}d\Lambda_{M}\right)(\theta)}=\frac{\hat{\mathcal{L}}\left(\int_{S}fd\Lambda_{G}\right)(\theta)}{\hat{\mathcal{L}}\left(\int_{S}fd\Lambda_{M}\right)(\theta)}.
\end{equation*}
\end{proof}
As we mentioned at the beginning of this section we have that  $\Lambda_{G}\sim(0,\nu_{1},G_{\cdot})$ and $\Lambda_{M}\sim(-\nu_{0},0,M_{\cdot})$, however all the results hold for more general ID r.m.. A sufficient condition for two generic ID r.m.~(call them $\Lambda_{J}$ and $\Lambda_{H}$) is that $\Lambda_{J}$ and $\Lambda_{H}$ generate $\Lambda$ and that their control measure, call them $\lambda_{X}$ and $\lambda_{Y}$, are s.t.~$\lambda_{X}(A)\leq C\lambda(A)$ and $\lambda_{X}(A)\leq C'\lambda(A)$ where $C,C'\in(0,\infty)$ for every $A\in\sigma(\mathcal{S})$. For the sake of brevity we left to the curious reader to check this statement. However, it would interesting to have the weakest possible conditions that two generating ID r.m.~must satisfy in order to satisfy all the results presented in this section, which mean satisfy (\ref{cf-integral-QID}).
\begin{open}
	Given a QID r.m., what is the largest class of generating ID r.m. that satisfy (\ref{cf-integral-QID})?
\end{open}
\noindent The answer to this question is also linked with Theorem \ref{theorem-representation}and Open Question \ref{open-representation}.
\subsection{The bounded case}\label{bounded}
In this section we will work with the following assumptions. We let $\mathcal{S}$ be a $\sigma$-algebra, $F(A,B)$ be a bimeasure on $\mathcal{S}\times\mathcal{B}(\mathbb{R})$ and \begin{equation}\label{final-assumption-bounded}
\sup\limits_{I}\sum_{i\in I}|F_{A_{i}}(B_{i})|<\infty,
\end{equation}
where the supremum is taken over all the finite sets $(A_{i},B_{i})_{i\in I}$ of elements of $\mathcal{S}\times\mathcal{B}(\mathbb{R})$ such that the rectangles $A_{i}\times B_{i}$ are disjoint. \\Let us introduce the set function $\nu(A):\mathcal{S}\mapsto [0,\infty)$ such that
\begin{equation*}
\nu(A):=\sup_{I_{A}}\sum_{i\in I_{A}}|F_{A_{i}}(B_{i})|,
\end{equation*}
where $I_{A}$ is defined as $I$ but with the constraint that $A_{i}\subset A$. To have a better idea of what kind of object $\nu$ is, compare it with the definition of total variation of a signed measure $(\ref{def-totalvariation})$.\\
We start with the following proposition on the control measure.
\begin{pro}\label{pro-le-misure-in-modo-generale}
	Let $\Lambda$ be a QID random measure. Let $\nu_{0}:\mathcal{S}\mapsto\mathbb{R}$ be a signed measure, $\nu_{1}:\mathcal{S}\mapsto\mathbb{R}$ be a measure, $F_{A}$ be a quasi-L\'{e}vy measure on $\mathbb{R}$ for every $A\in\mathcal{S}$, $\mathcal{S}\ni A\mapsto F_{A}(B)\in(-\infty,\infty)$ be a signed measure for every $B\in\mathcal{B}(\mathbb{R})$ s.t.~$(\nu_{0}(A), \nu_{1}(A),F_{A})$ is the characteristic triplet of $\Lambda(A)$, $\forall A\in\mathcal{S}$. Assume that $F$ satisfies $(\ref{final-assumption-bounded})$. Define 	
	\begin{equation}\label{lambda2}
	\lambda(A)=|\nu_{0}|(A)+\nu_{1}(A)+\nu(A).
	\end{equation}
	Then $\lambda:\mathcal{S}\mapsto[0,\infty)$ is a measure such that $\lambda(A_{n})\rightarrow0$ implies $\Lambda(A_{n})\stackrel{p}{\rightarrow}0$ for every $\{A_{n}\}\subset\mathcal{S}$.
\end{pro}
\begin{proof}
	From Theorem 4 in \cite{Horo} we have that $\nu(A)$ defines a measure. Then it is straightforward to see that $\lambda(A)$ is a measure. Now if $\nu(A_{n})\rightarrow0$ then $F_{A_{n}}^{+}(B)\rightarrow0$ and $F_{A_{n}}^{-}(B)\rightarrow0$ for every $B\in\mathcal{B}(\mathbb{R})$, hence by the L\'{e}vy continuity theorem we obtain the stated result. 
\end{proof}
Before presenting one of the main results of this section, which consists of an extension of Proposition 2.4 in \cite{RajRos}, we give a remark on its topology.
\begin{rem}
	Proposition 2.4 in \cite{RajRos} is stated for a \textnormal{standard Borel space}, where a measurable space $(X, \Gamma)$ is said to be \textnormal{standard Borel space} if there exists a metric on $X$ which makes it a complete separable metric space in such a way that $\Gamma$ is then the Borel $\sigma$-algebra $\mathcal{B}(X)$. However, the following result holds for more general topological spaces: the \textnormal{Lusin measurable spaces}. Lusin measurable spaces are measurable spaces isomorphic to a measurable space $(H,\mathcal{B}(H))$, where $H$ is homeomorphic to a Borel subset of a compact metrizable space. Notice that $H$ is usually called Lusin space or more correctly Lusin metrizable space. \\In this remark, we do not provide an extensive discussion (indeed, see \cite{Meyer}, \cite{Horo} and \cite{Horo2} for further details), we just mention that $H$ can be any Polish space (thus, including the real line).
\end{rem}
\begin{thm}\label{pr-monster}
	Let $(X,\mathcal{B})$ be a Lusin measurable space and let $(T,\mathcal{A})$ be an arbitrary measurable space. Let $Q_{0}(A,B)$ be a (possibly negative) function of $A\in\mathcal{A}$, $B\in\mathcal{B}$, satisfying:
	\\ \textnormal{(a)} for every $A\in\mathcal{A}$, $Q_{0}(A,\cdot)$ is a signed measure on $(X,\mathcal{B})$,
	\\ \textnormal{(b)} for every $B\in\mathcal{B}$, $Q_{0}(\cdot, B)$ is a signed measure on $(T,\mathcal{A})$,
	\\ \textnormal{(c)} $\sup\limits_{I}\sum_{i\in I}|Q_{0}(A_{i},B_{i})|<\infty$.
	\\Let $\nu(A):=\sup\limits_{I_{A}}\sum_{i\in I}|Q_{0}(A_{i},B_{i})|$. Then there exists a unique finite signed measure $Q$ on the product $\sigma$-algebra $\mathcal{A}\otimes\mathcal{B}$ s.t.
	\begin{equation}\label{formula-in-the-moster}
	Q(A\times B)=Q_{0}(A,B)=\int_{A}q(t,B)\nu(dt),
	\end{equation}
	with Jordan decomposition
	\begin{equation*}
	Q^{+}(C)=\int_{T}\int_{X}\textbf{1}_{C}(x,t)\tilde{q}_{+}(t,dx)\nu(dt)\quad\text{and}\quad Q^{-}(C)=\int_{T}\int_{X}\textbf{1}_{C}(x,t)\tilde{q}_{-}(t,dx)\nu(dt)
	\end{equation*}	
	for every $A\in\mathcal{A}$, $B\in\mathcal{B}$, $C\in \mathcal{A}\otimes \mathcal{B}$ where $q:T\times\mathcal{B}\rightarrow[-1,1]$ fulfils the following conditions:
	\\ \textnormal{(d)} for every $t$, $q(t,\cdot)$ is a signed measure on $\mathcal{B}$,
	\\ \textnormal{(e)} for every $B$, $q(\cdot,B)$ is $\mathcal{A}$-measurable,
	\\and where $\tilde{q}_{+}(t,\cdot)$ and $\tilde{q}_{-}(t,\cdot)$ are the Jordan decomposition of $q(t,\cdot)$.
	\\ Further, if $q_{1}(\cdot,\cdot)$ is some other function satisfying $(\ref{formula-in-the-moster})$, \textnormal{(d)} and \textnormal{(e)}, then off a set of $\nu$-measure zero, $q_{1}(t,\cdot)=q(t,\cdot)$.
\end{thm}
\begin{proof}
First, recall that, by Theorem 4 in \cite{Horo}, $\nu(\cdot)$ is a measure. Notice that by definition $\nu(\cdot)\gg Q_{0,+}(\cdot,B)$ and $\nu(\cdot)\gg Q_{0,-}(\cdot,B)$ for every $B\in\mathcal{B}$, where $Q_{0,+}(\cdot,B)$ and $Q_{0,-}(\cdot,B)$ are the Jordan decomposition of $Q_{0}(\cdot,B)$. Then, by the Radon-Nikodym theorem we have that $Q_{0,+}(A,B)=\int_{A}\tilde{q}_{+}(x,B)\nu(dx)$ and $Q_{0,-}(A,B)=\int_{A}\tilde{q}_{-}(x,B)\nu(dx)$. Hence, $Q_{0}(A,B)=\int_{A}q(x,B)\nu(dx)$ and $q(x,B)=\tilde{q}_{+}(x,B)-\tilde{q}_{-}(x,B)$ is $\nu$-a.s. uniquely defined (since $\tilde{q}_{+}(x,B)$ and $\tilde{q}_{-}(x,B)$ are the unique Radon-Nikodym derivatives) and, for every $B\in\mathcal{B}$, $q(\cdot,B)$ is $\mathcal{A}$-measurable.
\\Observe also that by definition $\nu(A)\geq Q_{0,+}(A,B)$ and $\nu(A)\geq Q_{0,+}(A,B)$ for every $A\in\mathcal{A}$ and every $B\in\mathcal{B}(\mathbb{R})$. Then, $\tilde{q}_{+}(x,B)\leq1$ and $\tilde{q}_{-}(x,B)\leq1$, $\nu$-a.e.. Further, since $Q(A,\cdot)$ is a (signed) measure for any $A\in\mathcal{A}$ then $q(x,\cdot)$ is also a (signed) measure with $|q(x,B)|\leq 1$ for every $B\in\mathcal{B}$. This is possible to see by adapting the arguments of the proof of Theorem 4 in \cite{Horo} when the function considered is the indicator function. Indeed, using their formalism $Q(A,\textbf{1}_{B}):=\int_{X}\textbf{1}_{B}(y)Q(A,dy)$, hence $Q(A,\textbf{1}_{B})=Q(A,B)$, and by the uniqueness of the Radon-Nikodym derivative $q(x,\textbf{1}_{B})=q(x,B)$, $\nu$-a.e..
\\Again by adapting the arguments of Theorem 4 in \cite{Horo} we get that there exists a finite signed measure $Q(A\times B)=\int_{A}q(x,B)\nu(dx)$. Moreover, by the uniqueness of $q$ we have that this signed measure is unique.

We prove now the stated Jordan decomposition of $Q$. From the proof of Theorem 4 in\cite{Horo} it is shown, using our notation, that $Q(C)=Q^{+}(C)-Q^{-}(C)$ for any $C\in\mathcal{B}\otimes\mathcal{A}$. However, it is not mentioned that $Q^{+}$ and $Q^{-}$ are the Jordan decomposition of $Q$, but this is indeed the case and we are going to prove it here.
\\If we apply the arguments of the proof of Proposition 6 in \cite{Horo} to indicator functions then it is possible to see that the decomposition of $q(\cdot,\textbf{1}_{B})$ into $q^{+}(\cdot,\textbf{1}_{B})$ and $q^{-}(\cdot,\textbf{1}_{B})$ in \cite{Horo} is actually the Jordan decomposition since $q^{+}(\cdot,\textbf{1}_{B})$ is defined as
\begin{equation*}
q^{+}(\cdot,\textbf{1}_{B}):=\textnormal{ess}\sup\limits_{0\leq g\leq \textbf{1}_{B}}q(\cdot,g)
\end{equation*}
where $g$ are measurable functions uniformly bounded on absolute values on the finite measure space $(T,\mathcal{A},\nu)$ s.t.~$g(y)\leq \textbf{1}_{B}(y)$, $\nu$-a.e.. Now, since we can approximate any such $g$ with indicator functions, namely the set of indicator functions are dense in the set of measurable functions with compact support (in this case $B$) and bounded by 1, then by standard density arguments we can substitute the $g$'s by indicator functions and using the fact that $q(x,\textbf{1}_{B})=q(x,B)$ $\nu$-almost everywhere we have
\begin{equation*}
q^{+}(\cdot,B):=\sup\limits_{K\subset B}q(\cdot,K)
\end{equation*}
which is the definition of the positive measure of the Jordan decomposition of the signed measure $q(x,\cdot)$. Further, since $q^{-}(\cdot,B)$ is defined as $q^{-}(\cdot,B):=q^{+}(\cdot,B)-q(\cdot,B)$ then it coincides with the negative measure of the Jordan decomposition.
\\ Since $Q(C)=\int_{A}\int_{B}\textbf{1}_{C}(x,t)\tilde{q}_{+}(t,dx)\nu(dt)-\int_{A}\int_{B}\textbf{1}_{C}(x,t)\tilde{q}_{-}(t,dx)\nu(dt)$, and $\int_{A}\int_{B}\textbf{1}_{C}(x,t)\tilde{q}_{+}(t,dx)\nu(dt)$ and $\int_{A}\int_{B}\textbf{1}_{C}(x,t)\tilde{q}_{-}(t,dx)\nu(dt)$ are mutually singular then we have the stated Jordan decomposition of $Q$.
\end{proof}
\begin{rem}
	Notice that it is not always true that $q^{+}=\tilde{q}_{+}$ since $Q_{0,+}(A,\cdot)$ and $Q_{0,-}(A,\cdot)$ are not necessarily measures (on the contrary $Q_{0,+}(\cdot,B)$ and $Q_{0,-}(\cdot,B)$ are measures since they are the Jordan decomposition of $Q_{0}(\cdot,B)$). The only thing we know is that $q_{+}(x,B)-q_{-}(x,B)=\tilde{q}_{+}(x,B)-\tilde{q}_{-}(x,B)$, where all the four functions are positive and with values less or equal than 1. In particular, the fact that $q^{+}(x,B)\leq 1$ and $q^{-}(x,B)\leq 1$ for $\nu$-almost every $x\in T$ and for every $B\in\mathcal{B}$ comes from the fact that $|q(x,B)|\leq 1$ for every $B\in\mathcal{B}$.
\end{rem}
\begin{lem}\label{l4-bounded}
	Let $F_{\cdot}$ be as in Proposition \ref{pro-le-misure-in-modo-generale}. Then there exists a unique finite measure $F$ on $\mathcal{S}\otimes\mathcal{B}(\mathbb{R})$ such that
	\begin{equation*}
	F(A\times B)=F_{A}(B),\quad\text{for all $A\in\mathcal{S}$, $B\in\mathcal{B}(\mathbb{R})$}.
	\end{equation*}
	Moreover, there exists a function $\rho:S\times\mathcal{B}(\mathbb{R})\mapsto[-\infty,\infty]$ such that
	\\ \textnormal{(i)} $\rho(s,\cdot)$ is a quasi-L\'{e}vy type measure on $\mathcal{B}(\mathbb{R})$, for every $s\in S$,
	\\ \textnormal{(ii)} $\rho(\cdot,B)$ is a Borel measurable function, for every $B\in\mathcal{B}(\mathbb{R})$,
	\\ \textnormal{(iii)} $\int_{S\times\mathbb{R}}h(s,x)F(ds,dx)=\int_{S}\int_{\mathbb{R}}h(s,x)\rho(s,dx)\lambda(ds)$, for every $\mathcal{S}\otimes\mathcal{B}(\mathbb{R})$-measurable function $h:S\times\mathbb{R}\mapsto[0,\infty]$. This equality can be extended to real and complex-valued functions $h$.
\end{lem}
\begin{proof}
	Using Theorem \ref{pr-monster}, we have that
	$F_{A}(B)=F(A\times B)=\int_{A}q(t,B)\nu(dt)$ where $q$ satisfies point (d) and (e) in that theorem. Since $\lambda\ll\nu$, then defining
	\begin{equation*}
\rho_{+}(s,dx):=\frac{d\nu}{d\lambda}(s)\tilde{q}_{+}(s,dx),\quad \rho_{-}(s,dx):=\frac{d\nu}{d\lambda}(s)\tilde{q}_{-}(s,dx)
	\end{equation*}
	\begin{equation*}
	\textnormal{and}\quad \rho(s,dx):=\rho_{+}(s,dx)-\rho_{-}(s,dx)
	\end{equation*}
	we have that (ii) is satisfied and that
	\begin{equation*}
	\int_{\mathbb{R}}(1\wedge x^{2})|\rho|(s,dx)=\frac{d\nu}{d\lambda}(s)\int_{\mathbb{R}}(1\wedge x^{2})\tilde{q}_{+}(s,dx)+\frac{d\nu}{d\lambda}(s)\int_{\mathbb{R}}(1\wedge x^{2})\tilde{q}_{-}(s,dx)
	\end{equation*}
		\begin{equation*}
		\leq \frac{d\nu}{d\lambda}(s)\int_{\mathbb{R}}\tilde{q}_{+}(s,dx)+\frac{d\nu}{d\lambda}(s)\int_{\mathbb{R}}\tilde{q}_{-}(s,dx)
	\leq\frac{d\nu}{d\lambda}(s)+\frac{d\nu}{d\lambda}(s)\leq 2,
	\end{equation*}
	where the last inequality comes from the fact that we can always assume that $\frac{d\nu}{d\lambda}(s)\leq 1$ for all $s$ (and the same for $\nu$). This proves (i).
	\\ Therefore, we have immediately (iii) since
	\begin{equation*}
	\int_{S}\int_{\mathbb{R}}\textbf{1}_{C}(s,x)\rho(s,dx)\lambda(ds)=\int_{S}\int_{\mathbb{R}}\textbf{1}_{C}(s,x)\tilde{q}(s,dx)\nu(ds)=F(C).
	\end{equation*}
	The extension to real and complex integrand follows by standard arguments (\textit{e.g.}~see \cite{Horo}). 
\end{proof}
\begin{co}\label{co2-bounded}
	Under the setting of Proposition \ref{pro-le-misure-in-modo-generale}, we have $F^{+}(C)=\int_{S}\int_{\mathbb{R}}\textbf{1}_{C}(s,x)\rho^{+}(s,dx)\lambda(ds)$ and $F^{-}(C)=\int_{S}\int_{\mathbb{R}}\textbf{1}_{C}(s,x)\rho^{-}(s,dx)\lambda(ds)$.
\end{co}
\begin{proof}
First, notice that $\int_{S}\int_{\mathbb{R}}\textbf{1}_{C}(s,x)\rho^{+}(s,dx)\lambda(ds)$ and $\int_{S}\int_{\mathbb{R}}\textbf{1}_{C}(s,x)\rho^{-}(s,dx)\lambda(ds)$ are measures on $\mathcal{S}\otimes\mathcal{B}(\mathbb{R})$ and that $F(C)=\int_{S}\int_{\mathbb{R}}\textbf{1}_{C}(s,x)\rho^{+}(s,dx)\lambda(ds)-\int_{S}\int_{\mathbb{R}}\textbf{1}_{C}(s,x)\rho^{-}(s,dx)\lambda(ds)$.
Finally, let $E^{+}_{s}$ and $E^{+}_{s}$ the Hahn decomposition of $\rho(s,\cdot)$ (which is the same as the one of $q(s,\cdot)$). Define $C^{+}:=\{(s,x)\in S\times\mathbb{R}:\,\, x\in E^{+}_{s} \}$ and $C^{-}:=\{(s,x)\in S\times\mathbb{R}:\,\, x\in E^{-}_{s} \}$. It is possible to see that $C^{+}$ and $C^{-}$ form an essential decomposition of $S\times\mathbb{R}$ and so the measures $\int_{S}\int_{\mathbb{R}}\textbf{1}_{C}(s,x)\rho^{+}(s,dx)\lambda(ds)$ and $\int_{S}\int_{\mathbb{R}}\textbf{1}_{C}(s,x)\rho^{-}(s,dx)\lambda(ds)$ are singular, thus concluding the proof.
\\ Notice that an alternative proof follows directly from Theorem \ref{pr-monster} and Lemma \ref{l4-bounded}.
\end{proof}
Under the same setting as in Proposition \ref{pro-le-misure-in-modo-generale}, we obtain the following proposition.
\begin{pro}\label{pr-3-bounded}
	Under the setting of Proposition \ref{pro-le-misure-in-modo-generale}, the c.f.~of $\Lambda(A)$ can be written in the form:
	\begin{equation*}
	\mathbb{E}(e^{i\theta\Lambda(A)})=\exp\left(\int_{A}K(\theta,s)\lambda(ds)\right),\quad \theta\in\mathbb{R},A\in\mathcal{S},
	\end{equation*}
	where
	\begin{equation*}
	K(\theta,s)=i\theta a(s)-\frac{\theta^{2}}{2}\sigma^{2}(s)+\int_{\mathbb{R}}e^{i\theta x}-1-i\theta\tau(x)\rho(s,dx),
	\end{equation*}
	$a(s)=\frac{d\nu_{0}}{d\lambda}(s)$, $\sigma^{2}(s)=\frac{d\nu_{1}}{d\lambda}(s)$ and $\rho$ is given by Lemma \ref{l4-bounded}, and $\exp(K(\theta,s))$ is the characteristic function of a QID random variable if it exists. Moreover, we have
	\begin{equation*}
	|a(s)|+ \sigma^{2}(s)+\frac{d\nu}{d\lambda}(s)=1\quad \text{$\lambda$-a.e.}.
	\end{equation*}
\end{pro}
\begin{proof}
	The first statement follows from the L\'{e}vy-Khintchine formulation $(\ref{cf})$ and Lemma \ref{l4-bounded}.
	
	The second statement follows from the fact that for every $A\in\mathcal{S}$, we have
	\begin{equation*}
	\int_{A}\left(|a(s)|+\sigma^{2}(s)+\frac{d\nu}{d\lambda}(s)\right)\lambda(ds)=|\nu_{0}|(A)+\nu_{1}(A)+\nu(A)=\lambda(A)=\int_{A}d\lambda(ds).
	\end{equation*}
\end{proof}
\begin{pro} Under the setting of Proposition \ref{pro-le-misure-in-modo-generale}, if $f$ is $\Lambda$-integrable, then $\int_{S}|K(tf(s),s)|\lambda(ds)<\infty$, where $K$ is given in Proposition \ref{pr-3-bounded}, and
	\begin{equation*}
	\hat{\mathcal{L}}\left(\int_{S}fd\Lambda\right)(\theta)=\exp\left(\int_{S} K(\theta f(s),s)\lambda(ds)\right),\quad \theta\in\mathbb{R}.
	\end{equation*}
\end{pro}
\begin{proof}
	The statement follows from the same arguments used in the proof Proposition \ref{pr-same-as-in-the-paper}.
\end{proof}
Differently to Theorem \ref{theorem1} in the following we have directly and if and only if result for the integrability conditions of $\int_{S}fd\Lambda$.
\begin{thm}\label{theorem-bounded}
	Let $f:S\rightarrow\mathbb{R}$ be a $\mathcal{S}$-measurable function and consider the same setting as in Proposition \ref{pro-le-misure-in-modo-generale}. Then $f$ is $\Lambda$-integrable if and only if the following three conditions hold:
	\\\textnormal{(i)} $\int_{S}|U(f(s),s)|\lambda(ds)<\infty$,
	\\\textnormal{(ii)} $\int_{S}|f(s)|^{2}\sigma^{2}(s)\lambda(ds)<\infty$,
	\\\textnormal{(iii)} $\int_{S}V_{0}(f(s),s)\lambda(ds)<\infty$,
	\begin{equation*}
	\textnormal{where}\quad U(u,s)=ua(s)+\int_{\mathbb{R}}\tau(xu)-u\tau(x)\rho(s,dx),\quad V_{0}(u,s)=\int_{\mathbb{R}}(1\wedge |xu|^{2})|\rho|(s,dx),
	\end{equation*}
	Further, if $f$ is $\Lambda$-integrable, then the c.f.~of $\int_{S}fd\Lambda$ can be written as
	\\\textnormal{(iv)} $\hat{\mathcal{L}}\left(\int_{S}fd\Lambda\right)(\theta)=\exp\left(i\theta a_{f}-\frac{1}{2}\theta^{2}\sigma_{f}^{2}+\int_{\mathbb{R}}e^{i\theta x}-1-i\theta\tau(x)F_{f}(dx) \right)$,
	where
	\begin{equation*}
	a_{f}=\int_{S}U(f(s),s)\lambda(ds),\quad \sigma_{f}^{2}=\int_{S}|f(s)|^{2}\sigma^{2}(s)\lambda(ds),\quad\textnormal{and}
	\end{equation*}
	\begin{equation*}
	F_{f}(B)=F(\{(s,x)\in S\times\mathbb{R}:\,\, f(s)x\in B\setminus\{0 \} \}),\quad B\in\mathcal{B}(\mathbb{R}).
	\end{equation*}
\end{thm}
\begin{proof}
	$\Leftarrow$:  Under the framework of this section we have that $F$ is a finite signed measure (see Lemma \ref{l4-bounded}) and so $\int_{\mathbb{R}}(1\wedge y^{2})|\tilde{F}_{f}|(dy)<\infty$. Moreover, since $F$ is a finite signed measure then by Corollary \ref{co2-bounded} we have
	\begin{equation*}
	\int_{S}\int_{\mathbb{R}}(1\wedge |f(s)x|^{2})|\rho|(s,dx)\lambda(ds)\leq 	\int_{S}\int_{\mathbb{R}}|\rho|(s,dx)\lambda(ds)<\infty.
	\end{equation*}	
	We conclude this part of the proof with the same arguments used in the proof of Theorem \ref{theorem2}.
	\\$\Rightarrow$: It follows from similar arguments as the ones used in the proof of Theorem \ref{theorem1}.
\end{proof}
\subsection{The general case}\label{Sec-generalcase}
In this section we extend the results presented in the previous section. We will work with the assumption that $\mathcal{S}$ is a $\sigma$-algebra and that $F(A,\cdot)$ is a signed measure for every $A\in\mathcal{S}$ and $F(\cdot, B)$ is a signed measure for every $B\in\mathcal{B}(\mathbb{R})$ s.t.~$0\notin \overline{B}$. We define for every $A\in\mathcal{S}$ and $B\in\mathcal{B}(\mathbb{R})$
\begin{equation*}
J(A,B):=\int_{B}(1\wedge x^{2})F_{A}(dx)
\end{equation*}
 and notice that it is a signed bimeasure on $\mathcal{S}\times\mathcal{B}(\mathbb{R})$. We will assume that 
  \begin{equation}\label{final-assumption}
  \sup_{I}\sum_{i\in I}|J(A_{i},B_{i})|<\infty.
  \end{equation}
  Informally, the above assumption is weaker than assuming the ``total variation" of the signed bimeasure $J$ to be finite. 
  \\Let us introduce the set function $\xi(A):\mathcal{S}\mapsto [0,\infty)$ such that
 \begin{equation*}
 \xi(A):=\sup_{I_{A}}\sum_{i\in I_{A}}|J(A_{i},B_{i})|.
 \end{equation*}
 Notice that $\xi$ is indeed a measure (see the proof of Theorem 4 in \cite{Horo}). We point out that it might be possible to assume a weaker condition than $(\ref{final-assumption})$, which is assuming that $\xi(A)$ is a measure. This is possible by looking at the proof of Theorem 4 in \cite{Horo}. However, for the sake of clarity we keep the assumption $(\ref{final-assumption})$. Now, let $E^{+}_{A}$ and $E^{-}_{A}$ the Hahn decomposition of $\mathbb{R}$ under the signed measure $F_{A}$. Observe that
 \begin{equation}\label{xi}
 \xi(A)=\sup\limits_{I_{A}}\sum_{i\in I_{A}}|J(A_{i},B_{i})|\geq \int_{E_{A}^{+}}(1\wedge x^{2})F_{A}(dx)-\int_{E_{A}^{-}}(1\wedge x^{2})F_{A}(dx)= \int_{\mathbb{R}}(1\wedge x^{2})|F_{A}|(dx).
 \end{equation}
  Therefore, since $\xi(S)$ is finite by assumption we have that $\int_{\mathbb{R}}(1\wedge x^{2})|F_{A}|(dx)<\infty$. 
  \\The name of this section comes from the fact that if $\mathcal{S}$ is a $\sigma$-algebra then the assumptions on $F$ presented in the previous sections (Sections \ref{Section-geerating two ID} and \ref{bounded}), are stricter then the ones of this section. Indeed, for the bounded case, notice that if $F(A,B)$ is a bimeasure on $\mathcal{S}\times\mathcal{B}(\mathbb{R})$ and $\sup\limits_{I}\sum_{i\in I}|F_{A_{i}}(B_{i})|<\infty$, then $F$ immediately satisfies the assumptions of this section. For the case of the two generating ID r.m., observe that
  \begin{equation*}
  \int_{B}(1\wedge x^{2})F_{A}(dx)=\int_{B}(1\wedge x^{2})G_{A}(dx)-\int_{B}(1\wedge x^{2})M_{A}(dx)
  \end{equation*}
  and so using the fact $S\in\mathcal{S}$ (since $\mathcal{S}$ is a $\sigma$-algebra) we have that 
  \begin{equation*}
\xi(S)=\sup\limits_{I}\sum_{i\in I}|J(A_{i},B_{i})|\leq\sup\limits_{I}\sum_{i\in I}\int_{B_{i}}(1\wedge x^{2})G_{A_{i}}(dx)+\int_{B_{i}}(1\wedge x^{2})M_{A_{i}}(dx)
  \end{equation*}
  \begin{equation*}
  \leq \int_{\mathbb{R}}(1\wedge x^{2})G_{S}(dx)+\int_{\mathbb{R}}(1\wedge x^{2})M_{S}(dx)<\infty
  \end{equation*}
   We are now ready to present the first results of this section.
\begin{pro}\label{pro-misure-generale}
	Let $\Lambda$ be a QID random measure. Let $\nu_{0}:\mathcal{S}\mapsto\mathbb{R}$ be a signed measure, $\nu_{1}:\mathcal{S}\mapsto\mathbb{R}$ be a measure, $F_{A}$ be a quasi-L\'{e}vy measure on $\mathbb{R}$ for every $A\in\mathcal{S}$, $\mathcal{S}\ni A\mapsto F_{A}(B)\in(-\infty,\infty)$ be a signed measure for every $B\in\mathcal{B}(\mathbb{R})$ s.t.~$0\notin \overline{B}$ and s.t.~$(\nu_{0}(A), \nu_{1}(A),F_{A})$ is the characteristic triplet of $\Lambda(A)$, $\forall A\in\mathcal{S}$. Assume that $F$ satisfies $(\ref{final-assumption})$. Define 	
	\begin{equation}\label{lambda3}
	\lambda(A)=|\nu_{0}|(A)+\nu_{1}(A)+\xi(A).
	\end{equation}
	Then $\lambda:\mathcal{S}\mapsto[0,\infty)$ is a measure s.t.~$\lambda(A_{n})\rightarrow0$ implies $\Lambda(A_{n})\stackrel{p}{\rightarrow}0$ for every $\{A_{n}\}\subset\mathcal{S}$.
\end{pro}
\begin{proof}
	Since by Theorem 4 in \cite{Horo} we know that $\xi$ is a measure, it is straightforward to see that $\lambda(A)$ is a measure. Now, let $\lambda(A_{n})\rightarrow0$ then we have that $|\nu_{0}|$, $\nu_{1}$ and $\xi$ goes to zero. Then since $\xi(A_{n})\rightarrow0$ implies that $\int_{\mathbb{R}}(1\wedge x^{2})|F_{A_{n}}|(dx)\rightarrow0$, then we have that $\Lambda(A_{n})\stackrel{p}{\rightarrow}0$ for every $\{A_{n}\}\subset\mathcal{S}$.
\end{proof}
\begin{lem}\label{l4}
	Let $F_{\cdot}$ be as in Proposition \ref{pro-misure-generale}. There exists a function $\rho:S\times\mathcal{B}(\mathbb{R})\mapsto[-\infty,\infty]$ such that
	\\ \textnormal{(i)} $\rho(s,\cdot)$ is a quasi-L\'{e}vy type measure on $\mathcal{B}(\mathbb{R})$, for every $s\in S$,
	\\ \textnormal{(ii)} the positive and negative part of $\rho(\cdot,B)$, denoted by $\rho^{+}(\cdot,B)$ and $\rho^{-}(\cdot,B)$ are Borel measurable functions, for every $B\in\mathcal{B}(\mathbb{R})$,
	\\ Moreover, there exist two unique $\sigma$-finite measures $F^{+}$ and $F^{-}$ on $\mathcal{S}\otimes\mathcal{B}(\mathbb{R})$ s.t.~$\int_{S\times\mathbb{R}}h(s,x)F^{+}(ds,dx)=\int_{S}\int_{\mathbb{R}}h(s,x)\rho^{-}(s,dx)\lambda(ds)$, for every $\mathcal{S}\otimes\mathcal{B}(\mathbb{R})$-measurable function $h:S\times\mathbb{R}\mapsto[0,\infty]$, and the same holds for $F^{-}$. This equality can be extended to real and complex-valued functions $h$. Finally, for every $A\in\mathcal{S}$ and for every $\mathcal{B}(\mathbb{R})$-measurable real function $g$ s.t.~$\int_{A}\int_{\mathbb{R}}g(x)|\rho|(s,dx)\lambda(ds)<\infty$, we have that 
	\begin{equation*}
	\int_{\mathbb{R}}g(x)F_{A}(dx)=\int_{A}\int_{\mathbb{R}}g(x)\rho(s,dx)\lambda(ds),
	\end{equation*}
	and for every $B\in\mathcal{B}(\mathbb{R})$ s.t.~$0\notin \overline{B}$,
		\begin{equation*}
		\tilde{F}^{+}(A,B)\geq F_{A}^{+}(B)\quad\text{and}\quad \tilde{F}^{-}(A,B)\geq F_{A}^{-}(B).
		\end{equation*}
\end{lem}
\begin{proof}
	First, notice that $J(A,B)$ satisfies the assumptions of Theorem \ref{pr-monster} with $(T,\mathcal{A})=(S,\mathcal{S})$ and $(X,\mathcal{B})=(\mathbb{R},\mathcal{B}(\mathbb{R}))$.
	Therefore, there exists a finite signed measure $Q$ on the product $\sigma$-algebra $\mathcal{S}\otimes\mathcal{B}(\mathbb{R})$ such that
	\begin{equation*}
	Q(A\times B)=J(A,B)=\int_{A}q(s,B)\xi(ds)=\int_{A}\tilde{q}^{+}(s,B)\xi(ds)-\int_{A}\tilde{q}^{-}(s,B)\xi(ds),
	\end{equation*}
	where $q$ satisfies $(d)$ and $(e)$ of Proposition 2.4 in \cite{RajRos}.	Since $\lambda\gg\xi$, define
		\begin{equation*}
		\rho_{+}(s,dx):=\frac{d\xi}{d\lambda}(s)(1\wedge x^{2})^{-1}\tilde{q}_{+}(s,dx),\quad\textnormal{and}\quad \rho_{-}(s,dx):=\frac{d\xi}{d\lambda}(s)(1\wedge x^{2})^{-1}\tilde{q}_{-}(s,dx).
		\end{equation*}
		Notice that 
		\begin{equation*}
		\int_{\mathbb{R}}(1\wedge x^{2})\rho^{+}(s,dx)=\frac{d\xi}{d\lambda}(s)\int_{\mathbb{R}}\tilde{q}^{+}(s,dx)\leq 1,
		\end{equation*}
		where the last inequality comes from the fact that we can always assume that $\frac{d\xi}{d\lambda}(s)\leq 1$ for all $s\in S$. Hence, $\rho^{+}(s,\cdot)$ is a L\'{e}vy measure on $\mathbb{R}$ for all $s\in S$. The same is true for $\rho^{-}(s,\cdot)$. Further, let
		\begin{equation*}
		\rho(s,B):=\rho_{+}(s,B)-\rho_{-}(s,B)\quad\text{for all $s\in S$, $B\in\mathcal{B}(\mathbb{R})$ s.t.~$0\notin \overline{B}$}.
		\end{equation*}
	Then $\rho(s,\cdot)$ is a quasi-L\'{e}vy type measure, thus obtaining (i). Using the fact that $\rho_{+}(s,B)$ and $\rho_{-}(s,B)$ are mutually singular for every $B\in\mathcal{B}(\mathbb{R})$ s.t.~$\overline{B}$, we obtain (ii). Now, let
	\begin{equation}\label{F(C)}
	\tilde{F}^{+}(C)=\int_{S}\int_{\mathbb{R}}\textbf{1}_{C}(s,x)\rho^{+}(s,dx)\lambda(ds),
	\end{equation}
	where $C\in\mathcal{S}\otimes\mathcal{B}(\mathbb{R})$, then $\tilde{F}^{+}$ is a well defined signed measure that satisfies, for every $A\in\mathcal{S}$ and $B\in\mathcal{B}(\mathbb{R})$,
	\begin{equation*}
	\tilde{F}^{+}(A\times B)=\int_{A}\int_{B}\rho^{+}(s,dx)\lambda(ds)	=\int_{A}\int_{B}(1\wedge x^{2})^{-1}q^{+}(s,dx)\xi(ds)
	\end{equation*}
	\begin{equation*}
	=\int_{A\times B}(1\wedge x^{2})^{-1}Q^{+}(ds,dx)	\geq\int_{B}(1\wedge x^{2})^{-1}J_{+}(A,dx)=\int_{B}F^{+}_{A}(dx)=F^{+}_{A}(B),
	\end{equation*}
	where we recall that the notation $M_{+}$ and $M_{-}$ for a bimeasure stands for the Jordan decomposition of $M(A,B)$ for fixed $A$. Finally, notice that for any $\mathcal{B}(\mathbb{R})$-measurable real function $g$ s.t.~$\int_{A}\int_{\mathbb{R}}g(x)|\rho|(s,dx)\lambda(ds)<\infty$ we have
	\begin{equation*}
\int_{A}\int_{\mathbb{R}}g(x)\rho(s,dx)\lambda(ds)
	\end{equation*}
		\begin{equation*}
		=\int_{A}\int_{\mathbb{R}}g(x)\rho^{+}(s,dx)\lambda(ds)-\int_{A}\int_{\mathbb{R}}g(x)\rho^{-}(s,dx)\lambda(ds)
		\end{equation*}
	\begin{equation*}
	=\int_{A\times \mathbb{R}}g(x)(1\wedge x^{2})^{-1}Q(ds,dx)=\int_{\mathbb{R}}g(x)(1\wedge x^{2})^{-1}J(A,dx)=\int_{\mathbb{R}}g(x)F_{A}(dx).
	\end{equation*}
\end{proof}
Using the above results, we obtain the following proposition.
\begin{pro}\label{pr-3}
	Under the setting of Proposition \ref{pro-misure-generale}, the c.f.~of $\Lambda(A)$ can be written in the form:
	\begin{equation*}
	\mathbb{E}(e^{i\theta\Lambda(A)})=\exp\left(\int_{A}K(\theta,s)\lambda(ds)\right),\quad \theta\in\mathbb{R},A\in\mathcal{S},
	\end{equation*}
	where
	\begin{equation*}
	K(\theta,s)=i\theta a(s)-\frac{\theta^{2}}{2}\sigma^{2}(s)+\int_{\mathbb{R}}e^{i\theta x}-1-i\theta\tau(x)\rho(s,dx),
	\end{equation*}
	$a(s)=\frac{d\nu_{0}}{d\lambda}(s)$, $\sigma^{2}(s)=\frac{d\nu_{1}}{d\lambda}(s)$ and $\rho$ is given by Lemma \ref{l4}, and $\exp(K(\theta,s))$ is the characteristic function of a QID random variable if it exists. Moreover, we have
	\begin{equation*}
	|a(s)|+ \sigma^{2}(s)+\frac{d\xi}{d\lambda}(s)=1,\quad\quad \text{$\lambda$-a.e.}.
	\end{equation*}
\end{pro}
\begin{proof}
	The first statement follows from the L\'{e}vy-Khintchine formulation $(\ref{cf})$ and Lemma \ref{l4}.
	
	The second statement follows from the fact that for every $A\in\mathcal{S}$, we have
	\begin{equation*}
	\int_{A}\left(|a(s)|+\sigma^{2}(s)+\frac{d\xi}{d\lambda}(s)\right)\lambda(ds)=|\nu_{0}|(A)+\nu_{1}(A)+\xi(A)=\lambda(A)=\int_{A}d\lambda(ds).
	\end{equation*}
\end{proof}
\begin{pro}
	Under the setting of Proposition \ref{pro-misure-generale}, if $f$ is $\Lambda$-integrable, then $\int_{S}|K(tf(s),s)|\lambda(ds)<\infty$, where $K$ is given in Proposition \ref{pr-3}, and
	\begin{equation*}
	\hat{\mathcal{L}}\left(\int_{S}fd\Lambda\right)(\theta)=\exp\left(\int_{S} K(\theta f(s),s)\lambda(ds)\right),\quad \theta\in\mathbb{R}.
	\end{equation*}
\end{pro}
\begin{proof}
	The statement follows from the same arguments used in the proof of Proposition \ref{pr-same-as-in-the-paper}.
\end{proof}
We state now the main theorem of this subsection on the integrability conditions of $\int_{S}fd\Lambda$.
\begin{thm}\label{theorem-general}
	Let $f:S\rightarrow\mathbb{R}$ be a $\mathcal{S}$-measurable function and consider the setting of Proposition \ref{pro-misure-generale}. Then $f$ is $\Lambda$-integrable if the following three conditions hold:
	\\\textnormal{(i)} $\int_{S}|U(f(s),s)|\lambda(ds)<\infty$,
	\\\textnormal{(ii)} $\int_{S}|f(s)|^{2}\sigma^{2}(s)\lambda(ds)<\infty$,
	\\\textnormal{(iii)} $\int_{S}V_{0}(f(s),s)\lambda(ds)<\infty$,
	\begin{equation*}
	\textnormal{where}\quad U(u,s)=ua(s)+\int_{\mathbb{R}}\tau(xu)-u\tau(x)\rho(s,dx),\quad V_{0}(u,s)=\int_{\mathbb{R}}(1\wedge |xu|^{2})|\rho|(s,dx).
	\end{equation*}
	Further, the c.f.~of $\int_{S}fd\Lambda$ can be written as
	\\\textnormal{(iv)} $\hat{\mathcal{L}}\left(\int_{S}fd\Lambda\right)(\theta)=\exp\left(i\theta a_{f}-\frac{1}{2}\theta^{2}\sigma_{f}^{2}+\int_{\mathbb{R}}e^{i\theta x}-1-i\theta\tau(x)F_{f}(dx) \right)$,
	where
	\begin{equation*}
	a_{f}=\int_{S}U(f(s),s)\lambda(ds),\quad \sigma_{f}^{2}=\int_{S}|f(s)|^{2}\sigma^{2}(s)\lambda(ds),\quad\textnormal{and}
	\end{equation*}
	$F_{f}(B)$ is the unique quasi-L\'{e}vy measure determined by the difference of the L\'{e}vy measures $\tilde{F}^{+}_{f}$ and $\tilde{F}^{-}_{f}$, which are defined as: for every $B\in\mathcal{B}(\mathbb{R})$
	\begin{equation*}
	\tilde{F}^{+}_{f}(B)=\tilde{F}^{+}(\{(s,x)\in S\times\mathbb{R}:\, f(s)x\in B\setminus\{0 \} \})\,\,\,\text{and}\,\,\, \tilde{F}^{-}_{f}(B)=\tilde{F}^{-}(\{(s,x)\in S\times\mathbb{R}:\, f(s)x\in B\setminus\{0 \} \}).
	\end{equation*}
\end{thm}
\begin{proof}
It follows from similar arguments as the ones used in the proof of Theorem \ref{theorem1}.
\end{proof}
\subsection{The symmetric case}
We conclude this section with a discussion on the symmetric case, namely the case where $\bar{\Lambda}(A)=\Lambda(A)-\Lambda'(A)$ and $\Lambda'(A)$ is an independent copy of $\Lambda(A)$. We start with the following general lemma.
\begin{lem}\label{l2}
	Let $\mu$ be a QID distribution. Then, the symmetrisation and the dual of $\mu$ are QID distributions.
\end{lem}
\begin{proof}
	The dual is straightforward. Regarding the symmetrisation, we have $|\mu|^{2}=\mu*\tilde{\mu}$, where $\tilde{\mu}$ denotes the dual of $\mu$. Since the class of QID distributions is closed under convolution (see Remark 2.6 of \cite{LPS}), then $|\mu|^{2}$ is QID.
\end{proof}
Then, we have the following result in the framework of Section \ref{bounded}.
\begin{pro} Let $\bar{\Lambda}(A)=\Lambda(A)-\Lambda'(A)$ where $ \Lambda'(A)$ is an independent copy of $\Lambda(A)$ for every $A\in\mathcal{S}$. Then, $\bar{\Lambda}$ is a QID random measure.
Further, under the setting of Proposition \ref{pro-le-misure-in-modo-generale}, for an arbitrary function $f:S\mapsto \mathbb{R}$, $f$ is $\bar{\Lambda}$-integrable if and only if it is $\Lambda$-integrable.
\end{pro}
\begin{proof}
	By Lemma \ref{l2} we have that $\bar{\Lambda}(A)$ is a QID r.v.~fore every $A\in\mathcal{A}$. Moreover, using the notations used before and Proposition 2.5 in \cite{Sato}, we have that
\begin{equation*}
\hat{\mathcal{L}}\left(\bar{\Lambda}(A)\right)(\theta)=|\hat{\mathcal{L}}\left(\Lambda(A)\right)(\theta)|^{2}=\exp\left(-\theta^{2}\nu_{1}(A)+2\int_{\mathbb{R}}\cos(\theta x)-1 F_{A}(dx)\right).
\end{equation*}
\begin{equation*}
=\exp\left(\int_{A}\left(-\theta^{2}\sigma^{2}(s)+2\int_{\mathbb{R}}\cos(\theta x)-1\rho(s,dx)\right)\lambda(ds) \right).
\end{equation*}
Then, by applying Theorem \ref{theorem-bounded} we obtain the stated result.
\end{proof}
\begin{rem}
	An ``if" result for the framework of Section \ref{Section-geerating two ID} is also possible if $f$ is $\Lambda_{G}$ and $\Lambda_{M}$ integrable, where $\Lambda_{G}$ and $\Lambda_{M}$ are the ID r.m.~that generate $\Lambda$. 
\end{rem}
\section{Continuity of the stochastic integral mapping}\label{Ch-Continuity}
In this section we are going to explore the set of $\Lambda$-integrable functions and show a continuity property of the linear the stochastic integral mapping $f\rightarrow\int_{S}fd\Lambda$ from this space into the $L_{p}$ space (more precisely to $L_{p}(\Omega,\mathbb{P})$). In particular, the space of integrable function is a subset of the corresponding Musielak-Orlicz modular space defined in Chapter III of \cite{RajRos}.
We begin with some preliminaries. Let $q\in[0,\infty)$. Consider the following condition:
\begin{equation*}
\mathbb{E}[|\Lambda(A)|^{q}]<\infty,\quad\quad\textnormal{for all $A\in\mathcal{S}$}.
\end{equation*}
Observe that for $q=0$ every $\Lambda$ satisfies this condition. Throughout this section, we shall assume that the above condition is satisfied. Further, we assume that for all $s\in S$
\begin{equation}\label{ass}
\int_{|x|>1}|x|^{q}\rho^{+}(s,dx)<\infty.
\end{equation}
From the arguments in the previous section we have that
\begin{equation*}
\int_{A}\int_{|x|>1}|x|^{q}\rho^{+}(s,dx)\lambda(ds)\geq \int_{|x|>1}|x|^{q}F_{A}^{+}(dx)
\end{equation*}
and then by Theorem 6.2 point (a) in \cite{LPS} the assumption (\ref{ass}) implies that 
\begin{equation*}
\int_{|x|>1}|x|^{q}F_{A}^{-}(dx)<\infty\quad\textnormal{and}\quad\int_{\mathbb{R}}|x|^{q}\mathcal{L}(\Lambda(A))(dx)<\infty.
\end{equation*}
Observe that while Theorem 6.2 point (a) in \cite{LPS} is stated for the centering function $\tau(x)=x\textbf{1}_{|x|\leq 1}$, the result holds for any centering function since the proof is based on results on ID distributions with no restrictions on the choice of the centering function (see Theorem 25.3 of \cite{Sato}).\\
Now, define, for $0\leq p\leq q$, $u\in\mathbb{R}$ and $s\in S$,
\begin{equation}\label{Phi}
\Phi_{p}(u,s)=U^{*}(u,s)+u^{2}\sigma^{2}(s)+V_{p}(u,s),
\end{equation}
where
\begin{equation*}
U^{*}(u,s)=\sup\limits_{|c|\leq 1}|U(cu,s)|\quad\textnormal{and}\quad V_{p}(u,s)=\int_{\mathbb{R}}|xu|^{p}\textbf{1}_{|xu|>1}(x) +|xu|^{2}\textbf{1}_{|xu|\leq 1}(x)|\rho|(s,dx).
\end{equation*}
In the following lemma, which is an equivalent of Lemma 3.1 in \cite{RajRos} in our framework, we abuse the notation and consider $\lambda$ to take all the different formulations it has taken in this work, namely $(\ref{lambda1})$, $(\ref{lambda2})$ and $(\ref{lambda3})$.
\begin{lem}
	The following are satisfied:
	\\\textnormal{(i)} for every $s\in S$, $\Phi_{p}(\cdot,s)$ is a continuous non-decreasing function on $[0,\infty)$ with $\Phi_{p}(0,s)=0$,
	\\\textnormal{(ii)} $\lambda(\{s:\Phi_{p}(u,s)=0\,\,\,\textnormal{for some $u=u(s)\neq 0$} \})=0$,
	\\\textnormal{(iii)} there exists a constant $C>0$ such that $\Phi_{p}(2u,s)\leq C \Phi_{p}(u,s)$, for all $u\geq 0$ and $s\in S$.
\end{lem}
\begin{proof}
	Points (i) and (iii) are proved adapting the same arguments of the proof of the points (i) and (iii) of Lemma 3.1 in \cite{RajRos} to our framework. We sketch them for the sake of completeness. First, it is easy to prove that, for any fixed $s$, $U(\cdot,s)$ is continuous and so that $U^{*}(\cdot,s)$ is continuous. Then by dominated convergence theorem, we obtain the continuity of $\Phi_{p}(\cdot,s)$. Further, notice that $U^{*}(\cdot,s)$ is non-decreasing and that 
	\begin{equation}\label{min-max}
|ux|^{p}I(|ux|>1)+|xu|^{2}I(|xu|\leq 1)=\begin{cases}
(|xu|^{p}\wedge|xu|^{2})\quad\textnormal{if $0\leq p\leq2$,}\\ (|xu|^{p}\vee|xu|^{2})\quad\textnormal{if $ p>2$.}\\
\end{cases}
	\end{equation}
	is increasing in $x\geq0$. To prove (iii) observe that from (\ref{min-max}) and Lemma \ref{l3} we have
	\begin{equation*}
	\Phi_{p}(2u,s)\leq 2|U(u,s)|+27V_{0}(u,s)+4u^{2}\sigma^{2}(s)+(2^{p}+4)V_{p}(u,s)\leq (2^{p}+31)\Phi_{p}(u,s).
	\end{equation*}
	To prove point (ii) we proceed as follow. Recall that we have different formulations of $\lambda$. If $\Phi_{p}(u,s)=0$ for some $u=u(s)\neq 0$ then $\sigma^{2}(s)=0$, $|\rho|(s,\mathbb{R})=0$ and $U(u,s)=0$. The last two equalities imply that $a(s)=0$. Thus,
	\begin{equation*}
S_{0}:=\{s:\Phi_{p}(u,s)=0\,\,\,\textnormal{for some $u=u(s)\neq 0$}  \}=\{s:a(s)=\sigma^{2}(s)=|\rho|(s,\mathbb{R})=0 \},
	\end{equation*}
	which shows also that $S_{0}$ is a measurable set. Let $A$ be any measurable set in $S_{0}$ and notice that $\nu_{0}(A)=\int_{A}a(s)\lambda(ds)=0$, hence we obtain that $|\nu_{0}|(S_{0})=0$.
\\ Now, observe that given the above arguments we have that
\begin{equation*}
\int_{S_{0}}\int_{\mathbb{R}}(1\wedge |x|^{2})|\rho|(s,dx)\lambda(ds)=0,
\end{equation*}
\begin{equation*}
\nu(S_{0})=\sup_{I_{S_{0}}}\sum_{i\in I_{S_{0}}}|F_{A_{i}}(B_{i})|=\sup_{I_{S_{0}}}\sum_{i\in I_{S_{0}}}\bigg|\int_{A_{i}}\int_{B_{i}}\rho(s,dx)\lambda(ds)\bigg|=0,
\end{equation*}
and
\begin{equation*}
\xi(S_{0})=\sup_{I_{S_{0}}}\sum_{i\in I_{S_{0}}}|\int_{B_{i}}(1\wedge x^{2})F_{A_{i}}(dx)|=\sup_{I_{S_{0}}}\sum_{i\in I_{S_{0}}}\bigg|\int_{A_{i}}\int_{B_{i}}(1\wedge x^{2})\rho(s,dx)\lambda(ds)\bigg|=0.
\end{equation*}
Therefore, we get (in its different formulations) that $\lambda(S_{0})=0$.
\end{proof}
\begin{lem}
Let $\{\mu_{n}\}$ be a sequence of QID probability distributions on $\mathbb{R}$ with c.t.~$(a_{n}, \sigma_{n},G_{n})$. Let $a_{n}\rightarrow0$, $\sigma^{2}_{n}\rightarrow0$ ad $\int_{\mathbb{R}}(1\wedge |x|^{2})|G_{n}|(dx)\rightarrow0$. Assume that $\int_{|x|>1}|x|^{b}|G_{n}|(dx)<\infty$ for all $n\in\mathbb{N}$. Then, for every $b>0$,
\begin{equation*}
\int_{|x|>1}|x|^{b}|G_{n}|(dx)\rightarrow0\Rightarrow\int_{\mathbb{R}}|x|^{b}\mu_{n}(dx)\rightarrow0.
\end{equation*}
\end{lem}
\begin{proof}
	For every $n\in\mathbb{N}$, let $\mu_{n}^{+}$ and $\mu_{n}^{-}$ the ID distributions with c.t.~$(a_{n},\sigma^{2}_{n}, G_{n}^{+})$ and $(0,0, G_{n}^{-})$ respectively. Then, from Lemma 3.2 in \cite{RajRos} we obtain that
	\begin{equation*}
	\int_{|x|>1}|x|^{b}G_{n}^{+}(dx)\rightarrow0\Rightarrow\int_{\mathbb{R}}|x|^{b}\mu^{+}_{n}(dx)\rightarrow0.
	\end{equation*}
	and
	\begin{equation*}
	\int_{|x|>1}|x|^{b}G_{n}^{-}(dx)\rightarrow0\Rightarrow\int_{\mathbb{R}}|x|^{b}\mu^{-}_{n}(dx)\rightarrow0.
	\end{equation*}
	
	Now, for every $n\in\mathbb{N}$, let $X_{n}$, $X^{+}_{n}$ and $X^{-}_{n}$ be the real valued r.v.~with distributions $\mu_{n}$, $\mu^{+}_{n}$ and $\mu^{-}_{n}$, respectively; thus, $X_{n}+X^{-}_{n}\stackrel{d}{=}X^{+}_{n}$ with $X_{n}$ and $X^{-}_{n}$ independent. Notice that for $b\leq 1$ we have
	\begin{equation*}
	\mathbb{E}[|X_{n}|^{b}]\leq \mathbb{E}[|X_{n}+X_{n}^{-}|^{b}]+ \mathbb{E}[|X_{n}^{-}|^{b}]= \mathbb{E}[|X_{n}^{+}|^{b}]+\mathbb{E}[|X_{n}^{-}|^{b}]
	\end{equation*}
	while for $b>1$ we have
		\begin{equation*}
		\mathbb{E}[|X_{n}|^{b}]\leq \lceil b\rceil\left(\mathbb{E}[|X_{n}+X_{n}^{-}|^{\lceil b\rceil}]+ \mathbb{E}[|X_{n}^{-}|^{\lceil b\rceil}]\right)=\lceil b\rceil\left(\mathbb{E}[|X_{n}^{+}|^{\lceil b\rceil}]+ \mathbb{E}[|X_{n}^{-}|^{\lceil b\rceil}]\right)
		\end{equation*}
		where $\lceil b\rceil$stands for the lowest natural number greater than $b$. Finally, by sending $n\rightarrow\infty$ we obtain the stated result.
\end{proof}
Before presenting the main result of this section, we need some preliminaries.
\\ Define the Musielak-Orlicz space as in \cite{RajRos}:
\begin{equation*}
L_{\Phi_{p}}(S;\lambda)=\left\{f\in L_{0}(S;\lambda):\,\,\,\int_{S}\Phi_{p}(|f(s)|,s)\lambda(ds)<\infty \right\}.
\end{equation*}
The space $L_{\Phi_{p}}(S;\lambda)$ is a complete linear metric space with the $F$-norm defined by 
\begin{equation*}
\| f\|_{\Phi_{p}}=\inf\limits_{c>0}\left\{\int_{S}\Phi_{p}(c^{-1}|f(s)|,s)\lambda(ds)\leq c \right\}.
\end{equation*}
Simple functions are dense in $L_{\Phi_{p}}(S;\lambda)$ and $L_{\Phi_{p}}(S;\lambda)\hookrightarrow L_{0}(S;\lambda)$ is continuous, where in the present case $L_{0}(S;\lambda)$ is equipped with the topology of convergence in $\lambda$ measure on every set of finite $\lambda$-measure. Moreover, $\| f_{n}\|_{\Phi_{p}}\rightarrow0\Leftrightarrow \int_{S}\Phi_{p}(|f(s)|,s)\lambda(ds)\rightarrow0$.
\\ We remark that the above definitions and results hold both in the case $\lambda$ is atomless and when it is not (see Chapter III of \cite{RajRos} and Musielak's book \cite{Musielak}).
\begin{thm}\label{Orlicz-theorem}
	Let $0\leq p\leq q$ and $\Phi_{p}$ defined as in $(\ref{Phi})$. Then
	\begin{equation*}
	\left\{\textnormal{$f:$ $f$ is $\Lambda$-integrable and }\mathbb{E}\left[\bigg|\int_{S}fd\Lambda\bigg|^{p}\right]<\infty\right\}\supset L_{\Phi_{p}}(S;\lambda)
	\end{equation*}
	and the linear mapping
	\begin{equation*}
	L_{\Phi_{p}}(S;\lambda)\ni f\mapsto\int_{S}fd\Lambda\in L_{p}(\Omega;\mathbb{P})
	\end{equation*}
	is continuous.
\end{thm}
\begin{rem}
	When $p=0$ the statement becomes $\left\{\textnormal{$f:$ $f$ is $\Lambda$-integrable} \right\}\supset L_{\Phi_{0}}(S;\lambda)$.
\end{rem}
\begin{proof}
	We partially follow the proof of Theorem 3.3 in \cite{RajRos}. Let $f\in L_{\Phi_{p}}(S;\lambda)$, which is equivalent to consider $\int_{S}\Phi_{p}(|f(s)|,s)\lambda(ds)<\infty$. The conditions (i), (ii) and (iii) of Theorem \ref{theorem-general} (and also Theorems \ref{theorem1}, \ref{theorem2}  and \ref{theorem-bounded}) are satisfied. Then, we have that
	\begin{equation*}
	\int_{\{|u|>1\}}|u|^{p}|F_{f}|(du)\leq \int_{S}\int_{\{|f(s)x|>1\}}|f(s)x|^{p}|\rho|(s,dx)\lambda(ds)
	\end{equation*}
	\begin{equation*}
	\leq \int_{S}V_{p}(|f(s)|,s) \lambda(ds) \leq \int_{S}\Phi_{p}(|f(s)|,s)\lambda(ds)<\infty.
	\end{equation*}
	Hence, thanks to Theorem 6.2 point (a) in \cite{LPS} we have that $\mathbb{E}\left[|\int_{S}fd\Lambda|^{p}\right]<\infty$.
\end{proof}
\section{Spectral representation of generated QID processes}\label{Ch-Spectr}
In this section we are going to introduce QID stochastic processes, and present various spectral representation for generated QID processes (see Theorem \ref{Biglast} below). We start with the definition of QID processes.
\begin{defn}[QID processes]\label{defBiglast}
	Let $T$ be an arbitrary index set. A stochastic process $X=\{X_{t};\,t\in T\}$ is said to be a \textnormal{$QID$ process} if and only if for every finite set of indices $t_{1},\ldots ,t_{k}$ in the index set $T$ 
	\begin{equation*}
	X_{t_{1},...,t_{k}}:=(X_{t_{1}},...,X_{t_{k}})
	\end{equation*}
	is a multivariate QID random variable.
\end{defn}
The existence of QID processes is ensured by the Kolmogorov extension theorem. From Definition \ref{defBiglast}, it is clear that the class of QID processes is strictly larger than the class of ID processes.
\\Before presenting other results we introduce some preliminaries. We use the general framework introduced in \cite{Ros}. Let $T$ be an arbitrary (possibly uncountable) index set, $\mathbb{R}^{T}$ denote the space of all functions $x:T\rightarrow\mathbb{R}$ and $\mathcal{B}^{T}$ be its cylindrical product $\sigma$-algebra. Given an underlying probability space $(\Omega,\mathcal{F},\mathbb{P})$, for a stochastic process $X=(X_{t})_{t\in T}$ we define its law $\mathcal{L}(X)$ as a probability measure on $(\mathbb{R}^{T},\mathcal{B}^{T})$ such that 
\begin{equation*}
\mathcal{L}(X)(A)=\mathbb{P}(\{\omega\in\Omega:(X_{t}(\omega))_{t\in T}\in A \}),\quad A\in\mathcal{B}^{T}
\end{equation*}
We denote by $x_{S}$ the restriction of $x$ to $S\subset T$ and by $0_{S}$ the origin of $\mathbb{R}^{S}$, which depending on the setting is a point or a one-point set. Further, we let $\hat{T}$ to be defined as $\hat{T}:=\{I\subset T:0<\textnormal{Card}(I)<\infty \}$, let $\pi_{S}:\mathbb{R}^{T}\rightarrow\mathbb{R}^{S}$ be the projection from $\mathbb{R}^{T}$ onto $\mathbb{R}^{S}$ (namely $\pi_{S}(x)=x_{S}$) and let $\mathcal{B}_{00}^{S}:=\{B\in\mathcal{B}^{S}:0_{S}\notin B \}$. Finally, we uses as a cutoff function the bounded measurable function $\chi$ defined as $\chi:\mathbb{R}\mapsto\mathbb{R}$ such that $\chi(v)=1+o(|v|)$ as $v\rightarrow0$ and $\chi(v)=O(|v|^{-1})$ as $|v|\rightarrow\infty$. Then, the truncation of $v=(v_{1},...,v_{n})\in\mathbb{R}^{n}$ is defined by $[[v]]=(v_{1}\chi(|v_{1|}),...,v_{n}\chi(|v_{n|}))$ and of $x\in\mathbb{R}^{S}$ by $[[x]](s)=x(s)\chi(|x(s)|)$, $s\in S$. 
\\Let $\{\nu_{I}:I\in\hat{T}\}$ be a family of finite dimensional L\'{e}vy measures on $(\mathbb{R}^{J},\mathcal{B}^{J})$ s.t.~for every $I,J\in\hat{T}$ with $I\subset J$
\begin{equation*}
\nu_{J}\circ\pi_{IJ}^{-1}=\nu_{I} \,\,\,\textnormal{on $\mathcal{B}^{I}_{00}$},
\end{equation*}
where $\pi_{IJ}:\mathbb{R}^{J}\rightarrow\mathbb{R}^{I}$ is the natural projection. Then we say that $\{\nu_{I}:I\in\hat{T}\}$ is \textit{consistent} (see Definition 2.6 in \cite{Ros}). Moreover, we say that a family of probability measures is \textit{projective} if the above condition is satisfied on the whole $\mathcal{B}^{S}$ (and not just on $\mathcal{B}^{I}_{00}$). 

The first result of this section concerns a general property of QID distributions.
\begin{pro}\label{proBiglast}
All the marginal distributions of a multivariate QID distribution are QID.
\end{pro}
\begin{proof}
Let $d\in\mathbb{N}$ and $X=(X_{1},...,X_{d})$ be QID. Then there exists at least two ID r.v.~$Y$ and $Z$ s.t.~$X+Y\stackrel{d}{=}Z$. Consider any $I\subset\{1,...,d\}$. Let $\theta=(\theta_{1},...,\theta_{d})\in\mathbb{R}^{d}$ be s.t.~$\theta_{i}=0$ for $i\notin I$ and let $\theta_{I}\in\mathbb{R}^{I}$ containing all the non-zero element of $\theta$. Then
\begin{equation}\label{proBiglast-eq}
\hat{\mathcal{L}}(X_{I})(\theta_{I})=\hat{\mathcal{L}}(X)(\theta)=\frac{\hat{\mathcal{L}}(Z)(\theta)}{\hat{\mathcal{L}}(Y)(\theta)}=\frac{\hat{\mathcal{L}}(Z_{I})(\theta_{I})}{\hat{\mathcal{L}}(Y_{I})(\theta_{I})}.
\end{equation}
Thus, we obtain the stated results by Theorem 3.1 in \cite{HornSteutel}, namely by the fact that all the marginal distributions of a multivariate ID distribution are ID.
\end{proof}
Recall from Proposition 3.2 in \cite{Kallenberg} that for two stochastic processes $X^{(1)}$ and $X^{(2)}$ we have that $X^{(1)}\stackrel{d}{=}X^{(2)}$ if and only if $(X^{(1)}_{t_{1}},...,X^{(1)}_{t_{k}})\stackrel{d}{=}(X^{(2)}_{t_{1}},...,X^{(2)}_{t_{k}})$ where $t_{1},...,t_{k}\in T$ and $k\in\mathbb{N}$. We have the following simple result.
\begin{thm}\label{Biglast}
	Let $T$ be an arbitrary index set. A stochastic process $X=\{X_{t};\,t\in T\}$ is QID if there exist two ID processes $Y$ and $Z$ such that $X+Y\stackrel{d}{=}Z$ with $Y$ independent of $X$. In this case, we say that $X$ is \textnormal{generated} by $Y$ and $Z$ and that $X$ is a \textnormal{generated QID process}
\end{thm}
\begin{proof}
	Let $X$ be a stochastic process and let $Y$ and $Z$ two ID processes s.t.~$X+Y\stackrel{d}{=}Z$ with $Y$ independent of $X$. Consider any finite set of indices $t_{1},...,t_{k}$ in the index set $T$. Since $(Y_{t_{1}},...,Y_{t_{k}})$ and $(Z_{t_{1}},...,Z_{t_{k}})$ are ID distributed, $(X_{t_{1}},...,X_{t_{k}})$ is QID distributed. Then by definition the process $X=\{X_{t};\,t\in T\}$ is QID.
\end{proof}
\noindent We sometimes also say that $Y$ and $Z$ are the \textit{generating ID process of} $X$. Moreover, we remark that while for any given two ID processes there exist a unique QID process (when it exists), it might happen that for a given QID processes there exist more than just two generating ID processes.
We present now the main question left open from this work.
\begin{open}
	Is it true that any QID process is generated by two ID processes? In other words, is it true that any QID process is a generated QID process?
\end{open}
The answer of this question is not trivial. It lies on the capacity of building canonical ID distributions from QID distributions and it involves the use of the axiom of choice. In particular, the axiom of choice is used to show the existence of a family of ID probability measures (actually there exist uncountably many such families). Then, in order to use the Kolmogorov extension theorem we need to show that this family of probability measures is consistent. However, despite Proposition \ref{proBiglast} and equation (\ref{proBiglast-eq}), which provide consistency for each finite set of indices, it is not clear how to obtain the result as shown in the following example.
\begin{exmp}
Consider the c.t.~of a QID r.v.~$X=(X_{1},X_{2})$ on $\mathbb{R}^{2}$ given by $(\gamma,\theta,\nu)$. Let $(\gamma_{1},\theta_{1},\nu_{1})$ and $(\gamma_{2},\theta_{2},\nu_{2})$ the c.t.~of $X_{1}$ and $X_{2}$. Hence, from Proposition \ref{proBiglast} we have $\nu_{1}(A)=\nu((A,\mathbb{R}))$ and $\nu_{2}(B)=\nu((\mathbb{R},B))$ for every $A,B\in\mathcal{B}_{00}(\mathbb{R})$. Assume that the quasi-L\'{e}vy measures are signed measures and consider their Jordan decomposition. For example let $E^{+}_{1}$ and $E^{+}_{2}$ be the Hahn decomposition of $\mathbb{R}$ under the signed measure $\nu_{1}$ and assume w.l.o.g.~that $\{0\}\notin E^{+}_{1}$. Then, $0\leq\nu^{+}_{1}(\mathbb{R})=\nu_{1}(E^{+}_{1})=\nu((E^{+}_{1},\mathbb{R}))$. Applying the same argument to $\nu_{2}$ we have that $0\leq\nu^{+}_{2}(\mathbb{R})=\nu_{2}(E^{+}_{2})=\nu((\mathbb{R}, E^{+}_{2}))$. Then we have that $(E^{+}_{1},\mathbb{R})\cup(\mathbb{R},E^{+}_{2})$ is a subset of $\mathbb{R}^{2}$ where $\nu$ is positive. On the other hand by applying the same arguments we have that $(E^{-}_{1}\setminus\{0\},\mathbb{R})\cup(\mathbb{R},E^{+}_{2}\setminus\{0\})$ is a subset of $\mathbb{R}^{2}$ where $\nu$ is negative. But these two sets have intersections, which means that the marginals of the Jordan decompositions of $\nu$ (call them $\nu^{+}$ and $\nu^{-}$ with marginals $\tilde{\nu}^{+}_{1}$, $\tilde{\nu}^{+}_{2}$, $\tilde{\nu}^{-}_{1}$ and $\tilde{\nu}^{-}_{2}$) are s.t.~$\nu_{1}^{+}(\mathbb{R})=\nu_{1}(E^{+}_{1})=\nu((E^{+}_{1},\mathbb{R}))\leq \nu^{+}((E^{+}_{1},\mathbb{R}))\leq \nu^{+}((\mathbb{R},\mathbb{R}))=\tilde{\nu}_{1}^{+}(\mathbb{R})$, where the last equality comes from the fact that $\nu^{+}$ is a L\'{e}vy measure of a ID r.v.~whose marginals exist and are ID thanks to Theorem 3.1 in \cite{HornSteutel}. A similar result hold for the other cases.
\\ This implies that the projection of $\nu^{+}$ is not always $\nu_{1}^{+}$, and similarly for $\nu^{-}$. In other words, \textnormal{the positive and negative part of a quasi-L\'{e}vy measure are always not consistent (and, hence, projective)}. Hence, given $X+Y\stackrel{d}{=}Z$ where $Z=(Z_{1},Z_{2})$ and $Y=(Y_{1},Y_{2})$ are two ID r.v.~on $\mathbb{R}^{2}$ with quasi-L\'{e}vy measures $\nu^{+}$ and $\nu^{-}$ respectively, it is not always true that $Z_{1}$ and $Y_{1}$ have quasi-L\'{e}vy measure given by the positive and negative part of $\nu_{1}$, and similarly for $Z_{2}$ and $Y_{2}$. Therefore, a family of ID probability measures where each probability measure has the L\'{e}vy measure given by the positive (or negative) part of the quasi-L\'{e}vy measure of the respective QID probability measure is not always projective.
\end{exmp}

For the moment let index set $T$ to be countable. Let $\mathcal{B}^{l_{2}}$ be the Borel $\sigma$-algebra on $l_{2}$. We note that since $l_{2}$ is separable the Borel and cylindrical $\sigma$-algebras coincide (see page 38 in \cite{Talagrand}). Then, we have the following definition.
\begin{defn}[quasi-L\'{e}vy type measure on $(l_{2},\mathcal{B}^{l_{2}})$]\label{def1-l2} Let $\mathcal{B}^{l_{2}}_{r}:=\{B\in\mathcal{B}^{l_{2}}|B\cap\{x\in l_{2}:\|x\|<r \}=\emptyset \}$ and $\mathcal{B}^{l_{2}}_{0}:=\bigcup_{r>0}\mathcal{B}^{l_{2}}_{r}$. Let $\nu:\mathcal{B}^{l_{2}}_{0}\rightarrow\mathbb{R}$ be a set function s.t.~$\nu_{|\mathcal{B}^{l_{2}}_{r}}$ is a finite signed measure for each $r > 0$ and denote the total variation, positive and negative part of $\nu_{|\mathcal{B}^{l_{2}}_{r}}$ by $|\nu_{|\mathcal{B}^{l_{2}}_{r}}|$, $\nu^{+}_{|\mathcal{B}^{l_{2}}_{r}}$ and $\nu^{-}_{|\mathcal{B}^{l_{2}}_{r}}$ respectively. Then the \textnormal{total variation} $|\nu|$, the \textnormal{positive part} $\nu^{+}$ and the \textnormal{negative part} $\nu^{-}$ of $\nu$ are defined to be the unique measures on $(l_{2},\mathcal{B}^{l_{2}})$ satisfying
	\begin{equation*}
	|\nu(\{0_{T}\})|=\nu^{+}(\{0_{T}\})=\nu^{-}(\{0_{T}\})=0
	\end{equation*}
	\begin{equation*}
	\text{and}\quad|\nu|(A)=|\nu_{|\mathcal{B}^{l_{2}}_{r}}|,\,\,\nu^{+}(A)=\nu_{|\mathcal{B}^{l_{2}}_{r}}^{+}(A),\,\,\nu^{-}(A)=\nu_{|\mathcal{B}^{l_{2}}_{r}}^{-}(A),
	\end{equation*}
	for $A\in\mathcal{B}^{l_{2}}_{r}$, for some $r>0$.
	
	A \textnormal{quasi-L\'{e}vy type measure on $(l_{2},\mathcal{B}^{l_{2}})$} is a function $\nu:\mathcal{B}^{l_{2}}_{0}\rightarrow\mathbb{R}$ with the above properties and s.t.~$\int_{l_{2}}(1\wedge \|x\|^{2})|\nu|(dx)<\infty$.
\end{defn}
\begin{lem}
$\mathcal{B}^{l_{2}}_{r}$ is a $\sigma$-algebra.
\end{lem}
\begin{proof}
Since the space $Y_{r}:=\{x\in l_{2}:\|x\|\geq r \}$ is a non-empty subspace of $l_{2}$ then we have $Z_{r}:=\{B\subseteq Y_{r}|B\in\mathcal{B}^{l_{2}} \}=\mathcal{B}^{l_{2}}_{r}$. Observe now that $Y_{r}$ is measurable (namely $Y_{r}\in\mathcal{B}^{l_{2}}$) because $l^{2}$ is a complete separable metric space and the norm is continuous, then $Z_{r}$ is the restriction of $\mathcal{B}^{l_{2}}$ on $Y_{r}$. Hence, $\mathcal{B}^{l_{2}}_{r}$ is a $\sigma$-algebra and the argument applies to any $r>0$.
\end{proof}
If $\nu:\mathcal{B}^{l_{2}}_{0}\rightarrow\mathbb{R}$ is s.t.~$\nu_{|\mathcal{B}^{l_{2}}_{r}}$ is a finite signed measure for each $r>0$, then we have that $|\nu_{|\mathcal{B}^{l_{2}}_{r}}|(A)=|\nu_{|\mathcal{B}^{l_{2}}_{s}}|(A)$ for every $A\in\mathcal{B}^{l_{2}}_{r}$ with $0<s\leq r$ and the same applies to $\nu_{|\mathcal{B}^{l_{2}}_{r}}^{+}$ and $\nu_{|\mathcal{B}^{l_{2}}_{r}}^{-}$. Moreover, since we are in the countable framework (the index set considered $T$ is countable since $l_{2}\subset\mathbb{R}^{\mathbb{N}}$ hence $0_{T}\in\mathcal{B}^{l_{2}}$) then $|\nu|$, $\nu^{+}$ and $\nu^{-}$ are well defined and the condition $|\nu(\{0\})|=\nu^{+}(\{0\})=\nu^{-}(\{0\})=0$ together with the Carath\'{e}odory's extension theorem ensure their uniqueness. Notice that we have the equivalent of Lemma \ref{lemma-last?} in this framework.
\begin{lem}\label{lemma-last2}
A set function $\nu:\mathcal{B}^{l_{2}}_{0}\rightarrow\mathbb{R}$ is a quasi-L\'{e}vy type measure if and only if there exist two L\'{e}vy measures $\nu^{(1)}$ and $\nu^{(2)}$ s.t.~$\nu_{|\mathcal{B}_{r}(\mathbb{R})}(A)=\nu_{|\mathcal{B}^{l_{2}}_{r}}^{(1)}(A)-\nu_{|\mathcal{B}^{l_{2}}_{r}}^{(2)}(A)$ for $A\in\mathcal{B}^{l_{2}}_{r}$, for some $r>0$. Moreover, $\nu$ is unique.
\end{lem}
\begin{proof}
	It follows from the same arguments used in the proof of Lemma \ref{lemma-last?}.
\end{proof}
We have the following representation of a certain QID processes.
\begin{thm}\label{l2-LKrepresentation}
	Let $X$ be a discrete parameter QID process with values in $l_{2}$ s.t.~there exist two generating ID processes with values in $l_{2}$. Then there exists an unique triplet $(z_{0},\mathcal{K},\nu)$ consisting of $z_{0}\in l_{2}$, a non-negative definite function operator $\mathcal{K}:l_{2}\rightarrow\mathbb{R}$ and a quasi-L\'{e}vy type measure $\nu$ on $(l_{2},\mathcal{B}^{l_{2}})$ s.t.
	\begin{equation}\label{QIDprocess}
	\hat{\mathcal{L}}(X)(y)=\exp\left(i\langle z_{0},y \rangle -\frac{1}{2}\langle y,\mathcal{K}y\rangle+\int_{l_{2}}\left( e^{i\langle y ,x \rangle}-1-i\langle\tau(x),y \rangle\right)\nu(dx)\right)
	\end{equation}
	where $y\in l_{2}$.
\end{thm}
\begin{proof}
	Since $X$, $Y$ and $Z$ are $l_{2}$-valued stochastic processes, then they can be seen as three r.v. on $(l_{2},\mathcal{B}^{l_{2}})$, with the property that $X+Y\stackrel{d}{=}Z$ with $Y$ independent of $X$. Now, let $\nu^{(1)}$ and $\nu^{(2)}$ be the L\'{e}vy measures of the $l_{2}$ valued ID processes $Z$ and $Y$, respectively. Then, we have
	\begin{equation*}
	\hat{\mathcal{L}}(X)(y)=\hat{\mathcal{L}}(Z)(y)/\hat{\mathcal{L}}(Y)(y) =\exp\bigg(i\langle z_{0},y \rangle
	\end{equation*}	
	\begin{equation*}
	-\frac{1}{2}\langle y,\mathcal{K}y\rangle+\int_{l_{2}}\left( e^{i\langle y ,x \rangle}-1-i\langle\tau(x),y \rangle\right)\nu^{(1)}(dx)-\int_{l_{2}}\left( e^{i\langle y ,x \rangle}-1-i\langle\tau(x),y \rangle\right)\nu^{(2)}(dx)\bigg)
	\end{equation*}
	 Let $\nu:\mathcal{B}_{0}^{l_{2}}\rightarrow\mathbb{R}$ be a set function s.t.~$\nu_{|\mathcal{B}^{l_{2}}_{r}}(A)=\nu^{(1)}_{|\mathcal{B}^{l_{2}}_{r}}(A)-\nu^{(2)}_{|\mathcal{B}^{l_{2}}_{r}}(A)$ for every $A\in\mathcal{B}_{r}^{l_{2}}$ with $r>0$. Then by Lemma \ref{lemma-last2}, $\nu$ is uniquely determined by the difference $\nu^{(1)}-\nu^{(2)}$. Indeed the absence of uniqueness of $\nu^{(1)}-\nu^{(2)}$ will violate the condition $X+Y\stackrel{d}{=}Z$ due to the uniqueness of the L\'{e}vy-Khintchine representation on $l_{2}$ (see Theorem 4.10 Chapter 4 in \cite{Parthasarathy}). Hence, we obtain the expression (\ref{QIDprocess}). Finally, it directly follows that $z_{0}$ and $\mathcal{K}$ are uniquely determined.
\end{proof}
\begin{thm}
	Conversely for every $(z_{0},\mathcal{K},\nu)$ as above there exists a QID process with values in $l_{2}$.
\end{thm}
\begin{proof}
	It is enough to see that the positive and negative part of $\nu$, namely $\nu^{+}$ and $\nu^{-}$, are L\'{e}vy measures on $\mathcal{B}^{l_{2}}$. Hence, we can extract two ID processes with values in $l_{2}$. Then, these two processes generate a QID with values in $l_{2}$.
\end{proof}
\begin{open}
	In Theorem \ref{l2-LKrepresentation} we have the condition that discrete parameter QID process with values in $l_{2}$ must be generated by two generating ID processes with values in $l_{2}$. In order to get rid of this condition we need to prove that for every discrete parameter QID process with values in $l_{2}$ there exist two generating ID processes with values in $l_{2}$. Is this possible?
\end{open}
\subsection{L\'{e}vy-Khintchine representation of generated QID processes}
In this section we are going to investigate the representation of generated QID processes. For the first result we will assume that the quasi-L\'{e}vy measure of QID processes, whose definition is going to be provided below, is a signed measure. This situation happens when one of the two generating ID processes has a finite L\'{e}vy measure (see Definition 2.1 in \cite{Ros}). This is because in that case $\nu=\nu_{1}-\nu_{2}$ is a signed measure. On the other hand, the second result concerns general generated QID processes.

We introduce now the definition of quasi-L\'{e}vy type measure in this framework.
\begin{defn}
	A signed measure $\nu$ on $(\mathbb{R}^{T},\mathcal{B}^{T})$ is said to be a quasi-L\'{e}vy type measure if the following two conditions hold
	\\ \textnormal{(QL1)} for every $t\in T$ we have $\int_{\mathbb{R}^{T}}|x(t)|^{2}\wedge 1|\nu|(dx)<\infty,$
	\\ \textnormal{(QL2)} for every $A\in\mathcal{B}^{T}$ we have $|\nu(A)|=\nu^{+}_{\ast}(A\setminus 0_{T})+\nu^{-}_{\ast}(A\setminus 0_{T})$, where $\nu^{+}_{\ast}$ and $\nu^{+}_{\ast}$ are the inner measures of the Jordan decomposition of $\nu$.
\end{defn}
We can present now the first main result of this section, namely the L\'{e}vy-Khintchine formulation of QID processes with quasi-L\'{e}vy measure being a signed measure.
\begin{thm}\label{theorem-LK-QIDprocesses-finite}
	Let $T$ be an arbitrary index set and let $X=(X_{t})_{t\in T}$ be a QID process s.t.~there exists two of generating ID process with one having finite L\'{e}vy measure. Then there exists a unique triplet $(\Sigma,\nu,b)$ where $\Sigma$ is a non-negative definite function on $T\times T$, $\nu$ is a quasi-L\'{e}vy type measure on $(\mathbb{R}^{T},\mathcal{B}^{T})$ and a function $b\in \mathbb{R}^{T}$ s.t.~for every $I\in \hat{T}$ and $\theta\in\mathbb{R}^{I}$
	\begin{equation}\label{Levy-Khintchine any T}
	\hat{\mathcal{L}}(X_{I})(\theta)=\exp\left( i\langle\theta,b_{I}\rangle-\frac{1}{2}\langle\theta,\Sigma_{I}\theta \rangle+\int_{\mathbb{R}^{T}}(e^{i\langle\theta,x_{I} \rangle}-1-i\langle\theta,[[x_{I}]] \rangle )\nu(dx)\right),
	\end{equation}
	where $\Sigma_{I}$ is the restriction of $\Sigma$ to $I\times I$. The triplet $(\Sigma,\nu,b)$ is called the generating triplet of $X$ and $\nu$ the quasi-L\'{e}vy measure of $X$.
\end{thm}
\begin{proof}
	Let $X^{1}$ and $X^{2}$ the ID processes generating $X$, with one having finite L\'{e}vy measure, and denote by $\nu^{(1)}$ and $\nu^{(2)}$ the respective L\'{e}vy measures on $(\mathbb{R}^{T},\mathcal{B}^{T})$. 
	Let $\nu:=\nu^{(1)}-\nu^{(2)}$. Then $\nu$ is a well-defined signed measure on $(\mathbb{R}^{T},\mathcal{B}^{T})$. Moreover, we have for $I\in\hat{T}$
	\begin{equation*}
	\hat{\mathcal{L}}(X_{I})=\frac{\hat{\mathcal{L}}(X^{1}_{I})}{\hat{\mathcal{L}}(X^{2}_{I})}=\exp\bigg( i\langle\theta,b_{I}\rangle-\frac{1}{2}\langle\theta,\Sigma_{I}\theta \rangle
	\end{equation*}
	\begin{equation*}
	+\int_{\mathbb{R}^{T}}(e^{i\langle\theta,x_{I} \rangle}-1-i\langle\theta,[[x_{I}]] \rangle )\nu^{(1)}(dx)-\int_{\mathbb{R}^{T}}(e^{i\langle\theta,x_{I} \rangle}-1-i\langle\theta,[[x_{I}]] \rangle )\nu^{(2)}(dx)\bigg)
	\end{equation*}
	\begin{equation*}
	=\exp\bigg( i\langle\theta,b_{I}\rangle-\frac{1}{2}\langle\theta,\Sigma_{I}\theta \rangle +\int_{\mathbb{R}^{T}}(e^{i\langle\theta,x_{I} \rangle}-1-i\langle\theta,[[x_{I}]] \rangle )\nu(dx)\bigg),
	\end{equation*}
	which gives (\ref{Levy-Khintchine any T}). From the above we also deduce that the $\nu\circ\pi^{-1}_{I}=\nu_{I}$ on $\mathcal{B}_{00}^{I}$, $I\in\hat{T}$, where $\nu_{I}$ is the quasi-L\'{e}vy measure of $X_{I}$. 
	\\Notice that $\forall t\in T$ we have that $\int_{\mathbb{R}^{T}}|x(t)|^{2}\wedge 1|\nu|(dx)\leq\int_{\mathbb{R}^{T}}|x(t)|^{2}\wedge 1\nu^{(1)}(dx)+\int_{\mathbb{R}^{T}}|x(t)|^{2}\wedge 1\nu^{(2)}(dx)<\infty$. Now let $\nu^{+}_{\ast}$ and $\nu^{+}_{\ast}$ are the inner measures of the Jordan decomposition of $\nu$. Since for every $A\in\mathcal{B}^{T}$ we have that $\nu^{(1)}(A)=\nu_{\ast}^{(1)}(A\setminus 0_{T})$ and $\nu^{(2)}(A)=\nu_{\ast}^{(2)}(A\setminus 0_{T})$, and $\nu^{(1)}(A)\geq \nu^{+}(A)=\nu_{\ast}^{+}(A)$ and $\nu^{(2)}(A)\geq \nu^{-}(A)=\nu_{\ast}^{-}(A)$ then $|\nu(A)|=\nu^{+}_{\ast}(A\setminus 0_{T})+\nu^{-}_{\ast}(A\setminus 0_{T})$. Therefore, $\nu$ is a quasi-L\'{e}vy measure on $(\mathbb{R}^{T},\mathcal{B}^{T})$ and, additionally, $\nu^{+}$ and $\nu^{-}$ are L\'{e}vy measures on $(\mathbb{R}^{T},\mathcal{B}^{T})$. In particular, notice that $\nu^{+}_{I}=\nu^{+}\circ\pi^{-1}$ and $\nu^{-}_{I}=\nu^{-}\circ\pi^{-1}$ are L\'{e}vy measures such that $\nu_{I}(A)=\nu^{+}_{I}(A)-\nu^{-}_{I}(A)$ for every $A\in\mathcal{B}(\mathbb{R})$, for $I\in\hat{T}$ (although it is not necessarily true that $\nu^{+}_{I}$ and $\nu^{-}_{I}$ are mutually singular on $(\mathbb{R}^{I},\mathcal{B}(\mathbb{R}^{I}))$).
	
	Let us prove uniqueness. Notice that the arguments used in the proof of Lemma \ref{lemma-last?}, which were applicable for the $(\mathbb{R}^{d},\mathcal{B}(\mathbb{R}^{d}))$ and $(l_{2},\mathcal{B}^{l_{2}})$ cases, do not apply here due to the different topological framework. 
	\\Assume that there exists another quasi-L\'{e}vy type measure $\tilde{\nu}$ on $(\mathbb{R}^{T},\mathcal{B}^{T})$ (hence $\tilde{\nu}^{+}$ and $\tilde{\nu}^{-}$ are two L\'{e}vy measures on $(\mathbb{R}^{T},\mathcal{B}^{T})$) s.t.
	\begin{equation}\label{eqBiglast}
	\hat{\mathcal{L}}(X_{I})=\exp\bigg( i\langle\theta,b_{I}\rangle-\frac{1}{2}\langle\theta,\Sigma_{I}\theta \rangle +\int_{\mathbb{R}^{T}}(e^{i\langle\theta,x_{I} \rangle}-1-i\langle\theta,[[x_{I}]] \rangle )\tilde{\nu}(dx)\bigg),
	\end{equation}
	but with $\tilde{\nu}\neq\nu$. Notice that (\ref{eqBiglast}) together with the uniqueness of quasi-L\'{e}vy measures imply that $\nu_{I}=\tilde{\nu}\circ\pi_{I}^{-1}=:\tilde{\nu}_{I}$ on $\mathcal{B}_{00}^{I}$, $I\in\hat{T}$, and that there exists a QID process $\tilde{X}$ with c.t.~$(\Sigma,\tilde{\nu},b)$ s.t. $\tilde{X}\stackrel{d}{=}X$.
	\\
	However, observe that since $\tilde{\nu}\neq\nu$ then $\tilde{\nu}^{+}-\tilde{\nu}^{-}\neq \nu^{+}-\nu^{-}$. Recall that by Corollary 2.9 that for every c.t.~$(\Sigma,\mu,b)$ where $\mu$ is a L\'{e}vy measures there exist a unique (in distribution) ID process. Let $b^{(1)}$ and $b^{(2)}$ two functions in $\mathbb{R}^{T}$ s.t.~$b=b^{(1)}-b^{(2)}$ and let $(\Sigma,\nu^{+},b^{(1)})$, $(0,\nu^{-},b^{(2)})$, $(\Sigma,\tilde{\nu}^{+},b^{(1)})$ and $(0,\tilde{\nu}^{-},b^{(2)})$ the c.t.~of the ID processes $Z$, $Y$, $\tilde{Z}$ and $\tilde{Y}$, respectively. Then observe that $X-Y\stackrel{d}{=}Z$ and $\tilde{X}-\tilde{Y}\stackrel{d}{=}\tilde{Z}$. Moreover, we have that, for every $I\in\hat{T}$,
	\begin{equation}\label{ZYZY}
	\frac{\hat{\mathcal{L}}(Z_{I})}{\hat{\mathcal{L}}(Y_{I})}=\frac{\hat{\mathcal{L}}(\tilde{Z}_{I})}{\hat{\mathcal{L}}(\tilde{Y}_{I})}\Rightarrow \hat{\mathcal{L}}(Z_{I})\hat{\mathcal{L}}(\tilde{Y}_{I})=\hat{\mathcal{L}}(\tilde{Z}_{I})\hat{\mathcal{L}}(Y_{I}).
	\end{equation}
	Now, let $\nu^{Z\tilde{Y}}_{I}:=\nu^{+}_{I}+\tilde{\nu}^{-}_{I}$ and $\nu^{\tilde{Z}Y}_{I}:=\tilde{\nu}^{+}_{I}+\nu^{-}_{I}$. It is possible to see that $\{\nu^{\tilde{Z}Y}_{I}:I\in\hat{T}\}$ and $\{\nu^{Z\tilde{Y}}_{I}:I\in\hat{T}\}$ are two consistent families of L\'{e}vy measures. Hence, they generate two unique L\'{e}vy measures $\nu^{\tilde{Z}Y}$ and $\nu^{Z\tilde{Y}}$ on $(\mathbb{R}^{T},\mathcal{B}^{T})$. Observe that $\nu^{+}+\tilde{\nu}^{-}$ and $\tilde{\nu}^{+}+\nu^{-}$ are two L\'{e}vy measure on $(\mathbb{R}^{T},\mathcal{B}^{T})$ because the sum of two L\'{e}vy measures on $(\mathbb{R}^{T},\mathcal{B}^{T})$ is still a L\'{e}vy measure. Since  $\nu^{\tilde{Z}Y}\circ\pi^{-1}_{I}=\nu^{\tilde{Z}Y}_{I}=\tilde{\nu}^{+}_{I}+\nu^{-}_{I}=(\nu^{+}+\tilde{\nu}^{-})\circ\pi^{-1}_{I}$ for every $I\in\hat{T}$ and by the uniqueness of L\'{e}vy measures on $(\mathbb{R}^{T},\mathcal{B}^{T})$, we conclude that $\nu^{\tilde{Z}Y}=\nu^{+}+\tilde{\nu}^{-}$ and similarly that $\nu^{Z\tilde{Y}}=\tilde{\nu}^{+}+\nu^{-}$. However, from (\ref{ZYZY}) we have that $\nu^{Z\tilde{Y}}_{I}=\nu^{\tilde{Z}Y}_{I}$ which implies that $\nu^{+}+\tilde{\nu}^{-}=\tilde{\nu}^{+}+\nu^{-}$, hence a contradiction to $\tilde{\nu}^{+}-\tilde{\nu}^{-}\neq \nu^{+}-\nu^{-}$ and, so, to $\tilde{\nu}\neq\nu$.
\end{proof}
Another way of proving the uniqueness when $\nu$ is finite is the following.  Let $\lambda_{I}:=\nu_{I}-\tilde{\nu}_{I}=0$ on $\mathcal{B}_{0}(\mathbb{R}^{I})$, $I\in\hat{T}$. Observe that $\{\lambda_{I}:I\in\hat{T}\}$ forms a consistent family of L\'{e}vy measures. Then by uniqueness of L\'{e}vy measures on $(\mathbb{R}^{T},\mathcal{B}^{T})$ (see Theorem 2.8 and Corollary 2.9 in \cite{Ros}) we have that $\lambda\circ\pi^{-1}_{I}=\lambda_{I}=0$ and in particular $\lambda=0$. Moreover, following Theorem 2.8 in \cite{Ros} for any other measure $\rho$ on $(\mathbb{R}^{T},\mathcal{B}^{T})$ such that $\rho \circ\pi^{-1}_{I}=\lambda_{I}=0$ we have that $\rho\geq\lambda$, hence $\rho\geq0$. Hence, whenever $\nu-\tilde{\nu}$ defines a measure (which is true for example when either $\nu$ or $\tilde{\nu}$ is finite) we have $\nu-\tilde{\nu}\geq 0\Rightarrow \nu\geq \tilde{\nu}$. However, we may apply the same arguments to $\bar{\lambda}_{I}:=\tilde{\nu}_{I}-\nu_{I}=0$, concluding that $\nu-\tilde{\nu}\geq 0\Rightarrow \nu\geq \tilde{\nu}\Rightarrow \nu= \tilde{\nu}$.
\begin{thm}
Conversely for every $(\Sigma,\nu,b)$ as in Theorem \ref{theorem-LK-QIDprocesses-finite} there exists a QID process which is unique in distribution.
\end{thm}
\begin{proof}
	It is enough to see that the positive and negative part of $\nu$, namely $\nu^{+}$ and $\nu^{-}$, are L\'{e}vy measures on $(\mathbb{R}^{T},\mathcal{B}^{T})$. Hence, we can extract two ID processes which in turn generate a QID process which is unique in distribution.
\end{proof}
From the arguments used in the previous result we have the following L\'{e}vy-Khintchine representation of general QID processes, namely the QID equivalent of Theorem 2.8 in \cite{Ros}.
\begin{thm}\label{theorem-LK-QIDprocesses}
	Let $T$ be an arbitrary index set and let $X=(X_{t})_{t\in T}$ be a generated QID process. Then there exists a triplet $(\Sigma,(\nu^{+},\nu^{-}),b)$ where $\Sigma$ and $b$ are as in the Theorem \ref{theorem-LK-QIDprocesses-finite} and $(\nu^{+},\nu^{-})$ is a couple of L\'{e}vy measure on $(\mathbb{R}^{T},\mathcal{B}^{T})$ s.t.~for every $I\in \hat{T}$ and $\theta\in\mathbb{R}^{I}$
	\begin{equation*}
	\hat{\mathcal{L}}(X_{I})=\exp\bigg( i\langle\theta,b_{I}\rangle-\frac{1}{2}\langle\theta,\Sigma_{I}\theta \rangle
	\end{equation*}
	\begin{equation}\label{Levy-KhintchineANY}
	+\int_{\mathbb{R}^{T}}(e^{i\langle\theta,x_{I} \rangle}-1-i\langle\theta,[[x_{I}]] \rangle )\nu^{+}(dx)-\int_{\mathbb{R}^{T}}(e^{i\langle\theta,x_{I} \rangle}-1-i\langle\theta,[[x_{I}]] \rangle )\nu^{-}(dx)\bigg).
	\end{equation}
	$(\Sigma,(\nu^{+},\nu^{-}),b)$ is called the generating triplet of $X$, and it is unique up to other couples of L\'{e}vy measures $(\nu^{Y},\nu^{Z})$ of ID processes that generate $X$. Conversely for every $(\Sigma,(\nu^{+},\nu^{-}),b)$ as above there exists a QID process which is unique in distribution.
\end{thm}
\begin{proof}
	It follows from the same arguments used in the proof of Theorem \ref{theorem-LK-QIDprocesses-finite}.
\end{proof}
\begin{rem}
	The reason why we have the uniqueness ``up to other couples of L\'{e}vy measures $(\nu^{Y},\nu^{Z})$ of ID processes that generate $X$" is because given two ID processes, say $Y$ and $Z$, that generate $X$ we cannot exclude that there exist two other ID processes, say $\bar{Y}$ and $\bar{Z}$ that generate. In other words, there might exist four (or more) L\'{e}vy measures call them $\nu^{Y},\nu^{Z},\bar{\nu^{Y}}$ and $\bar{\nu^{Z}}$ corresponding to $Y,Z,\bar{Y}$ and $\bar{Z}$, respectively, s.t.~for every $I\in \hat{T}$ and $\theta\in\mathbb{R}^{I}$
\begin{equation*}
\int_{\mathbb{R}^{T}}(e^{i\langle\theta,x_{I} \rangle}-1-i\langle\theta,[[x_{I}]] \rangle )\nu^{Z}(dx)-\int_{\mathbb{R}^{T}}(e^{i\langle\theta,x_{I} \rangle}-1-i\langle\theta,[[x_{I}]] \rangle )\nu^{Y}(dx)
\end{equation*}
	\begin{equation*}
	=\int_{\mathbb{R}^{T}}(e^{i\langle\theta,x_{I} \rangle}-1-i\langle\theta,[[x_{I}]] \rangle )\bar{\nu}^{Z}(dx)-\int_{\mathbb{R}^{T}}(e^{i\langle\theta,x_{I} \rangle}-1-i\langle\theta,[[x_{I}]] \rangle )\bar{\nu}^{Y}(dx).
	\end{equation*}
\end{rem}
Notice that $\nu:=\nu^{+}-\nu^{-}$ might not be defined for two reasons. First, because we might have $\infty-\infty$. Second because we cannot apply the arguments used for the definitions of the quasi-L\'{e}vy measures on $(\mathbb{R}^{d},\mathcal{B}(\mathbb{R}^{d}))$ or on $(l_{2},\mathcal{B}^{l_{2}})$ (see Definitions \ref{def1} and \ref{def1-l2}).
\begin{open}
	It is not possible to extend the arguments applied to the definition of quasi-L\'{e}vy type measure in $\mathbb{R}^{d}$ or in $l_{2}$, because the restrictions do not form $\sigma$-algebras. However it might be possible to still obtain a sensible definition of quasi-L\'{e}vy type measure on $\mathbb{R}^{T}$ without necessarily being a signed measure. In other words, is it possible to obtain a triplet from the above quadruplet?
\end{open}
\subsection{Spectral representation of discrete parameters QID processes}
Before presenting the result we need some preliminaries. Let $X=\{X_{n};\,n=1,2,...\}$ be a discrete parameters QID process with generating ID processes $Y$ and $Z$; let $b_{n}>0$ be s.t.~$\{b_{n}X_{n}\},\{b_{n}Y_{n}\},\{b_{n}Z_{n}\}\in l_{2}$ almost surely. Let $\mathcal{L}(\{b_{n}X_{n}\})\sim(z_{0},\mathcal{K},M)$, where $z_{0}\in l_{2}$, $\mathcal{K}$ is the covariance operator and $M$ is the quasi L\'{e}vy measure of $\mu$. Notice that since $X+Y\stackrel{d}{=}Z$ then $\{b_{n}X_{n}\}+\{b_{n}Y_{n}\}\stackrel{d}{=}\{b_{n}Z_{n}\}$, which implies that $\{b_{n}Y_{n}\}$ with c.t.~$(z_{0}^{(Y)},\mathcal{K}^{(Y)},M^{(Y)})$ and $\{b_{n}Z_{n}\}$ with c.t.~$(z_{0}^{(Z)},\mathcal{K}^{(Z)},M^{(Z)})$ are such that $z_{0}=z_{0}^{(Z)}-z_{0}^{(Y)}$, $\mathcal{K}=\mathcal{K}^{(Z)}-\mathcal{K}^{(Y)}$ and $M$ is uniquely generated by $M^{(Z)}$ and $M^{(Y)}$. However, notice that we can consider w.l.o.g. that $\{b_{n}Z_{n}\}\sim(0,\mathcal{K},M^{(Z)})$ and $\{b_{n}Y_{n}\}\sim(-z_{0},0,M^{(Y)})$. This is because of the following argument. Let $Q$ and $R$ be two $l_{2}$ ID process with c.t.~$(0,\mathcal{K},M^{(Z)})$ and $(-z_{0},0,M^{(Y)})$, thus, $\{b_{n}X_{n}\}-R\stackrel{d}{=}Q$. Then, we have that $X-\{b_{n}^{-1}R_{n}\}\stackrel{d}{=}\{b_{n}^{-1}Q_{n}\}$, and that $\{b_{n}^{-1}R_{n}\}$ and $\{b_{n}^{-1}Q_{n}\}$ are ID (where the latter can be seen from their finite dimensional distributions). Therefore, throughout this section for any discrete parameters QID process we consider such generating ID processes.
\\
Let us continue with the preliminaries. For every $y\in l_{2}$, we have that
\begin{equation*}
\mathcal{K}(y)=\sum_{j=1}^{\infty}\beta_{j}\langle e_{j},y \rangle e_{j},
\end{equation*}
where $\beta_{j}\geq 0$, $\sum_{j=1}^{\infty}\beta_{j}<\infty$ and $\{e_{j}\}$ is an orthonormal set in $l_{2}$. Moreover, let
\begin{equation*}
\nu_{0}=\begin{cases}
\|z_{0}\|\delta_{(z_{0}/\|z_{0}\|)} & \textnormal{if $z_{0}\neq 0$},\\ 0 &\textnormal{if $z_{0}= 0$},
\end{cases}\quad\quad\quad\textnormal{and}\quad\quad\quad\nu_{1}=\sum_{j=1}^{\infty}\beta_{j}\delta_{(e_{j})}
\end{equation*}
to be two finite measures on $\mathcal{B}(\partial U)$. Let us recall and reformulate Lemma 4.1 and Proposition 4.2 in \cite{RajRos}
\begin{pro}
	[Part of Lemma 4.1 and of Proposition 4.2 in \cite{RajRos}] Let $M$ be a L\'{e}vy measure on $l_{2}$, let $\partial U$ be the boundary of the unit ball in $l_{2}$ and let $\Psi$ be the map $\Psi:\partial U\times \mathbb{R}^{+}\rightarrow l_{2}\setminus\{0\}$ s.t.~$\Psi(u,x)=xu$. Then $F$ is a unique measure on $\mathcal{B}(\partial U\times\mathbb{R}^{+})$ satisfying
	\begin{equation*}
	M=F\circ\psi^{-1}.
	\end{equation*}
	Moreover, let $F_{A}(\cdot)=F(A\times\cdot)$. Then $F_{A}(\cdot)$ can be naturally extended to $\mathbb{R}$ and is a L\'{e}vy measure on $\mathbb{R}$.
\end{pro}
\begin{proof}
	See Lemma 4.1, Proposition 4.2 and the preliminaries of Section VI in \cite{RajRos}.
\end{proof}
From the above we let $F_{A}^{(Z)}$ and $F_{A}^{(Y)}$ be the unique L\'{e}vy measures associated with $M^{(Z)}$ and $M^{(Y)}$, for every $A\in\mathcal{B}(\partial U)$. Let $F_{A}$ be the unique quasi-L\'{e}vy type measure generated by $F_{A}^{(Z)}$ and $F_{A}^{(Y)}$. Let $\Lambda_{Z}$ and $\Lambda_{Y}$ be the associate ID random measure of $Z$ and $Y$, namely $\Lambda_{Z}\sim(0,\nu_{1},F_{\cdot}^{(Z)})$ and $\Lambda_{Y}\sim(-\nu_{0},0,F_{\cdot}^{(Y)})$. Similarly we have the following definition.
\begin{defn}
	Let $X$, $\nu_{0}$, $\nu_{1}$ and $F_{\cdot}$ be as above, then the QID random measure on $\mathcal{B}(\partial U)$ with parameters $(\nu_{0},\nu_{1},F_{\cdot})$ will be called \textnormal{the associated QID random measure of $X$}, if it exists. Let us denote it by $\Lambda$.
\end{defn}
Observe that the control measure $\lambda$ of $\Lambda$ is given by $\lambda(A)=\nu_{0}(A)+\nu_{1}(A)+\int_{\mathbb{R}}(1\wedge x^{2})F_{A}^{(Z)}+\int_{\mathbb{R}}(1\wedge x^{2})F_{A}^{(Y)}$ for every $A\in\mathcal{B}(\partial U)$. Moreover, denote by $\pi_{j}$ the $j$-th coordinate projection in $l_{2}$. We can now finally present the main result of this section.
\begin{thm}\label{theorem-representation}
	Let $X=\{X_{n}\}$ be a discrete parameters QID process and consider its two generating discrete parameters ID processes $Y$ and $Z$ as above. Then, assuming that $\left\{\int_{\partial U}f_{n}d\Lambda \right\}_{n\in\mathbb{N}}$ is independent of $X$, we have
	\begin{equation*}
	\{X_{n}\}_{n\in\mathbb{N}}-\left\{\int_{\partial U}f_{n}d\Lambda \right\}_{n\in\mathbb{N}}\stackrel{d}{=}\left\{\int_{\partial U}f_{n}d\Lambda_{Z} \right\}_{n\in\mathbb{N}}.
	\end{equation*}
	Moreover, if the associated QID random measure of $X$ exists then $f_{n}$'s are $\Lambda$-integrable and
	\begin{equation}\label{representation}
	\{X_{n}\}_{n\in\mathbb{N}}\stackrel{d}{=}\left\{\int_{\partial U}f_{n}d\Lambda \right\}_{n\in\mathbb{N}}.
	\end{equation}	
\end{thm}
\begin{proof}
	Observe that thanks to Theorem 4.9 in \cite{RajRos} for any fixed $k\in\mathbb{N}$ and $a_{1},...,a_{k}\in\mathbb{R}$ we have that
	\begin{equation*}
	\hat{\mathcal{L}}\left(\sum_{j=1}^{k}a_{j}b_{n}X_{n}\right)=\frac{\hat{\mathcal{L}}\left(\sum_{j=1}^{k}a_{j}b_{n}Z_{n}\right)}{\hat{\mathcal{L}}\left(\sum_{j=1}^{k}a_{j}b_{n}Y_{n}\right)}=\frac{\hat{\mathcal{L}}\left(\sum_{j=1}^{k}a_{j}\int_{\partial U}f_{j}d\Lambda_{Z}\right)}{\hat{\mathcal{L}}\left(\sum_{j=1}^{k}a_{j}\int_{\partial U}f_{j}d\Lambda_{Y}\right)}
	\end{equation*}
	Thus, we obtain first statement. For the second we have that $f_{n}$'s are integrable thanks to \ref{theorem2}, where notice that $\Lambda_{G}$ and $\Lambda_{M}$ there are $\Lambda_{Z}$ and $\Lambda_{Y}$ here. Moreover, thanks to Proposition \ref{proposition-lambda-lambdag-lambdam} we have that for any fixed $k\in\mathbb{N}$ and $a_{1},...,a_{k}\in\mathbb{R}$ we have that
	\begin{equation*}
	\frac{\hat{\mathcal{L}}\left(\sum_{j=1}^{k}a_{j}\int_{\partial U}f_{j}d\Lambda_{Z}\right)}{\hat{\mathcal{L}}\left(\sum_{j=1}^{k}a_{j}\int_{\partial U}f_{j}d\Lambda_{Y}\right)}=\hat{\mathcal{L}}\left(\sum_{j=1}^{k}a_{j}\int_{\partial U}f_{j}d\Lambda\right),
	\end{equation*}
	hence, obtaining the result.
\end{proof}
\begin{open}\label{open-representation}
	Is it possible to show that for every discrete parameters QID process there exists a QID r.m. s.t.~we have the representation $(\ref{representation})$? In other words for every $X$ there exists a $\Lambda$ s.t.~$(\ref{representation})$ holds?
\end{open}
\subsection{Further results and examples of QID processes}
We move now to the presentation of additional results on QID processes and of some examples. Recall that a process $X$ is QID if $X+X^{(2)}\stackrel{d}{=}X^{(1)}$ where $X^{(2)}$ and $X^{(1)}$ are ID processes with $X^{2}$ independent of $X$. From this we have many examples. However, let us start with some results.
\\ The first idea that might come to mind is: If $X^{(2)}$ and $X^{(1)}$ are L\'{e}vy processes is $X$ a QID process without being an ID process? The answer is negative, as shown in the following result.
\begin{pro}
	Let $X^{(1)}$ and $X^{(2)}$ be two L\'{e}vy processes with generating triplets $(a^{(1)},\sigma^{(1)},\nu^{(1)})$ and $(a^{(2)},\sigma^{(2)},\nu^{(2)})$. Let $X$ be a stochastic process independent of $X^{(2)}$. The equality $X_{t}+X^{(2)}_{t}\stackrel{d}{=}X^{(1)}_{t}$ for every $t\geq 0$ holds only if $\sigma^{(1)}\geq \sigma^{(2)}$ and $\nu^{(1)}\geq \nu^{(2)}$. In this case $X_{t}$ is an ID r.v.~for every $t\geq0$.
\end{pro}
\begin{proof}
If $\sigma^{(1)}<\sigma^{(2)}$ then $X_{1}$ will not be a random variable. If $\nu^{(1)}<\nu^{(2)}$ then $X_{t}$ is a QID r.v. for every $t\geq 0$ without being ID. But then, for every $n$ we would have that $X_{\frac{1}{n}}$ is a QID r.v. hence for every $n\in\mathbb{N}$ we would have that $\hat{\mathcal{L}}(X_{1})=(\hat{\mathcal{L}}(X_{\frac{1}{n}}))^{n}$, which implies that $X_{1}$ thus a contradiction. However, if $\sigma^{(1)}\geq \sigma^{(2)}$ and $\nu^{(1)}\geq \nu^{(2)}$ then $X_{t}$ is a ID r.v. for every $t\geq 0$, hence, $X$ might potentially be a L\'{e}vy process.
\end{proof}
However, we have the following positive general result. We use the representation function $c(\cdot)$ (see \cite{LPS}) which is a function $c : \mathbb{R}\rightarrow\mathbb{R}$ which is bounded, Borel measurable and satisfies $\lim\limits_{x\rightarrow0}(c(x)-x)/x^{2}=0$.
\begin{pro}\label{final-general}
Let $X_{0}$ be any QID random variable and let $(a,\sigma,\nu)$ be its characteristic triplet. Then
\begin{equation*}
\hat{\mathcal{L}}(X_{t}):=\exp\left(i\theta h(t)+i\theta a_{t}f(t)-\frac{1}{2}\sigma_{t}^{2}f(t)^{2}\theta^{2}+\int_{\mathbb{R}}e^{i\theta x}-1-i\theta c(x)\nu_{t}(f(t)^{-1}dx) \right)
\end{equation*}
is the characteristic function of a QID random variable for every $t\geq 0$, where $h(\cdot)$, $a_{\cdot}$, $f(\cdot)$ and $\sigma_{\cdot}$ are real valued functions with $f(t)\neq0$ and $\sigma^{2}_{t}\geq \sigma^{2}$ for all $t\geq0$, $\nu_{t}(\cdot)$ is a set function such that $\nu_{t}(A)\geq \nu(A)$ for every $t\geq0$ and $A\in\mathcal{B}_{0}(\mathbb{R})$, and $c(\cdot)$ is any representation function.
\\Moreover, let $k\in\mathbb{N}$ and let $X_{0}^{(1)},...,X_{0}^{(k)}$ be QID random variables. Let  $h^{(j)}(\cdot)$, $a^{(j)}_{\cdot}$, $f^{(j)}(\cdot)$, $\sigma^{(j)}_{\cdot}$ and $\nu_{t}^{(j)}(\cdot)$ satisfy the above respective properties and similarly define $\hat{\mathcal{L}}(X^{(j)}_{t})$. Then
\begin{equation*}
\hat{\mathcal{L}}(Y):=\hat{\mathcal{L}}(X^{(1)}_{t_{1}})\cdots \hat{\mathcal{L}}(X^{(k)}_{t_{k}})
\end{equation*}
is the characteristic function of a QID random variable for every $t_{1},...,t_{k}\geq 0$.
\end{pro}
\begin{proof}
Let for the moment $f(t)=1$ and $h(t)=0$ for every $t\geq 0$. Since for any QID distribution $\mu$ with c.t.~$(\gamma,b,\eta)$ if there exists $b'\geq b$ and $\eta'(A)\geq \eta(A)$ for every $A\in\mathcal{B}_{0}(B)$ then the c.t.~$(\gamma'-\gamma,b'-b,\eta'-\eta)$, where $\gamma'\in\mathbb{R}$, is the c.t.~of a ID distribution, call it $\tilde{\mu}$, (see Remark 2.6 point (a) in \cite{LPS}) hence since the class of QID distributions is closed under convolution then we have that $\mu'=\mu*\tilde{\mu}$ is QID with c.t.~$(\gamma',b',\eta')$ (see Remark 2.6 point (c) in \cite{LPS}). Thus, we have that
\begin{equation*}
\exp\left(a_{t}\theta-\frac{1}{2}\sigma_{t}^{2}\theta^{2}+\int_{\mathbb{R}}e^{i\theta x}-1-i\theta c(x)\nu_{t}(dx) \right)
\end{equation*}
is the c.f.~of a QID distribution. Now, since the class of QID distributions is closed under shifts and dilation (see Remark 2.6 point (b) in \cite{LPS}) then
	\begin{equation*}
	\exp\left(i\theta\left(\tilde{h}(t)+\int_{\mathbb{R}}c(mx)-mc(x)\nu_{t}(dx)+a_{t}f(t)\right)-\frac{1}{2}\sigma_{t}^{2}f(t)^{2}\theta^{2}+\int_{\mathbb{R}}e^{i\theta x}-1-i\theta c(x)\nu_{t}(f(t)^{-1}dx) \right)
	\end{equation*}
	is the c.f.~of QID distribution, where $\tilde{h}$ is any real valued function. Hence, by setting $\tilde{h}$ such that $h(t)=\tilde{h}(t)+\int_{\mathbb{R}}c(mx)-mc(x)\nu_{t}(dx)$ we have the first statement.
	
	The second statement follows by the fact that the class of QID distributions is closed under convolution.

\end{proof}
\begin{co}Under the same conditions of Proposition \ref{final-general} we have that for every $t\geq0$ there exist two QID random variables $Y_{t}$ and $Z_{t}$ such that
\begin{equation*}
\hat{\mathcal{L}}(Y_{t})=\exp\left(i\theta a_{t}-\frac{1}{2}\sigma_{t}^{2}\theta^{2}+\int_{\mathbb{R}}e^{i\theta x}-1-i\theta c(x)\nu_{t}(dx) \right)
\end{equation*}
and
\begin{equation*}
\hat{\mathcal{L}}(Z_{t})=\exp\left(i\theta h(t)+i\theta af(t)-\frac{1}{2}\sigma^{2}f(t)^{2}\theta^{2}+\int_{\mathbb{R}}e^{i\theta x}-1-i\theta c(x)\nu(f(t)^{-1}dx) \right).
\end{equation*}
\end{co}
\begin{proof}
	It follows directly from Proposition \ref{final-general}.
\end{proof}
A first example of QID process has been investigated in \cite{Zhang}, although they mostly focus on the level of distributions and in particular on a subclass of QID distributions call discrete pseudo-compound Poisson distributions, which have applications in insurance mathematics. Moreover, from the formulation $X+X^{(2)}\stackrel{d}{=}X^{(1)}$ and from the previous results many examples can be constructed. We do not deal with the existence of particular QID stochastic processes since it is a major topic and we leave it for further research. We only point out that most of the potential results seem to rely on the multidimensional extension of the results presented in \cite{LPS}. Hence, the following question.
\begin{open}
	Is it possible to extend the results in \cite{LPS} and \cite{Berger} to the multidimensional case, namely to distributions on $\mathbb{R}^{d}$?
\end{open}
\section{The atomless condition and the density of QID r.m.}\label{Sec-Atomless}
One of the main properties that independently scattered random measure might satisfy is the atomless condition. Indeed, the work of Pr\'{e}kopa \cite{Prekopa} is centred on this condition and in Theorem 2.2 of \cite{Prekopa} he proves that if an independently scattered random measure satisfies the atomless condition then it is an ID random measure. Recall that every $\sigma$-ring is a $\delta$-ring, but not every $\delta$-ring is a $\sigma$-ring, and that every $\sigma$-algebra is a $\sigma$-ring, but not every $\sigma$-ring is a $\sigma$-algebra.
\begin{defn}[atom, atomless]
	Let $\Lambda$ be a completely additive set function defined on a $\sigma$-ring $\mathcal{S}$. A set $A\in\mathcal{S}$ is called an \textnormal{atom} relative to the set function $\Lambda$ if for every $C\subseteq A$ with $C\in\mathcal{S}$ we have either $\Lambda(C)=0$ or $\Lambda(C)=\Lambda(A)$. Moreover, the completely additive set function $\Lambda$ will be called \textnormal{atomless} if for every atom $A$ we have $\Lambda(A)=0$.
\end{defn}
The atomless condition is for random measures what the continuity in probability is for continuous time stochastic processes. Notice that we have mentioned explicitly continuous time processes because for discrete time ones the condition is meaningless.\\
Then, we have the following result.
\begin{co}
	It does not exist a QID r.m.~on a $\sigma$-ring which is atomless and which is not ID.
\end{co}
\begin{proof}
	The result is straightforward since any atomless independently scattered r.m.~is ID, see Theorem 2.2 in \cite{Prekopa}.
\end{proof}
In Section $\ref{chapter-IDvQID}$, we have shown the connections between QID and ID random measure. From this discussion it appears clear that in case the ID random measures considered are atomless the generated QID r.m.~is indeed an ID r.m.. This lead us to the following question: is it true that ID r.m.~are always atomless? The answer is negative. This is because in case either the drift $\nu_{0}$ or the Gaussian part $\nu_{1}$ or the measure $F_{\cdot}(B)$ for some $B\in\mathcal{B}(\mathbb{R})$ s.t.~$0\notin \overline{B}$ have at least one atom then the corresponding ID measure is not atomless. Therefore, the results shown in Section $\ref{chapter-IDvQID}$ applies to any possible ID and QID r.m.. We refer to Chapter 3.3 in \cite{Kallenberg2} for a lucid exposition of ID random measures. 

We point out that the atomless condition is not relevant for the spectral representation of the discrete parameter QID process presented in Section \ref{Ch-Spectr}. The associated r.m.~might potentially have atoms and this is true also in the ID framework. 

We present now some interesting results. We will restrict to the framework of \cite{Kallenberg2}, which is still very general.
\\

\textit{Notation:} For the rest of this section for ``random measure" we mean a ``$\mathbb{R}_{+}$-valued random measure", and for ``random signed measure" we mean ``$\mathbb{R}$-valued random measure"
\\

Let $(S,d)$ be a separable and complete metric space. Let $\textbf{S}$ the associated metric topology and let $\hat{\mathbf{S}}$ be the ring composed by bounded Borel sets in $S$. Let $\hat{C}_{S}$ be the space of all bounded continuous functions $f:S\rightarrow\mathbb{R}_{+}$ with bounded support. Let $\mathcal{M}_{S}$ be the space of locally finite measures, namely $\mu\in\mathcal{M}_{S}$ if $\mu(B)<\infty$ for every $B\in\hat{\mathbf{S}}$. The space $\mathcal{M}_{S}$ might be endowed with the vague topology, denoted by $\mathbf{B}_{\mathcal{M}_{S}}$, generated by the integration maps $\pi_{f}:\mu\mapsto\int f(x) \mu(dx)$, for all $f\in \hat{C}_{S}$. The vague topology is the coarsest topology making all $\pi_{f}$ continuous. The measurable space $(\mathcal{M}_{s},\mathbf{B}_{\mathcal{M}_{S}})$ is a Polish space. The associated notion of vague convergence denoted by $\mu_{n}\stackrel{v}{\rightarrow}\mu$ is defined by the condition $\int f(x) \mu_{n}(dx)\rightarrow\int f(x) \mu(dx)$ for all $f\in\hat{C}_{S}$.

A random measure $\xi$ on $S$, with underlying probability space $(\Omega,\mathcal{F},\mathbb{P})$, is a function $\Omega\times\textbf{S}\rightarrow[0,\infty]$, such that $\xi(\omega,B)$ is a $\mathcal{F}$-measurable in $\omega\in\Omega$ for fixed $B$ and a locally finite measure in $B\in\textbf{S}$ for fixed $\omega$.\footnote{An equivalent definition is the following: a random measure $\xi$ is a measurable mapping from $(\Omega,\mathcal{F},\mathbb{P})$ to $(\mathcal{M}_{S},\mathcal{B}_{\mathcal{M}_{S}})$, where $\mathcal{B}_{\mathcal{M}_{S}}$ is the topology generated by all projection maps $\pi_{B}:\mu\mapsto\mu(B)$ with $B\in\mathbf{S}$, or, equivalently, by all integration maps $\pi_{f}$ with measurable $f\geq0$. From Lemma 4.1 in \cite{Kallenberg0} or Theorem 4.2 in \cite{Kallenberg2}, we know that $\mathcal{B}_{\mathcal{M}_{S}}$ and $\mathbf{B}_{\mathcal{M}_{S}}$ coincide. Hence it is equivalent to consider a random measure as a measurable mapping from $(\Omega,\mathcal{F},\mathbb{P})$ to $(\mathcal{M}_{S},\textbf{B}_{\mathcal{M}_{S}})$ or to $(\mathcal{M}_{S},\mathcal{B}_{\mathcal{M}_{S}})$.} Then, convergence in distribution of $\xi_{n}$ to $\xi$ means that $\mathbb{E}[g(\xi_{n})]\rightarrow\mathbb{E}[g(\xi)]$ for every bounded continuous function $g$ on $\mathcal{M}_{S}$, or equivalently that $\mathcal{L}(\xi_{n})\stackrel{w}{\rightarrow}\mathcal{L}(\xi)$, where for any bounded measures $\mu_{n}$ and $\mu$, the weak convergence $\mu_{n}\stackrel{w}{\rightarrow}\mu$ means that $\int g(y)\mu_{n}(dy)\rightarrow\int g(y)\mu(dy)$ for all $g$ as above. We write $\xi_{n}\stackrel{vd}{\rightarrow}\xi$ to stress that the convergence of distribution is for random measures considered as random elements in the space $\mathcal{M}_{S}$ with vague topology. In this setting, a fixed atom of a random measure $\xi$ is an element $s\in S$ such that $\mathbb{P}(|\xi(\{s\})|>0)>0$.

We would like to consider QID random measures on $(S,\textbf{S})$ as independently scattered random measures s.t.~$\xi(B)$ is QID for every $B\in\hat{\textbf{S}}$. This is the usual definition of QID (and similarly of Poisson and of ID) random measures on $(S,\textbf{S})$. However, our previous definition (see Definition \ref{defQIDr.m.}) is different. In Definition \ref{defQIDr.m.}, QID random measure are real valued (stochastic processes) and so they cannot take infinite values, which in this framework means that they can only be defined on a ring $\textbf{U}\subseteq \hat{\textbf{S}}$. Thus, the question is how can we combine the two definitions? The answer is the following. By starting with the QID random measures as defined in Definition \ref{defQIDr.m.}, we consider only the ones which are $\mathbb{R}_{+}$-valued and then, thanks to a result by Harris (see \cite{Harris}), we can uniquely extend them to random measures on $(S,\textbf{S})$. Let us report here the mentioned result by Harris.
\begin{thm}
	[see Theorem 2.15 in \cite{Kallenberg2}] Given a
	process $\eta \geq 0$ on a generating ring $\textbf{U}\subset\hat{\textbf{S}}$, there exists a random measure $\xi$
	on $S$ with $\xi(U) = \eta(U)$ a.s.~for all $U\in\textbf{U}$ , iff
	\\\textnormal{(i)} $\eta(A \cup B) = \eta(A) + \eta(B)$ a.s., $A, B \in\textbf{U}$ disjoint,
	\\\textnormal{(ii)} $\eta(A_{n})\stackrel{P}{\rightarrow}0$ as $A_{n}\searrow\emptyset$ along $\textbf{U}$.\\
	In that case, $\xi$ is a.s.~unique.
\end{thm}
Observe that a QID random measure as defined in Definition \ref{defQIDr.m.} satisfies the above conditions, and so it extends uniquely to $S$. Therefore, following our definitions, any QID random measure on $(S,\textbf{S})$ is the unique extension of QID random measure on $\hat{\textbf{S}}$ or on any generating ring $\textbf{U}\subset\hat{\textbf{S}}$.

The practical example one should think when dealing with this framework (and in general with the framework of this work) is the following: a random measure on $\mathcal{B}_{b}(\mathbb{R})$ (\textit{i.e.}~the set of bounded intervals of $\mathbb{R}$), which has almost surely finite values for any $B\in\mathcal{B}_{b}(\mathbb{R})$. Observe that $\mathcal{B}_{b}(\mathbb{R})$ is not an algebra because $\mathbb{R}\notin \mathcal{B}_{b}(\mathbb{R})$ and this is why it is important to work with rings. In this example, using the notation of this work, we have that $S=\mathbb{R}$ and ${\mathcal {S}}= \hat{\textbf{S}}=\mathcal{B}_{b}(\mathbb{R})$. Moreover, notice that the same applies to the case of $\mathcal{B}_{b}(\mathbb{R}_{+})$ instead of $\mathcal{B}_{b}(\mathbb{R})$. From this example, it also appears clear and natural the condition that imposes the existence of an increasing sequence of sets $S_{1},S_{2},\dots \in \mathcal{B}_{b}(\mathbb{R})$ s.t.~$\bigcup _{n\in \mathbb {N} }S_{n}=S$; indeed, take $S_{n}$ to be concentric balls of radii $n$. This condition is at the base of our work, of the work of Rajput and Rosinski \cite{RajRos}, and of Kallenberg's book (see page 15 in \cite{Kallenberg2}).

We report now another fundamental result by Harris, see \cite{Harris}.
\begin{thm}[see Theorem 4.11 in \cite{Kallenberg2}]\label{density-T1}
	Let $\xi,\xi_{1},\xi_{2},...$ be random measures on $S$. Then these conditions are equivalent:
	\\\textnormal{(i)} $\xi_{n}\stackrel{vd}{\rightarrow}\xi$,
	\\\textnormal{(ii)} $\int f(x)\xi_{n}(dx)\stackrel{d}{\rightarrow}\int f(x)\xi(dx)$ for all $f\in\hat{C}_{S}$,
	\\\textnormal{(ii)} $\mathbb{E}[\exp(-\int f(x)\xi_{n}(dx))]\rightarrow\mathbb{E}[\exp(-\int f(x)\xi(dx))]$ for all $f\in\hat{C}_{S}$ with $f\leq1$.	
\end{thm}
We are ready to present the following density result.
\begin{thm}\label{density-T2}QID random measures are dense in the space of independently scattered random measures, considered as random elements in $\mathcal{M}_{S}$ endowed with the vague topology, under the convergence in distribution.
\end{thm}
\begin{proof}
From Theorem 7.1 \cite{Kallenberg0} we know that any independently scattered random measure has the following unique representation
\begin{equation*}
\xi\stackrel{a.s.}{=}\alpha+\sum_{j=1}^{K}\beta_{j}\delta_{s_{j}}
\end{equation*}
with $K\leq\infty$, where $\{s_{j}:j\geq1\}$ is the set of fixed atoms of $\xi$, $\alpha$ is a random measure without fixed atoms with independent increments (hence, $\alpha$ is an atomless ID r.m.), and $\beta_{j}$, $j\geq1$, are $\mathbb{R}_{+}$-valued r.v., which are mutually independent and independent of $\alpha$.

Let $\xi_{n}$ be an independently scattered r.m.~defined as 
\begin{equation*}
\xi_{n}\stackrel{a.s.}{=}\alpha+\sum_{j=1}^{K}\beta_{n,j}\delta_{s_{j}}
\end{equation*}
with $\beta_{n,j}$ QID random variables s.t. $\beta_{n,j}\stackrel{d}{\rightarrow}\beta_{j}$ as $n\rightarrow\infty$, for every $j\geq1$, and where $\alpha$ and $\{s_{j}:j\geq1\}$ are the same as above. First, notice that for fixed $j$ the sequence $\beta_{1,j},\beta_{2,j},...$ exists thanks to Theorem 4.1 from \cite{LPS}. Second, observe that, for each $n\in\mathbb{N}$, $\xi_{n}$ is a QID random measure because $\alpha$ is ID and $\beta_{j}$ are QID and they are independent of each other.

Now, we need to show that $\xi_{n}\stackrel{vd}{\rightarrow}\xi$. From Theorem \ref{density-T1}, it is sufficient to show that $\int f(x)\xi_{n}(dx)\stackrel{d}{\rightarrow}\int f(x)\xi(dx)$ for all $f\in\hat{C}_{S}$. Since $\alpha$ is both an element of $\xi_{n}$ and $\xi$ and it is independent of the $\beta_{n,j}$, $j\geq1$, this reduces to prove that $\sum_{j=1}^{K} f(s_{j})\beta_{n,j}\stackrel{d}{\rightarrow}\sum_{j=1}^{K} f(s_{j})\beta_{j}$ for all $f\in\hat{C}_{S}$. If $K$ is finite for any bounded set then the result follows easily from independence of the $\beta_{n,j}$, $j=1,...,K$, from the fact that $\beta_{n,j}\stackrel{d}{\rightarrow}\beta_{j}$ as $n\rightarrow\infty$, for every $j=1,...,K$ and from the continuous mapping theorem. Thus, assume that $K$ is infinite, formally that there are countably many fixed point for at least one bounded set. Consider any $f\in\hat{C}_{S}$. Then, we need to show that $\sum_{j=1}^{\infty} f(s_{j})\beta_{n,j}\stackrel{d}{\rightarrow}\sum_{j=1}^{\infty} f(s_{j})\beta_{j}$. This convergence also holds true for the same arguments as the ones mentioned for the case of $K$ finite, however we need to be careful because of the infinities. First, observe that, for each $n\in\mathbb{N}$, $\sum_{j=1}^{\infty} f(s_{j})\beta_{n,j}<\infty$ a.s.~and that $\sum_{j=1}^{\infty} f(s_{j})\beta_{j}<\infty$ a.s.. This is because $f\in\hat{C}_{S}$, hence, by denoting $B$ the support of $f$, we have that almost surely $\xi_{n}(B)<\infty$, $n\in\mathbb{N}$, and $\xi(B)<\infty$ since $B\in\hat{\textbf{S}}$, and that $f$ is bounded. Then, looking at the characteristic functions, using the continuity of $e^{z}$, for $z\in\mathbb{C}$, and the dominated convergence theorem, we have that for every $\theta\in\mathbb{R}$
\begin{equation*}
1\geq \mathbb{E}\bigg[\exp\bigg(i\theta\sum_{j=1}^{\infty} f(s_{j})\beta_{n,j}\bigg)\bigg]=\mathbb{E}\bigg[\lim\limits_{N\rightarrow\infty}\exp\bigg(i\theta\sum_{j=1}^{N} f(s_{j})\beta_{n,j}\bigg)\bigg]=\lim\limits_{N\rightarrow\infty}\mathbb{E}\bigg[\exp\bigg(i\theta\sum_{j=1}^{N} f(s_{j})\beta_{n,j}\bigg)\bigg]
\end{equation*}
\begin{equation*}
=\lim\limits_{N\rightarrow\infty}\prod_{j=1}^{N}\mathbb{E}\bigg[\exp\bigg(i\theta f(s_{j})\beta_{n,j}\bigg)\bigg]=\prod_{j=1}^{\infty}\mathbb{E}\bigg[\exp\bigg(i\theta f(s_{j})\beta_{n,j}\bigg)\bigg]
\end{equation*}
Since $\beta_{n,j}\stackrel{d}{\rightarrow}\beta_{j}$ then by continuous mapping theorem we have $f(s_{j})\beta_{n,j}\stackrel{d}{\rightarrow}f(s_{j})\beta_{j}$ and so
\begin{equation*}
\prod_{j=1}^{\infty}\mathbb{E}\bigg[\exp\bigg(i\theta f(s_{j})\beta_{n,j}\bigg)\bigg]\rightarrow \prod_{j=1}^{\infty}\mathbb{E}\bigg[\exp\bigg(i\theta f(s_{j})\beta_{j}\bigg)\bigg]=\mathbb{E}\bigg[\exp\bigg(i\theta\sum_{j=1}^{\infty} f(s_{j})\beta_{j}\bigg)\bigg],\quad \textnormal{as $n\rightarrow\infty$.}
\end{equation*}
Thus, we have that $\sum_{j=1}^{\infty} f(s_{j})\beta_{n,j}\stackrel{d}{\rightarrow}\sum_{j=1}^{\infty} f(s_{j})\beta_{j}$ and since $f$ was any function in $\hat{C}_{S}$ we obtain the stated result.
\end{proof}
It is possible to consider also the set of bounded measures, denoted by $\hat{\mathcal{M}}_{S}$, which can be endowed with the vague topology, as for $\mathcal{M}_{S}$, but also with the weak topology. The weak topology on $\hat{\mathcal{M}}_{S}$ is the topology generated by the integration maps $\pi_{f}$ for all bounded continuous functions. Then, for random measures $\xi,\xi_{1},\xi_{2},...$ considered as random elements in $\hat{\mathcal{M}}_{S}$, endowed with the weak topology, we will denote by $\xi_{n}\stackrel{wd}{\rightarrow}\xi$ the convergence in distribution. Observe that in this setting a QID random measures as defined in Definition are QID random measures on $(S,\textbf{S})$ (hence we do not need to extend them) because for every $B\in\textbf{S}$ they are all a.s.~bounded.

We will use the following result of Kallenberg to prove our next result.
\begin{thm}[see Theorem 4.19 in \cite{Kallenberg2}]\label{density-T"}
	Let $\xi,\xi_{1},\xi_{2},...$ be a.s.~bounded random measures on $S$. Then these conditions are equivalent
	\\\textnormal{(i)} $\xi_{n}\stackrel{wd}{\rightarrow}\xi$,
	\\\textnormal{(ii)} $\xi_{n}\stackrel{vd}{\rightarrow}\xi$, and $\xi_{n}(S)\stackrel{d}{\rightarrow}\xi(S)$.
\end{thm}
We are now ready to present our next result, which is similar to Theorem \ref{density-T2}, but applies to $\hat{\mathcal{M}}_{S}$ and involves both the vague and the weak topology.
\begin{thm}QID random measures are dense in the space of independently scattered random measures, considered as random elements in $\hat{\mathcal{M}}_{S}$ endowed with the vague topology or with the weak topology, under the convergence in distribution.
\end{thm}
\begin{proof}
Consider first the case of $\hat{\mathcal{M}}_{S}$ endowed with the vague topology. Then, by the same arguments as the ones used in the proof of Theorem \ref{density-T2} we obtain the result.

For the weak topology case, by the same arguments as the ones used in the proof of Theorem \ref{density-T2} we have that $\xi_{n}\stackrel{vd}{\rightarrow}\xi$. Hence, according to Theorem \ref{density-T"} it remains to prove that $\xi_{n}(S)\stackrel{d}{\rightarrow}\xi(S)$, namely that $\alpha(S)+\sum_{j=1}^{\infty}\beta_{n,j}\stackrel{d}{\rightarrow}\alpha(S)+\sum_{j=1}^{\infty}\beta_{j}$. However, this has been proved in the proof of Theorem \ref{density-T2} -- indeed, consider $f\equiv1$ and notice that $\xi(S)<\infty$ a.s.~since $\xi$ is almost surely bounded. Thus, the proof is complete.
\end{proof}
So far in this section we have only discussed measures and random measures which take only non-negative values. The main reason is because, as far as we know, there are no results in the literature on convergence of \textit{real valued} random measures, namely of random signed measures. This is a pity since the random measures considered in this work are real valued (\textit{e.g.}~see Definition \ref{defQIDr.m.} of QID random measures). We believe in fact that our density results extend to the general (signed) case, at least for bounded random measures. Proving this requires the extension of several results of the first four chapters of \cite{Kallenberg2} to the signed case. Although we leave this as a topic of further research, we present next a first important result in this direction.

Let $\mathcal{M}^{-}_{S}$ be the space of locally finite signed measures. Recall that for a signed measure $|\mu(B)|<\infty$ iff $|\mu|(B)<\infty$. Hence, $\mu\in\mathcal{M}^{-}_{S}$ if $|\mu(B)|<\infty$ for every $B\in\hat{\textbf{S}}$. The next result is an extension of Theorem 7.1 in \cite{Kallenberg0} and Lemma 2.1 in \cite{Hellmund}.
\begin{thm} Every signed measure on $S$ has an atomic decomposition:
	\begin{equation*}
\mu=\gamma+\sum_{j=1}^{N}\lambda_{j}\delta_{t_{j}}
	\end{equation*}
	where $N\in\mathbb{Z}_{+}\cup\{\infty\}$, $t_{1},t_{2},...\in S$,  $\gamma$ is an atomless measure, and $\lambda_{j}$, $j\geq1$, are non-negative constants.
	
	Moreover, any random signed measure $\xi$ on $S$ with independent increments has the following almost sure unique representation:
	\begin{equation*}
	\xi=\alpha+\sum_{j=1}^{K}\beta_{j}\delta_{s_{j}}
	\end{equation*}
	for some fixed $K\in\mathbb{Z}_{+}\cup\{\infty\}$ and $s_{1},s_{2},...\in S$, some infinitely divisible independently scattered random signed measure $\alpha$ without fixed atoms, and some $\mathbb{R}$-valued r.v.~$\beta_{j}$, $j\geq1$, which are mutually independent and independent of $\alpha$.
\end{thm}
\begin{proof}
	The first statement follows from the fact that every measure $\mu^{+}$ on $S$ can be written as $\mu^{+}=\gamma^{+}+\sum_{j=1}^{N^{+}}\lambda^{+}_{j}\delta_{t^{+}_{j}}$. Then by Jordan decomposition we know that there exist two measures $\mu^{+}$ and $\mu^{-}$ such that $\mu=\mu^{+}-\mu^{-}$ and $\mu^{+}\perp\mu^{-}$, hence we obtain the stated representation.
	
	Regarding the second statement, we have that $\xi$ has at most countably many fixed atoms. To see this, observe that it is sufficient to prove that for fixed bounded set and $\epsilon>0$ there cannot be infinitely many atoms $\{s_{n}|n\in\mathbb{N}\}$ s.t.~$\mathbb{P}(|\xi(\{s_{n}\})|\geq\epsilon)\geq\epsilon$. Assume that this is true then $\mathbb{P}(\limsup\limits_{n\rightarrow\infty}|\xi|(s_{n})\geq\epsilon)\geq\epsilon$, hence $\sum_{n=1}^{\infty}|\xi|(s_{n})$ cannot converge, thus we obtain a contradiction.
	
	Once we subtract the fixed atoms we are left with an atomless independently scattered random measure $\alpha$ and, by Theorem 2.2 in \cite{Prekopa}, we conclude that it is ID. Finally, the independence follows from the independently scattered property of $\xi$ and the atomless property of $\alpha$.
\end{proof}
We end this section with the aforementioned conjecture. Let $\hat{\mathcal{M}}^{-}_{S}$ be the space of bounded signed measures. Observe that this space is quite interesting. First, it is a vector space since it is closed under summation and multiplication by a constant. Moreover, the total variation defines a norm, which makes $\hat{\mathcal{M}}^{-}_{S}$ a Banach space. Indeed, this space has been intensively studied under various properties of the signed measures and of the space $S$ (\textit{e.g.}~see the Riesz-Markov-Kakutani representation theorem).
\begin{conj}
Real valued QID random measures are dense in the space of independently scattered real valued random measures, considered as random elements in $\mathcal{M}^{-}_{S}$ endowed with the vague topology, under the convergence in distribution. 

Similarly, bounded real valued QID random measures are dense in the space of bounded independently scattered real valued random measures, considered as random elements in $\hat{\mathcal{M}}^{-}_{S}$ endowed with the vague topology or with the weak topology, under the convergence in distribution.
\end{conj}
\section*{Conclusion}
In this work we have extended the theory of QID distributions to random measures and stochastic processes. In particular, we studied the existence and uniqueness of QID r.m.~and how they are related to ID r.m.. We have seen how to define stochastic integrals w.r.t.~QID r.m.~and show for which conditions on a functions $f$ we have that $\int fd\Lambda$ is defined. Moreover, we have investigated various representations of a QID r.m.~and of a QID stochastic process. We ended with density results for QID random measures.

There are several and pivotal questions that are left open. First, we know that QID distributions are dense in the space of all probability distributions under the weak convergence and we showed that similar results hold for QID random measures. Can we extend this processes? In other words, can we prove that for any stochastic process there is a sequence of QID processes that converge to it in distribution? Or at least in finite dimensional distribution, namely without tightness?
\\ Second, what properties does a QID process satisfy? Is there a similar L\'{e}vy-It\^{o} decomposition?
\\ Third, it has been seen that QID distributions are related to the Riemann zeta function. What insights on this function can we extract from QID random measures and stochastic processes?
\section*{Acknowledgement}
The author would like to thank Fabio Bernasconi, Georgios Chalivopulos, Mikko Pakkanen and Almut Veraart for useful discussions, and the CDT in MPE and the Grantham Institute for financial support. 

\end{document}